%% file: traceless1.tex
\documentclass[11pt]{amsart}
\usepackage[utf8]{inputenc}
\usepackage{cite}

\usepackage{epsfig}
\usepackage{amssymb,amsmath,amsthm,amscd}
\usepackage{latexsym}
\usepackage[margin=1in]{geometry}
\usepackage[all,knot]{xy}
\usepackage{color}
\usepackage{mathrsfs}
\usepackage{eucal}
\usepackage{subcaption}
\usepackage{etoolbox}

\usepackage{graphicx}
\usepackage{epstopdf}
\usepackage[colorlinks=true,allcolors=blue]{hyperref}

\DeclareMathAlphabet{\mathpzc}{OT1}{pzc}{m}{it}

\numberwithin{equation}{section}

\newtheorem{thm}{Theorem}[section]
\newtheorem{prop}[thm]{Proposition}
\newtheorem{cor}[thm]{Corollary}
\newtheorem{lem}[thm]{Lemma}

\newtheorem{conj}{Conjecture}[section]

\theoremstyle{definition}
\newtheorem{defn}[thm]{Definition}
\newtheorem{exmp}[thm]{Example}

\newtheorem{rem}[thm]{Remark}

\newcommand{\Z}{\mathbb{Z}}
\newcommand{\R}{\mathbb{R}}
\newcommand{\C}{\mathbb{C}}

\newcommand{\F}{\mathbb{F}}
\newcommand{\J}{\mathcal{J}}
\renewcommand{\O}{\mathcal{O}}

\newcommand{\M}{\mathscr{M}}
\newcommand{\scrR}{\mathscr{R}}
\newcommand{\G}{\mathscr{G}}

\newcommand{\A}{\mathscr{A}}

\newcommand{\D}{\mathbb{D}}
\newcommand{\bbD}{\mathbb{D}}
\newcommand{\scrD}{\mathscr{D}}

\newcommand{\dbar}{\overline{\partial}}
\newcommand{\PD}{\mathrm{PD}}
\newcommand{\rk}{\mathrm{rk}}
\newcommand{\Id}{\operatorname{id}}
\newcommand{\Hom}{\mathrm{Hom}}

\newcommand{\su}{\mathfrak{su}}
\newcommand{\Tr}{\operatorname{Tr}}

\newcommand{\Tor}{\operatorname{Tor}}

\renewcommand{\L}{\mathscr{L}}

\newcommand{\im}{\operatorname{im}}

\newcommand{\gr}{\operatorname{gr}}
\newcommand{\id}{\operatorname{id}}
\newcommand{\calN}{\mathcal{N}}
\renewcommand{\Re}{\operatorname{Re}}

\newcommand{\T}{\mathbb{T}}
\newcommand{\CF}{\mathrm{CF}}

\newcommand{\HF}{\mathrm{HF}}
\newcommand{\HFhat}{\widehat{\mathbf{HF}}}

\newcommand{\g}{\mathfrak{g}}

\newcommand{\SI}{\operatorname{SI}}
\newcommand{\CSI}{\operatorname{CSI}}

\newcommand{\Symp}{\mathbf{Symp}}

\newcommand{\bfalpha}{\boldsymbol{\alpha}}
\newcommand{\bfbeta}{\boldsymbol{\beta}}
\newcommand{\bfgamma}{\boldsymbol{\gamma}}
\newcommand{\bfdelta}{\boldsymbol{\delta}}
\newcommand{\bfx}{\mathbf{x}}
\newcommand{\Whit}{\operatorname{Whit}}
\newcommand{\Map}{\operatorname{Map}}
\newcommand{\calH}{\mathcal{H}}
\newcommand{\bfi}{\mathbf{i}}
\newcommand{\bfj}{\mathbf{j}}
\newcommand{\bfk}{\mathbf{k}}

\newcommand{\bbK}{\mathbb{K}}
\newcommand{\Bord}{\mathbf{Bord}}
\newcommand{\Crit}{\operatorname{Crit}}

\newcommand{\bfA}{\mathbf{A}}
\newcommand{\Hol}{\operatorname{Hol}}
\newcommand{\calD}{\mathcal{D}}

\newcommand{\Kh}{\operatorname{Kh}}

\newcommand{\Cone}{\operatorname{Cone}}

\newcommand{\calM}{\mathcal{M}}
\newcommand{\calE}{\mathcal{E}}
\newcommand{\calI}{\mathcal{I}}

\newcommand{\CS}{\operatorname{CS}}
\newcommand{\IC}{\mathrm{IC}}
\newcommand{\I}{\mathrm{I}}
\newcommand{\scrC}{\mathscr{C}}
\newcommand{\scrF}{\mathscr{F}}

\newcommand{\ad}{\operatorname{ad}}

\newcommand{\SU}{\mathrm{SU}}
\newcommand{\SO}{\mathrm{SO}}

\newcommand{\MCG}{\mathrm{MCG}}
\newcommand{\crit}{\mathrm{crit}}

\newcommand\todo[1]{\textbf{\textcolor{red}{#1}}}
\newcommand\hide[1]{}

\title{A Symplectic Instanton Homology via Traceless Character Varieties}
\author{Henry T. Horton}

\begin{document}

\begin{abstract}
Since its inception, Floer homology has been an important tool in low-dimensional topology. Floer theoretic invariants of $3$-manifolds tend to be either gauge theoretic or symplecto-geometric in nature, and there is a general philosophy that each gauge theoretic Floer homology should have a corresponding symplectic Floer homology and vice-versa. In this article, we construct a Lagrangian Floer invariant for any closed, oriented $3$-manifold $Y$ (called the {\bf symplectic instanton homology} of $Y$ and denoted $\SI(Y)$) which is conjecturally equivalent to a Floer homology defined using a certain variant of Yang-Mills gauge theory. The crucial ingredient for defining $\SI(Y)$ is the use of traceless character varieties in the symplectic setting, which allow us to avoid the debilitating technical hurdles present when one attempts to define a symplectic version of instanton Floer homologies. Furthermore, by studying the effect of Dehn surgeries on traceless character varieties, we establish a surgery exact triangle using work of Seidel that relates the geometry of Lefschetz fibrations with exact triangles in Lagrangian Floer theory.
\end{abstract}

\maketitle

\tableofcontents

\addtocontents{toc}{\protect\setcounter{tocdepth}{1}}

\include{introduction1}
\include{flat-moduli}
\include{heegaard-lagrangians}
\include{SI-heegaard-splitting}

\include{naturality}
\include{SI-cerf}
\include{SI-nontrivial}

\include{surgery-triangle}

\include{computations}

\appendix

\include{heegaard-diagrams}
\include{quilted-floer}
\include{dehn-twists}

\bibliographystyle{plain}
\bibliography{traceless}

\end{document}

%% file: introduction1.tex
\section{Introduction}


As originally defined by Andreas Floer \cite{instanton-invariant}, instanton homology is an invariant of integer homology $3$-spheres constructed using an infinite-dimensional analogue of the Morse-Smale-Witten chain complex. Roughly speaking, for a integer homology $3$-sphere $Y$, Floer's instanton chain complex is generated by gauge equivalence classes of nontrivial (perturbed) flat connections on the trivial $\SU(2)$-bundle over $Y$, and the differential counts anti-self-dual $\SU(2)$-connections (instantons) on $Y \times \R$ which have the appropriate asymptotics.

Using similar ideas, Floer \cite{Floer1} also defined a homological invariant for pairs of Lagrangian submanifolds $L_0$, $L_1$ in some fixed symplectic manifold $(M, \omega)$. The chain complex for this Lagrangian Floer homology is generated by intersection points between the Lagrangians, and the differential counts pseudoholomorphic strips  $u: \R \times [0,1] \longrightarrow M$ with $u(\R, 0) \subset L_0$, $u(\R, 1) \subset L_1$, and the appropriate asymptotic behavior.

Atiyah \cite{atiyah-floer} had the remarkable insight that Floer's instanton homology should have an interpretation in terms of Lagrangian Floer theory. Namely, if one chooses a genus $g$ Heegaard splitting for a $3$-manifold $Y$, one can consider the $\SU(2)$-character varieties $L_\alpha$, $L_\beta$ of the two pieces as lying in the $\SU(2)$-character variety $M(\Sigma_g)$ of the Heegaard surface, by restriction of representations. $M(\Sigma_g)$ is a stratified symplectic space, and $L_\alpha$, $L_\beta$ are Lagrangian, so one could hope to define the Lagrangian Floer homology of $(L_\alpha, L_\beta)$. The {\bf Atiyah-Floer conjecture} says that, assuming this Lagrangian Floer homology can be defined, it is equal to the instanton Floer homology of $Y$ when $Y$ is a integer homology $3$-sphere.

Unfortunately, there are difficult technical issues that make the Atiyah-Floer conjecture challenging to prove. The principal issue is that $M(\Sigma_g)$ is not a smooth manifold, preventing one from defining Lagrangian Floer groups in a straightforward way. Salamon and Wehrheim \cite{salamon-af,salamon-wehrheim,wehrheim-af} have initiated a program to understand and prove the Atiyah-Floer conjecture, but the problem remains open.

The purpose of the present work is to move towards a better understanding of the interplay between instanton homology and Lagrangian Floer theory by considering a suitable modification of the relevant $\SU(2)$-character varieties that prevents the existence of singularities. Although the resulting Lagrangian Floer homology is no longer equal to Floer's instanton homology, even for $S^3$, a suitable modification of instanton homology due to Kronheimer and Mrowka \cite{yaft} appears to be the partner for our theory in a variant of the Atiyah-Floer conjecture.

The idea of our construction is roughly the following. Let $Y$ be a closed, oriented $3$-manifold and $\Sigma_g$ a genus $g$ Heegaard surface in $Y$, so that $Y = H_\alpha \cup_{\Sigma_g} H_\beta$ for two genus $g$ handlebodies $H_\alpha$, $H_\beta$. Now, choose a point $z \in \Sigma_g$ and in a small neighborhood of $z$, remove a regular neighborhood of a $\theta$-graph (\emph{i.e.} a graph with two vertices and three edges connecting them) from $Y$, where the $\theta$-graph is embedded so that each edge intersects $\Sigma_g$ once and each handlebody $H_\alpha$, $H_\beta$ contains one of the vertices. Write $\Sigma_g^\theta$, $H_\alpha^\theta$, and $H_\beta^\theta$ for the intersections of the pieces of the Heegaard decomposition with the complement of this $\theta$-graph.

Now, instead of looking at all conjugacy classes of $\SU(2)$-representations of the fundamental groups of each piece of the decomposition, we will add the condition that meridians of the edges of the $\theta$-graph should be sent to the conjugacy class of traceless $\SU(2)$ matrices. Hence, for example, the appropriate character variety to associate to $\Sigma_g^\theta$ should be
\[
	\scrR_{g,3} = \left.\left\{ A_1, B_1, \dots, A_g, B_g, C_1, C_2, C_3 \in \SU(2) ~\left|~ \begin{array}{c}\prod_{k = 1}^g [A_k, B_k] = C_1C_2C_3, \\ \operatorname{Tr}(C_k) = 0 \end{array}\right\}\right.\right/\text{conj}.
\]
Write $L_\alpha$, $L_\beta \subset \scrR_{g,3}$ for the images of the traceless character varieties of $H_\alpha^\theta$, $H_\beta^\theta$ under restriction to the boundary.

Similar to the case before the $\theta$-graph was removed, $\scrR_{g,3}$ is symplectic and $L_\alpha$, $L_\beta$ are Lagrangian submanifolds of $\scrR_{g,3}$. Furthermore, $\scrR_{g,3}$ is in fact smooth and compact, and we may define the Lagrangian Floer groups $\HF(L_\alpha, L_\beta)$ without difficulty. Our first main result is that this Floer homology is an invariant of the original $3$-manifold $Y$ (see Corollary \ref{cor:naturality}):

\begin{thm}
The Lagrangian Floer homology described above is independent of the Heegaard splitting of $Y$, and it associates a finitely generated abelian group $\SI(Y) = \HF(L_\alpha, L_\beta)$ to $Y$.
\end{thm}

We call $\SI(Y)$ the {\bf symplectic instanton homology} of $Y$. One of the first properties of symplectic instanton homology that we can prove is that it obeys a K\"unneth principle for connected sums (see Theorem \ref{thm:connectsum}):
\[
	\SI(Y \# Y') \cong (\SI(Y) \otimes \SI(Y')) \oplus \Tor(\SI(Y), \SI(Y')).
\]
There is further structure possessed by symplectic instanton homology. Starting from the seminal work of Floer on instanton homology \cite{floer-dehn} (see also \cite{braam-donaldson}), it has been noted that Floer homology invariants of closed $3$-manifolds behave in a controlled way under Dehn surgery. More precisely, let $Y$ be a closed, oriented $3$-manifold and let $(K, \lambda)$ be a framed knot in $Y$. Letting $Y_{\lambda}(K)$ denote the result of $\lambda$-surgery on $K$ (and similarly defining $Y_{\lambda + \mu}(K)$ for $\mu$ a meridian of $K$), there are standard cobordisms $W: Y \longrightarrow Y_\lambda(K)$, $W_\lambda: Y_\lambda(K) \longrightarrow Y_{\lambda + \mu}(K)$, and $W_{\lambda + \mu}: Y_{\lambda + \mu}(K) \longrightarrow Y$, each of which consists of a single $2$-handle attachment. For a ``Floer homology theory''\footnote{Here $\Bord_3^0$ denotes the category whose objects are \emph{connected} closed oriented $3$-manifolds and whose morphisms are \emph{connected} compact oriented $4$-manifolds with appropriate boundary.}
\[
	\A: \Bord_3^0 \longrightarrow \F_2 \,\mathbf{\text{-}Vect},
\]
there should be an exact triangle of the form
\[
	\xymatrix{\A(Y) \ar[rr]^{\A(W)} & & \A(Y_\lambda(K)) \ar[dl]^{\A(W_\lambda)} \\
	 & \A(Y_{\lambda + \mu}(K)) \ar[ul]^{\A(W_{\lambda+\mu})} & }
\]
Exact triangles of this form have been established for many types of Floer homologies of $3$-manifolds: instanton homology \cite{floer-dehn}, monopole Floer homology \cite{monopoles-lens-space}, and Heegaard Floer homology \cite{oz-sz-app} are some prominent examples.

In this article, we establish a surgery exact triangle for symplectic instanton homology of the form described above (although we will postpone the functorial identification of the maps to another article \cite{horton2}). As was the case for Floer's classical instanton homology, we will be required to extend our definition of symplectic instanton homology to incorporate additional structure on $Y$ for the exact triangle to possibly hold. While Floer's invariant needed to be generalized from $\SU(2)$-bundles to $\SO(3)$-bundles on the $3$-manifold $Y$, our symplectic instanton homology needs to take into account $\SO(3)$-bundles on $Y$. In dimension $3$, such bundles are classified by their second Stiefel-Whitney class, and by Poincar\'e duality we may therefore think of this data as coming from a mod $2$ homology class $\omega \in H_1(Y; \F_2)$.

With the above stated, we may describe the major results of this article concerning symplectic instanton homology for nontrivial $\SO(3)$-bundles and Dehn surgery. The first is the construction of an extension of $\SI(Y)$ that takes into account the additional data of a mod $2$ homology class $\omega \in H_1(Y; \F_2)$.

\begin{thm}
For any closed, oriented $3$-manifold $Y$ and homology class $\omega \in H_1(Y; \F_2)$, there is a well-defined {\bf symplectic instanton homology group} $\SI(Y,\omega)$ that is a natural invariant of the pair $(Y,\omega)$.
\end{thm}

With the invariance and functoriality of $\SI(Y,\omega)$ in place, we can now properly state the surgery exact triangle for symplectic instanton homology:

\begin{thm}
For any framed knot $(K, \lambda)$ in a closed, oriented $3$-manifold $Y$, there is an exact triangle of the form
\[
	\xymatrix{\SI(Y, \omega_K) \ar[rr] & & \SI(Y_\lambda(K)) \ar[dl] \\
	 & \SI(Y_{\lambda + \mu}(K)) \ar[ul] & }
\]
\end{thm}

Note that the surgery exact triangle is for Dehn surgery on a \emph{knot}. For general \emph{links}, the surgery exact triangle must be replaced with a spectral sequence. The construction of this spectral sequence and discussion of its applications can be found in the sequel to this article \cite{horton2}.

\textbf{Acknowledgements.} We thank Paul Kirk, Dylan Thurston, and Chris Woodward for useful conversations on various aspects of this work. We also thank Ciprian Manolescu for comments on the first version of this article which lead to a more careful proof of naturality. This article consists of material which makes up part of the author's Ph.D thesis, written at Indiana University.

%% file: flat-moduli.tex
\section{Symplectic Geometry of Flat Moduli Spaces}

\label{chap:flat-moduli}

Here we will recall the basic definitions and properties of the moduli spaces which we will be concerned with in this work. In what follows, we use the following identification between $\SU(2)$ and the unit quaternions freely:
\[
	\begin{pmatrix} a & b \\ -\bar{b} & \bar{a} \end{pmatrix} \leftrightarrow a + b\bfj.
\]
Note that an $\SU(2)$ matrix having trace zero is equivalent to the corresponding unit quaternion having zero real part.

\subsection{Smooth Topology of Traceless Character Varieties}

\label{sect:smooth-moduli}

Let $M$ be a compact $3$-manifold (possibly with boundary) or a closed surface (without boundary). Let $T \subset M$ be a properly embedded codimension-$2$ submanifold. For each connected component $T_i$ of $T$, one may choose an associated meridian $\mu_i \in \pi_1(M \setminus T)$, which is unique up to conjugation. The {\bf traceless $\SU(2)$-character variety} of $(M,T)$ is the representation space
\[
	\scrR(M,T) = \{\rho: \pi_1(M \setminus T) \longrightarrow \SU(2) \mid \Tr(\rho(\mu_i)) = 0\}/\text{conjugation}.
\]

Because of its importance in what is to follow, we introduce a special notation for the traceless character varieties of punctured surfaces:
\[
	\scrR_{g,n} = \scrR(\Sigma_g, \{n \text{ pts}\}).
\]
Note that fixing a standard basis for $\pi_1(\Sigma_g \setminus \{n \text{ pts}\})$ gives the {\bf holonomy description}
\[
	\scrR_{g,n} = \left.\left\{ A_1, \dots, A_g, B_1, \dots, B_g, C_1, \dots, C_n \in \SU(2) ~\left|~ \begin{array}{c} \displaystyle \prod_{j = 1}^g [A_j, B_j] \prod_{k = 1}^n C_k = I \\ \Tr(C_k) = 0 \end{array} \right\}\right/\text{conj.}\right.,
\]
and we can therefore denote elements of $\scrR_{g,n}$ by $[A_1, \dots, A_g, B_1, \dots, B_g, C_1, \dots, C_n]$.

\begin{prop}
If $n$ is odd, then $\scrR_{g,n}$ is a smooth manifold of dimension $6g - 6 + 2n$.
\end{prop}

\begin{proof}
We must show that if $n$ is odd, then any traceless representation $\rho: \pi_1(\Sigma_{g,n}) \longrightarrow \SU(2))$ has central stabilizer. Note that the only possible stabilizers of a traceless representation (assuming $n > 0$) are $\Z/2$ or $\mathrm{U}(1)$. Hence we just need to show any such $\rho$ cannot have stabilizer a circle subgroup of $\SU(2)$.

Suppose for the sake of contradiction that $\rho$ does have a circle subgroup as its stabilizer. Then each $A_k, B_k, C_k$ lies in the same circle subgroup of $\SU(2)$, and in particular they all commute. Hence
\[
	C_1 C_2 \cdots C_n = 1.
\]
Note that any circle subgroup of $\SU(2)$ intersects the conjugacy class of traceless matrices precisely in a pair $\{Q, -Q\}$ for some $Q \in \SU(2)$ with $\operatorname{Tr}(Q) = 0$. Therefore $C_j = \pm Q$ for each $j = 1, \dots, n$ and we see that
\[
	\pm Q^n = 1
\]
for some odd $n$. Since all traceless $\SU(2)$-matrices have order $4$, the above equation is equivalent to $\pm Q = 1$, which is not possible when $\operatorname{Tr}(Q) = 0$.
\end{proof}

On the other hand, we have the following:

\begin{prop}
If $n$ is even, then $\scrR_{g,n}$ has singularities.
\end{prop}

\begin{proof}
It suffices to find a representation with noncentral stabilizer. One such representation is
\begin{align*}
	A_k \mapsto \bfi, & \quad \quad k = 1, \dots, g; \\
	B_k \mapsto \bfi, & \quad \quad k = 1, \dots, g; \\
	C_j \mapsto (-1)^j \bfi, & \quad \quad j = 1, \dots, n.
\end{align*}
This certainly defines a traceless representation $\rho \in \Hom(\pi_1(\Sigma_g \setminus \{n \text{ pts}\}), \SU(2))$ for $n$ even, and the stabilizer of $\rho$ is the circle subgroup through $\bfi$. Therefore $[\rho] \in \scrR_{g,n}$ is a singular point of the traceless character variety.
\end{proof}

We can explicitly describe $\scrR_{g,n}$ for $g = 0$ and $n$ small.

\begin{exmp}
\label{exmp:Rgn}
(1) $\scrR_{0,3}$ is a single point. To see this, suppose $C_1, C_2, C_3 \in \SU(2)$ are traceless and $C_1 C_2 C_3 = 1$. We may conjugate so that $C_1 = \bfi$, and we may further conjugate, preserving the condition $C_1 = \bfi$, so that $C_2$ lies in the $\bfi\bfj$-plane, \emph{i.e.} $C_2 = e^{\gamma \bfk}\bfi$ for some $\gamma \in [0, \pi]$. Now
\[
	\bfi e^{\gamma \bfk}\bfi C_3 = 1 \quad \implies \quad C_3 = - e^{\gamma \bfk}.
\]
The requirement that $\operatorname{Tr}(C_3) = 0$ and $\gamma \in [0, \pi]$ forces $\gamma = \pi/2$, so that $C_2 = \bfj$ and $C_3 = -\bfk$. Hence the only element of $\scrR_{0,3}$ is $[\bfi, \bfj, -\bfk]$.

(2) $\scrR_{0,4}$ is the pillowcase orbifold, as is shown in \cite[Proposition 3.1]{HHK1}. Hedden, Herald, and Kirk \cite{HHK2} study the Lagrangian Floer homology of certain traceless character varieties of tangles in $\scrR_{0,4}$, and in certain explicit examples they find the resulting Floer homology matches Kronheimer and Mrowka's singular instanton homology \cite{Kh-detector}, suggesting a possible relationship by a variant of the Atiyah-Floer conjecture. 

(3) $\scrR_{0,5} \cong \C P^2 \# 5\overline{\C P}^2$. To see this, we first note that $\scrR_{0,5}$ is a del Pezzo surface, since it is K\"ahler (by the Mehta-Seshadri theorem, see Remark \ref{rem:mehta-seshadri} below) and monotone (as explained in Proposition \ref{prop:monotone} below). Del Pezzo surfaces are completely classified: smoothly, they are diffeomorphic to $S^2 \times S^2$ or $\C P^2 \# d\overline{\C P}^2$ ($d \leq 8$). Therefore we can determine $\scrR_{0,5}$ by computing its homology. Boden \cite{boden-orbifold} showed that the Poincar\'e polynomial of $\scrR_{0,5}$ is
\[
	P_t(\scrR_{0,5}) = 1 + 6t^2 + t^4,
\]
so that by the fact that $\scrR_{0,5}$ is del Pezzo, we have $\scrR_{0,5} \cong \C P^2 \# 5\overline{\C P}^2$.

(4) $\scrR_{0,6}$ is a branched double cover of $\C P^3$ with branch set a singular Kummer surface $K$. Here one identifies $\C P^3$ with the classical genus $2$ $\SU(2)$-character variety $\chi(\Sigma_2)$, and $K$ is identified as the Jacobian variety of abelian representations $\mathrm{Jac}(\Sigma_2) \subset \Hom(\pi_1(\Sigma_2), \SU(2))$ modulo complex conjugation. The details of this identification are given in \cite{kirk-6puncture}.
\end{exmp}

\subsection{Gauge Theory and Tangent Spaces}

\label{sect:gauge-theory}

The traceless character variety $\scrR_{g,n}$ has a very useful description as a particular moduli space of flat connections. In this section, we give the details of this gauge theoretic interpretation of $\scrR_{g,n}$. Throughout, we fix a maximal torus $\T_{\bfi} = \{\exp(2\pi \bfi t) \mid t \in [0, 1)\}$ for $\SU(2)$ as well as its Lie algebra $\mathfrak{t}_{\bfi} = \{\bfi t \mid t \in \R \} \subset \su(2)$.

Let $\Sigma_{g,n} = \Sigma_g \setminus \{p_1, \dots, p_n\}$ denote the surface of genus $g$ with $n$ punctures. Fix disjoint holomorphic coordinate charts $(D_i, z_i)$ near each $p_i$ in $\Sigma_g$ such that $z_i: D_i \longrightarrow \D$ is a complex analytic isomorphism (in particular, each chart is a disk) with $z_i^{-1}(0) = p_i$. If we write $w_i = -\log z_i$, $D_i^\ast = D_i \setminus \{p_i\}$, and $C_i = \{(\tau, \theta) \in (0, \infty) \times [0,2\pi]\}/((\tau, 0) \sim (\tau, 2\pi))$, then we obtain a parametrization
\[
	w_i: D_i^\ast \longrightarrow C_i
\]
of each noncompact end of $\Sigma_{g,n}$ as an infinite cylinder.

Let $P$ denote the trivial $\SU(2)$-bundle over $\Sigma_{g,n}$. Fix a base connection $\nabla_0$ on $P$ that has the form $d + \tfrac{1}{4}\bfi \, d\theta_i$ on each cylindrical end $D_i^\ast \subset \Sigma_{g,n}$ (note that such a connection has holonomy $\exp(2\pi \cdot \tfrac{1}{4}\bfi) = \bfi$ around each puncture). We then have an affine space of connections
\[
	\A_{\bfi}(\Sigma_{g,n}) = \nabla_0 + \{a \in \Omega^1(\Sigma_{g,n}; \su(2)) \mid a|_{D_i^\ast} \equiv \xi_i \, d\theta_i, \xi_i \in \mathfrak{t}_{\bfi}, \text{ for } i = 1, \dots, n\}
\]
all of which clearly have traceless holonomy around each puncture $p_i$. The natural group of gauge transformations acting on $\A_{\bfi}(\Sigma_{g,n})$ in this case is
\[
	\G_{\bfi}(\Sigma_{g,n}) = \{\varphi: \Sigma_{g,n} \longrightarrow \SU(2) \mid \varphi|_{D_i^\ast} \equiv g_i \in \T_{\bfi} \text{ for } i = 1, \dots, n\}.
\]
There is also a group of ``asymptotically trivial'' gauge transformations
\[
	\G_{\bfi,0}(\Sigma_{g,n}) = \{\varphi \in \G_{\bfi}(\Sigma_{g,n}) \mid \varphi|_{D_i^\ast} \equiv 1 \text{ for } i = 1, \dots, n\}
\]
which fits into a short exact sequence
\[
	0 \longrightarrow \G_{\bfi,0}(\Sigma_{g,n}) \longrightarrow \G_{\bfi}(\Sigma_{g,n}) \longrightarrow \T_{\bfi}^n \longrightarrow 0.
\]
If $\A_{\bfi}^\flat(\Sigma_{g,n}) \subset \A_{\bfi}(\Sigma_{g,n})$ denotes the subspace of flat connections, then there are two moduli spaces of interest:
\[
	\M_{\bfi}(\Sigma_{g,n}) = \A_{\bfi}^\flat(\Sigma_{g,n})/\G_{\bfi}(\Sigma_{g,n}), \quad\quad\quad \scrF_{\bfi}(\Sigma_{g,n}) = \A_{\bfi}^\flat(\Sigma_{g,n})/\G_{\bfi,0}(\Sigma_{g,n}).
\]
Clearly there is a fiber bundle
\begin{align}
	\label{eqn:gauge-fibration}
	\xymatrix{\T_{\bfi}^n/\{\pm 1\} \ar[r] & \scrF_{\bfi}(\Sigma_{g,n}) \ar[d] \\ & \M_{\bfi}(\Sigma_{g,n})}
\end{align}
(where the fibers are orbits under the action of $\G_{\bfi}/\G_{\bfi,0} \cong \T_{\bfi}^n$; the fiber is $\T_{\bfi}^n/\{\pm 1\}$ since $\{\pm 1\}$ is the stabilizer of each orbit) and the usual holonomy correspondence gives an identification $\M_{\bfi}(\Sigma_{g,n}) \cong \scrR_{g,n}$.

Although the definition of $\scrR_{g,n}$ as a traceless $\SU(2)$-character variety is appealing (and we will use this description frequently), many basic properties of $\scrR_{g,n}$ are easier to see by considering the equivalent space $\M_{\bfi}(\Sigma_{g,n})$ instead. As a first example, we determine the tangent space $T_{[A]}\M_{\bfi}(\Sigma_{g,n})$ at an arbitrary point $[A] \in \M_{\bfi}(\Sigma_{g,n})$. Define
\[
	\Omega^1_{\mathfrak{t}_{\bfi}}(\Sigma_{g,n}; \su(2)) = \{a \in \Omega^1(\Sigma_{g,n}; \su(2)) \mid a|_{D_i^\ast} \equiv \xi_i \, d\theta_i, \xi_i \in \mathfrak{t}_{\bfi}, \text{ for } i = 1, \dots, n\},
\]
\[
	\Omega^k_0(\Sigma_{g,n}; \su(2)) = \{a \in \Omega^k(\Sigma_{g,n}; \su(2)) \mid a|_{D_i^\ast} \equiv 0 \text{ for } i = 1, \dots, n\}.
\]
Note that the definition of $\A_{\bfi}(\Sigma_{g,n})$ as an affine space implies that
\[
	T_A \A_{\bfi}(\Sigma_{g,n}) = \Omega^1_{\mathfrak{t}_{\bfi}}(\Sigma_{g,n}; \su(2)),
\]
and the linearization of the flatness condition gives
\[
	T_A \A_{\bfi}^\flat(\Sigma_{g,n}) = \ker(d_A: \Omega^1_{\mathfrak{t}_{\bfi}}(\Sigma_{g,n}; \su(2)) \longrightarrow \Omega_0^2(\Sigma_{g,n}; \su(2))).
\]
To determine the tangent spaces of the quotients $\scrF_{\bfi}(\Sigma_{g,n})$ and $\M_{\bfi}(\Sigma_{g,n})$, note that the tangent space (in $\A_{\bfi}^\flat(\Sigma_{g,n})$) at $A$ to the gauge orbit $\G_{\bfi,0}(\Sigma_{g,n}) \cdot A$ (resp.\ $\G_{\bfi}(\Sigma_{g,n}) \cdot A$) consists of $1$-forms $d_A \varphi$, where $\varphi \in \Omega_0^0(\Sigma_{g,n}; \su(2))$ (resp.\ $\varphi \in \Omega_{\mathfrak{t}_{\bfi}}^0(\Sigma_{g,n}; \su(2))$), and therefore passing to quotients gives
\[
	T_{[A]} \scrF_{\bfi}(\Sigma_{g,n}) = \ker(d_A: \Omega_{\mathfrak{t}_{\bfi}}^1(\Sigma_{g,n}; \su(2)) \longrightarrow \Omega_0^2(\Sigma_{g,n}; \su(2))),
\]
\[
	T_{[A]} \M_{\bfi}(\Sigma_{g,n}) = T_{[A]} \scrF_{\bfi}(\Sigma_{g,n})/\{a \in T_{[A]} \scrF_{\bfi}(\Sigma_{g,n}) \mid a|_{D_i^\ast} = d_A s_i\}.
\]

\begin{rem}
\label{rem:mehta-seshadri}
We have considered $\scrR_{g,n}$ in two different ways thus far: as a space of conjugacy classes of traceless $\SU(2)$-representations of $\pi_1(\Sigma_{g,n})$, and as a space of gauge equivalence classes of flat $\SU(2)$-connections on $\Sigma_{g,n}$ with traceless holonomy around the boundary components. There is a third point of view that describes $\scrR_{g,n}$ algebro-geometrically. The Mehta-Seshadri theorem \cite{mehta-seshadri} states that $\scrR_{g,n}$ can be identified with a moduli space of semi-stable rank $2$ holomorphic parabolic bundles on $\Sigma_g$, with all parabolic weights equal to $1/4$. We will not use this description in any of our results, so we avoid going into details here.
\end{rem}

\subsection{Algebraic Topology of $\scrR_{g,n}$}

\label{sect:homotopy-gps}

We will make use of some specific information about low-dimensional homotopy groups of $\scrR_{g,n}$. First is the well-known fact that $\scrR_{g,n}$ is $1$-connected (see for example, \cite[Theorem 2.1]{andersen}):

\begin{prop}
$\scrR_{g,n}$ is connected and simply connected.
\end{prop}

Furthermore, $\pi_2(\scrR_{g,n})$ is of interest to us, as the differential on our Floer homology groups count holomorphic representatives of certain relative homotopy classes of disks in $\scrR_{g,n}$.

\begin{prop}
\label{homotopyModuli}
When $n > 0$, $\pi_2(\scrR_{g,n})$ is free of rank $n+1$.
\end{prop}

\begin{proof}
Recursive relations for the Poincar\'e polynomials of various character varieties have long been known. Boden \cite{boden-orbifold} gave a recursive formula for $P_t(\scrR_{g,n})$ in particular, and more recently Street \cite[Theorem 3.8]{street} has shown that Boden's recursive formula can actually be made explicit:
\[
	P_t(\scrR_{g,n}) = \frac{(1 + t^3)^n(1 + t^3)^{2g} - 2^{n-1}t^{2g+n-1}(1+t)^{2g}(1+t^2)}{(1 - t^2)(1 - t^4)}.
\]
From this formula, it is seen that $H^2(\scrR_{g,n}) = \Z^{n+1}$. Since $\scrR_{g,n}$ is $1$-connected, the Hurewicz theorem then implies that $\pi_2(\scrR_{g,n}) = \Z^{n+1}$.
\end{proof}

We can actually describe the $n + 1$ generators of $\pi_2(\scrR_{g,n})$ in more detail. The long exact sequence of the fibration (\ref{eqn:gauge-fibration}) gives
\[
	\xymatrix@R=10pt@C-10pt{\pi_2(\T_{\bfi}^n/\{\pm 1\}) \ar[r] \ar@{=}[d] & \pi_2(\scrF_{\bfi}(\Sigma_{g,n})) \ar[r] \ar@{=}[d] & \pi_2(\scrR_{g,n}) \ar[r] \ar@{=}[d] & \pi_1(\T_{\bfi}^n/\{\pm 1\}) \ar[r] \ar@{=}[d] & \pi_1(\scrF_{\bfi}(\Sigma_{g,n})) \ar@{=}[d] \\
	0 & \Z & \pi_2(\scrF_{\bfi}(\Sigma_{g,n}), \T_{\bfi}^n/\{\pm 1\}) & \Z^n & 0 }
\]
One of the generators of $\pi_2(\scrR_{g,n})$, which we denote $\beta$, is simply the image of the generator of $\pi_2(\scrF_{\bfi}(\Sigma_{g,n}))$ under the projection $\pi: \scrF_{\bfi}(\Sigma_{g,n}) \longrightarrow \M_{\bfi}(\Sigma_{g,n}) = \scrR_{g,n}$. In fact, one may construct a universal rank $2$ complex vector bundle $\mathbb{U} \longrightarrow \Sigma_{g} \times \scrR_{g,n}$ and show that $\beta$ is the Poincar\'e dual of $c_2(\mathbb{U})/[\Sigma_g]$, where $/$ denotes the slant product \cite{biswas-canonical}.

The remaining generators $\gamma_1, \dots, \gamma_n$ are constructed using generators of $\pi_1(\T_{\bfi}^n/\{\pm 1\}) \cong \Z^n$ as follows. Fix a base connection $\nabla_0 = d + a_0 \in \A_{\bfi}^\flat(\Sigma_{g,n})$ in temporal gauge, so that in particular $a_0 = \tfrac{1}{4}\bfi \, d\theta_k$ in the neighborhood $(D_k^\ast, w_k)$ of the $k^\text{th}$ puncture of $\Sigma_{g,n}$. If $f_k: [0,1] \longrightarrow \T_{\bfi}^n \subset \G_{\bfi}(\Sigma_{g,n})$ is the loop of gauge transformations (unique up to homotopy) such that $f_k(t) \equiv e^{\bfi \pi t}$ in the neighborhood $(D_k^\ast, w_k)$ of the $k^\text{th}$ puncture of $\Sigma_{g,n}$ and is the identity outside of a slightly larger neighborhood, then $f_k$ induces a loop $\R/\Z \longrightarrow \T_{\bfi}^n/\{\pm 1\}$ and the short exact sequence above implies there is a disk $\scrD_k$ in $\scrF_\bfi(\Sigma_{g,n})$ whose boundary is the loop $d + f_k(t)^{-1}a_0f_k(t)$. The image of $\scrD_k$ under the projection $\scrF_{\bfi}(\Sigma_{g,n}) \longrightarrow \M_{\bfi}(\Sigma_{g,n}) = \scrR_{g,n}$ represents the homotopy class $\gamma_k \in \pi_2(\scrR_{g,n})$.

\subsection{Symplectic Geometry of $\scrR_{g,n}$}

Going back to the work of Atiyah-Bott \cite{atiyah-bott} and Goldman \cite{goldman}, it has been observed that character varieties of surfaces admit naturally defined symplectic structures. This remains true for relative character varieties such as the traceless character varieties $\scrR_{g,n}$. In this section, we describe the symplectic form $\omega_{g,n}$ on $\scrR_{g,n}$ and some of its crucial properties.

On $T_{A}\A_{\bfi}^\flat(\Sigma_{g,n})$, we define a $2$-form
\begin{align}
	\label{eqn:symp-form}
	\omega_{g,n}(a, b) & = \frac{1}{4\pi^2} \int_{\Sigma_{g,n}} \Tr(a \wedge b).
\end{align}
$\omega_{g,n}$ is closed and degenerate in the $\G_{\bfi, 0}(\Sigma_{g,n})$-directions, and hence descends to $\scrF_{\bfi}(\Sigma_{g,n})$. It is also easily checked that $\omega_{g,n}$ is degenerate in the $\T_{\bfi}^n/\{\pm 1\}$-directions in $\scrF_{\bfi}$, so that it further descends to $\M_{\bfi}(\Sigma_{g,n}) = \scrR_{g,n}$.

We intend to use Lagrangian Floer homology in $(\scrR_{g,n}, \omega_{g,n})$, and there are serious technical difficulties in defining such Floer homologies in general. Therefore is it preferable to ensure that the symplectic geometry of $(\scrR_{g,n}, \omega_{g,n})$ is suitably simple. As explained by Oh \cite{oh-I}, the technical details behind the Floer homology of certain Lagrangians in monotone symplectic manifolds are fairly simple. Recall that $(M, \omega)$ is {\bf monotone} (with {\bf monotonicity constant} $\tau$) if there exists some $\tau > 0$ such that
\[
	[\omega] = \tau c_1(TM).
\]
The symplectic structure $\omega_{g,n}$ on $\scrR_{g,n}$ turns out to be monotone:

\begin{prop}
\label{prop:monotone}
$(\scrR_{g,n}, \omega_{g,n})$ is a monotone symplectic manifold with monotonicity constant $\tau = \tfrac{1}{4}$.
\end{prop}

\begin{proof}
This is proved in \cite[Theorem 4.2]{meinrenken-woodward}, but we sketch an alternative proof in order to emphasize why we have chosen to send meridians of punctures in $\Sigma_{g,n}$ to traceless $\SU(2)$-matrices rather than some other conjugacy class in $\SU(2)$.

Consider the {\bf extended moduli space} \cite{jeffrey-extended} $\M^{\g}(\Sigma_{g,n})$ of flat connections on the trivial $\SU(2)$-bundle over $\Sigma_{g,n}$, defined as follows. As before, think of $\Sigma_{g,n}$ as a genus $g$ surface with $n$ punctures, and fix cylindrical neighborhoods $(D_k^\ast, w_k)$ centered at each puncture, with the cylinder coordinates of the $k^\text{th}$ boundary component written $(t_k, \theta_k)$. We still denote by $\G_{\bfi,0}(\Sigma_{g,n})$ the group of gauge transformations that are the identity on all of the collar neighborhoods chosen above. Then
\[
	\M^{\g}(\Sigma_{g,n}) = \{(A, \xi_1, \dots, \xi_n) \in \A^\flat(\Sigma_{g,n}) \times \su(2)^n \mid A|_{D_k^\ast} = \xi_k \, d\theta_k, k = 1, \dots, n\}/\G_{\bfi,0}(\Sigma_{g,n}).
\]
Jeffrey \cite{jeffrey-extended} shows that $\M^{\g}(\Sigma_{g,n})$ has an open smooth stratum, and on this stratum the equation (\ref{eqn:symp-form}) defines a symplectic form on tangent spaces
\[
	T_{[A, \xi_1, \dots, \xi_n]} \M^{\g}(\Sigma_{g,n}) = \frac{\ker(d_A: \Omega_{\g}^1(\Sigma_{g,n}; \su(2)) \longrightarrow \Omega_0^2(\Sigma_{g,n}; \su(2)))}{\im(d_A: \Omega_0^0(\Sigma_{g,n}; \su(2)) \longrightarrow \Omega_{\g}^1(\Sigma_{g,n}; \su(2)))},
\]
where
\[
	\Omega^1_{\g}(\Sigma_{g,n}; \su(2)) = \{a \in \Omega^1(\Sigma_{g,n}; \su(2)) \mid a|_{D_i^\ast} \equiv \xi_i \, d\theta_i, \xi_i \in \su(2), \text{ for } i = 1, \dots, n\}.
\]
The map
\[
	\Lambda: \M^{\g}(\Sigma_{g,n}) \longrightarrow \su(2)^n,
\]
\[
	[A, \xi_1, \dots, \xi_n] \mapsto (\xi_1, \dots, \xi_n)
\]
is also shown to be a moment map for the obvious action of $(G^{\ad})^n$ on $\M^{\g}(\Sigma_{g,n})$. Note in particular that if $\O_{\bfi/4}$ denotes the (co)adjoint orbit of $\tfrac{1}{4}\bfi$ in $\su(2)$, then $\Lambda^{-1}(\O_{\bfi/4} \times \cdots \times \O_{\bfi/4})$ fibers over $\scrF_{\bfi}(\Sigma_{g,n})$ with fiber $\O_{\bfi/4} \times \cdots \times \O_{\bfi/4}$, and furthermore
\[
	\scrR_{g,n} = \M_{\bfi}(\Sigma_{g,n}) = \Lambda^{-1}(\O_{\bfi/4} \times \cdots \times \O_{\bfi/4})/(G^{\ad})^n,
\]
so that $\scrR_{g,n}$ is a symplectic reduction of $\M^{\g}(\Sigma_{g,n})$ by a Hamiltonian $(G^{\ad})^n$-action.

$\M^{\g}(\Sigma_{g,n})$ is known to be monotone with monotonicity constant $\tfrac{1}{4}$ \cite[Theorem 3.4]{manolescu-woodward}, and furthermore the Kostant-Kirillov-Souriau symplectic form on the coadjoint orbit $\O_{\bfi/4}$ is monotone with monotonicity constant $\tfrac{1}{4}$ as well -- this is the special property of the conjugacy class of traceless elements of $\SU(2)$, which are the image of $\O_{\bfi/4}$ under $\exp(2\pi \bullet)$, that has led us to single them out in our constructions. Now, according to \cite[Lemma 4.5]{manolescu-woodward}, the cohomology class of the reduced symplectic form $\omega_{g,n}$ on $\scrR_{g,n}$ satisfies $[\omega_{g,n}] = \tfrac{1}{4}c_1(\scrR_{g,n})$, by plugging in $\kappa = \lambda = \tfrac{1}{4}$ into their formula. Therefore $\scrR_{g,n}$ is monotone with monotonicity constant $\tau = \tfrac{1}{4}$, as claimed.
\end{proof}

Monotonicity on its own isn't enough to ensure well-definedness of Floer homology, but it is a good start. What is further needed is some condition that ensures we have control over disk and sphere bubbles in the Gromov compactification of Maslov index $2$ pseudoholomorphic disks. The codimension of sphere bubbles is controlled by the first Chern class of $(M, \omega)$. In particular, if one defines the {\bf minimal Chern number} $N_M$ of $(M, \omega)$ to be the positive generator of the image of
\[
	c_1(TM): \pi_2(M) \longrightarrow \Z,
\]
where we consider $\pi_2(M) \subset H_2(M;\Z)$ via the Hurewicz map, we will be able to conclude that sphere bubbles in the Gromov compactification occur in codimension at least $2N_M$. If $L \subset M$ is a Lagrangian submanifold, then the {\bf minimal Maslov number} $N_L$ of $L$ is the positive generator of the image of the Maslov class
\[
	\mu_L: \pi_2(M,L) \longrightarrow \Z.
\]
$L$ is said to be {\bf monotone} if $2[\omega]|_{\pi_2(M,L)} = \kappa\mu_L$ for some $\kappa > 0$; $\kappa$ is necessarily the same as the monotonicity constant for $(M,\omega)$ in this case. When $L$ is simply connected (as will be the case in our application), then $L$ is automatically monotone (as long as $(M, \omega)$ is) and $N_L = 2N_M$. We will study disk bubbles for our Lagrangians later in Section \ref{sect:lagrangians}; for the rest of this section we focus on computing the minimal Chern number $N_{\scrR_{g,n}}$.

When $(M, \omega)$ is monotone with rational monotonicity constant (as is the case for $(\scrR_{g,n}, \omega_{g,n})$), it is clear that we can constrain the minimal Chern number if a suitable multiple of $[\omega]$ is an integral cohomology class. Therefore we establish the following:

\begin{prop}
$4[\omega_{g,n}] \in H^2(\scrR_{g,n}; \R)$ is a primitive integral cohomology class.
\end{prop}

\hide{
\begin{proof}
The idea is to generalize the fact that the Chern-Simons line bundle \cite{rsw} is a prequantum line bundle for the natural symplectic form on $\chi(\Sigma_g)$ to the case of the traceless character varieties $\scrR_{g,n}$. Andersen et al. \cite{andersen} give a definition of the Chern-Simons line bundle $\tilde{\L}_\text{CS} \longrightarrow \A^\flat(\Sigma_{g,n})$ on the space of flat $\SU(2)$-connections on the trivial bundle over the punctured surface $\Sigma_{g,n}$, and furthermore they give conditions on the conjugacy classes of the holonomies around the punctures that must be satisfied for the tensor power $\tilde{\L}_\text{CS}^{\otimes k}$ to descend to a line bundle $\L_\text{CS}^{\otimes k} \longrightarrow \A^\flat(\Sigma_{g,n})/\G(\Sigma_{g,n})$. 

It turns out that for traceless holonomies around each puncture (and an odd number of punctures), $\tilde{\L}_\text{CS}$ does not descend to $\scrR_{g,n}$, but $\tilde{\L}_\text{CS}^{\otimes 2}$ does by \cite[Theorem 3.6]{andersen}. Furthermore, $\L_\text{CS}^{\otimes 2}$ is a prequantum line bundle for $(\scrR_{g,n}, 2\omega_{g,n})$, which implies that $2[\omega_{g,n}]$ is an integral cohomology class.
\end{proof}
}

\begin{proof}
We show this directly, by evaluating $\omega_{g,n}$ on the generators $\beta, \gamma_1, \dots, \gamma_n$ of $\pi_2(\scrR_{g,n}) \cong H_2(\scrR_{g,n}; \Z)$ described in Section \ref{sect:homotopy-gps}.

First, note that $\langle \omega_{g,n}, \beta \rangle = 1$ by \cite[Section 5.4]{street} and the fact that $\beta$ is the Poincar\'e dual of $c_2(\mathbb{U})/[\Sigma_g]$ (Street uses Pontryagin classes instead of Chern classes, so our $\beta$ is $-4$ times the Poincar\'e dual of what he denotes $[\Sigma]$).

Now we evaluate $\langle \omega_{g,n}, \gamma_k\rangle$ ($k = 1, \dots, n$) by working in $\scrF_\bfi$. If $f_k: [0,1] \longrightarrow \T_{\bfi}^n \subset \G_{\bfi,0}$ is the path of gauge transformations (unique up to homotopy) such that $f_k(t) \equiv e^{\bfi \pi t}$ in a collar neighborhood of the $k^\text{th}$ boundary component of $\Sigma_{g,n}$ and is the identity outside of a slightly larger collar neighborhood, then as noted in Section \ref{sect:homotopy-gps}, there is a disk $\scrD_k$ in $\scrF_\bfi$ whose boundary is the loop $f_k(t)^{-1}A_0f_k(t)$ whose image in $\scrR_{g,n}$ is precisely the sphere $\gamma_k$. If we let $\mathpzc{a}: \D \longrightarrow \scrF_\bfi$ denote this disk, then we compute the following:
\begin{align*}
	\int_{\scrD_k} \omega_{g,n} & = \frac{1}{4\pi^2} \int_{\D} \int_{\Sigma_{g,n}} \Tr(d\mathpzc{a} \wedge d\mathpzc{a}) \\
	 & = \frac{1}{8\pi^2} \int_{\Sigma_{g,n}} \int_{\partial \D} \Tr(\mathpzc{a} \wedge d\mathpzc{a}) \\
	 & = \frac{1}{8\pi^2} \int_{\Sigma_{g,n}} \int_0^1 \Tr(f_k(t)^{-1}A_0f_k(t) \wedge d(f_k(t)^{-1}A_0f_k(t))) \, dt \\
	 & = \frac{1}{8\pi^2} \int_{\Sigma_{g,n}} \int_0^1 \Tr(f_k(t)^{-1}A_0f_k(t) \wedge [f_k(t)^{-1}A_0f_k(t),f_k(t)^{-1}f_k^\prime(t)]) \, dt \\
	 & = \frac{1}{8\pi^2} \int_{\Sigma_{g,n}} \int_0^1 \Tr(f_k^\prime(t)f_k(t)^{-1}[A_0, A_0]) \, dt \\
	 & = -\frac{1}{4\pi^2} \int_{\Sigma_{g,n}} \int_0^1 \Tr(f_k^\prime(t)f_k(t)^{-1}dA_0) \, dt \\
	 & = -\frac{1}{4\pi^2} \int_{\partial \Sigma_{g,n}} \int_0^1 \Tr(f_k^\prime(t)f_k(t)^{-1}A_0) \, dt \\
	 & = -\frac{1}{4\pi^2} \cdot 2\pi \int_0^1 2 \Re(\bfi \pi \cdot \tfrac{1}{4}\bfi) \,dt \\
	 & = \frac{1}{4}.
\end{align*}

It follows from the above computations that $\langle 4[\omega_{g,n}], \alpha\rangle \in \Z$ for any $\alpha \in H_2(\scrR_{g,n}; \Z)$, so that $4[\omega_{g,n}] \in H^2(\scrR_{g,n}; \Z) \subset H^2(\scrR_{g,n}; \R)$ is primitive.
\end{proof}

The above Proposition and the fact that $(\scrR_{g,n}, \omega_{g,n})$ is monotone with monotonicity constant $\tfrac{1}{4}$ immediately imply the following:

\begin{cor}
The minimal Chern number of $\scrR_{g,n}$ is $1$.
\label{cor:minchern}
\end{cor}

%% file: heegaard-lagrangians.tex
\section{Lagrangians and Whitney $n$-gons for Heegaard Splittings}

\label{sect:lagrangians}

The most commonly encountered type of handlebody decompositions of a $3$-manifold are Heegaard splittings. Since Heegaard splittings are easily visualized in terms of Heegaard diagrams and the definition of symplectic instanton homology greatly simplifies in this case, we make our initial definition of the symplectic instanton homology of a $3$-manifold $Y$ in terms of a Heegaard splitting $Y = H_0 \cup_{\Sigma_g} H_1$. We have included an exposition on the basic theory of Heegaard diagrams in Appendix \ref{sect:heegaard-diagrams} for the reader's convenience.

In this section, we fix a genus $g$ pointed Heegaard diagram $(\Sigma_g, \bfalpha, \bfbeta, \bfx)$. Here, $\bfx$ is a {\bf thick basepoint}, which means that it actually represents a small disk neighborhood with \emph{three} punctures in it (see Figure \ref{fig:3p-disk}). Hence, associated to $(\Sigma_g, \bfx)$ is the traceless character variety $\scrR_{g,3}$.

\begin{figure}[h]
	\centering
	\includegraphics[scale=1.25]{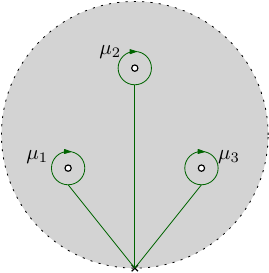}
	\caption{A thrice punctured disk with meridians $\mu_1$, $\mu_2$, and $\mu_3$ of the punctures pictured.}
	\label{fig:3p-disk}
\end{figure}

\subsection{Lagrangians from Heegaard Splittings}


To the set of attaching curves $\bfalpha = (\alpha_1, \dots, \alpha_g)$, we associate the set
\begin{align*}
	L_\alpha & = \{[\rho] \in \scrR_{g,3} : \rho([\alpha_1]) = \cdots = \rho([\alpha_g]) = I\} \subset \scrR_{g,3}.
\end{align*}
Note that $L_\alpha$ is well-defined independent of any choice of orientations of the curves $\alpha_k$, and it is furthermore independent of the choice of paths to the basepoint $x \in \Sigma_{g,3}$ needed to make each $\alpha_k$ into a based curve (\emph{i.e.} a representative of a class in $\pi_1(\Sigma_{g,3}, x)$).

A useful way to think about $L_\alpha$ is as follows: let $H_\alpha$ be the genus $g$ handlebody determined by the attaching curves $\bfalpha$. Thinking of the thrice-punctured surface $\Sigma_{g,3}$ as the boundary of $H_\alpha$, there is a standardly embedded ``tripod'' graph $\epsilon$ (see Figure \ref{fig:tripod}) in a neighborhood of $\partial H_\alpha$ whose intersection with $\partial H_\alpha$ corresponds to the three punctures in $\Sigma_{g,3}$. If we write $\mu_1$, $\mu_2$, and $\mu_3$ for the meridians of the three edges of $\epsilon$ in $\pi_1(H_\alpha \setminus \epsilon)$, we have the traceless character variety
\[
	\scrR(H_\alpha, \epsilon) = \{\rho \in \Hom(\pi_1(H_\alpha \setminus \epsilon), \SU(2)) \mid \Tr(\rho(\mu_k)) = 0\}/\text{conjugation}.
\]
It is then easy to see that $L_\alpha$ is the image of $\scrR(H_\alpha, \epsilon)$ in $\scrR_{g,3}$ under restriction to the boundary, using the fact that the product $\mu_1\mu_2\mu_3$ bounds an obvious disk in $H_\alpha$.

\begin{figure}[h]
	\centering
	\includegraphics[scale=1.25]{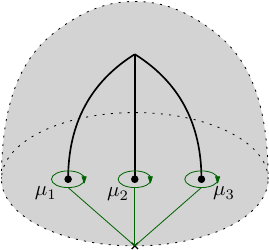}
	\caption{Local image of a ``tripod'' graph near the boundary of a handlebody, with meridians $\mu_1$, $\mu_2$, and $\mu_3$ labeled.}
	\label{fig:tripod}
\end{figure}

There is also a gauge theoretic description of $L_\alpha$ which is useful. Starting with the handlebody and graph pair $(H_\alpha, \epsilon)$ of the previous paragraph, remove a solid $3$-ball neighborhood of the $3$-valent vertex of $\epsilon$ from $H_\alpha$ to get a cobordism $W_\alpha: S^2 \longrightarrow \Sigma_g$ containing a trivial $3$-stranded tangle $T$. We may then consider the moduli space of flat $\SU(2)$-connections
\[
	\M(W_\alpha, T) = \{A \in \A(W_\alpha) \mid F_A = 0, \Hol_A(\mu_k) \in C(\bfi), k = 1,2,3\}/\G(W_\alpha).
\]
By restriction to the boundary, we get a map
\[
	\M(W_\alpha, T) \longrightarrow \M(S^2, \{3 \text{ pts}\}) \times \M(\Sigma_g, \{3 \text{ pts}\}) \cong \M(\Sigma_g, \{3 \text{ pts}\}).
\]
Under the holonomy correspondence, $\M(\Sigma_g, \{3 \text{ pts}\})$ is identified with $\scrR_{g,3}$, and the image of $\M(W_\alpha, T)$ in $\M(\Sigma_g, \{3 \text{ pts}\})$ is identified with $L_\alpha$.

\begin{prop}
$L_\alpha$ is an embedded Lagrangian submanifold of $\scrR_{g,3}$.
\end{prop}

\begin{proof}
In the case where $\bfalpha$ is the ``standard'' set of $\alpha$-curves (see Figure \ref{fig:standard-alpha} in Appendix \ref{sect:heegaard-diagrams}), we have
\[
	L_\alpha = \{[A_1, \dots, A_g, B_1, \dots, B_g, C_1, C_2, C_3] \in \scrR_{g,3} : A_k = I \text{ for } k = 1,\dots,g\},
\]
and it is clear that this standard $L_\alpha$ is a smooth, embedded submanifold of $\scrR_{g,3}$.

To see that the ``standard'' $L_\alpha$ is Lagrangian in $\scrR_{g,3}$, we use the gauge theoretic description of the moduli spaces and symplectic form. $L_\alpha \subset \scrR_{g,3}$ consists of gauge equivalence classes of $\SU(2)$-connections $[A]$ which extend to an $\SU(2)$-connection $[\tilde{A}]$ on $H_\alpha \setminus \nu T$ that has traceless holonomy around any of the edges of the tripod graph $T \subset H_\alpha$. Here $\nu T$ denotes a \emph{closed} regular neighborhood of $T$, so that $H_\alpha \setminus \nu T$ is noncompact.

The tangent space $T_{[A]} \scrR_{g,3}$ may be identified with
\[
	\im(H^1_c(\Sigma_{g,3}; \su(2)_A) \longrightarrow H^1(\Sigma_{g,3}; \su(2)_A))
\]
and the symplectic form $\omega$ on $T_{[A]} \scrR_{g,3}$ is given by
\[
	\omega_{[A]}(\xi, \eta) = \frac{1}{4\pi^2}\int_{\Sigma_{g,3}} \Tr (\xi \wedge \eta).
\]
Similarly, we have
\[
	T_{[\tilde{A}]} \scrR(H_\alpha, T) = \im(H^1_c(H_\alpha \setminus \nu T; \su(2)_{\tilde{A}}) \longrightarrow H^1(H_\alpha \setminus \nu T; \su(2)_{\tilde{A}})).
\]

Let $\tilde{\xi}$, $\tilde{\eta}$ be two tangent vectors to $\scrR(H_\alpha, T)$ at a point $[\tilde{A}]$, and denote their images in $L_\alpha$ by $\xi, \eta \in T_{[A]} \scrR_{g,3}$, where $[A]$ is the restriction of $[\tilde{A}]$ to $\Sigma_{g,3}$. By Stokes' theorem,
\begin{align*}
	0 & = \int_{H_\alpha \setminus \nu T} d\langle \tilde{\xi} \wedge \tilde{\eta} \rangle \\
	 & = \int_{\Sigma_{g,3}} \langle \xi \wedge \eta \rangle + \int_{\partial \nu T} \langle \tilde{\xi} \wedge \tilde{\eta} \rangle \\
	 & = 2\pi\omega_{([A])}(\xi, \eta) + \int_{\partial \nu T} \langle \tilde{\xi} \wedge \tilde{\eta} \rangle.
\end{align*}
The last term on the right is zero, since the forms $\xi$, $\eta$ are compactly supported in $H_\alpha \setminus \nu T$. Therefore $L_\alpha$ is isotropic in $\scrR_{g,3}$. It is also clearly half-dimensional, so that it is a Lagrangian.

For the general case where $\bfalpha$ is not necessarily standard, notice that up to isotopy, any set of attaching curves can be obtained from the standard one by a mapping class $\phi \in \MCG(\Sigma_{g,3})$. Since $\MCG(\Sigma_{g,3})$ acts on $\scrR_{g,3}$ by symplectomorphisms, it follows that $L_\alpha$ Lagrangian for any $g$-tuple of attaching curves $\bfalpha$.
\end{proof}

We can determine what manifold $L_\alpha$ is a Lagrangian embedding of:

\begin{thm}
For any $g$-tuple of attaching curves $\boldsymbol{\alpha}$, $L_\alpha \cong (S^3)^g$.
\label{thm:lagrangianS3}
\end{thm}

\begin{proof}
Again, we first consider the case where the $\alpha$-curves are standard. Then we may make the identification
\[
	L_\alpha = \{B_1, \dots, B_g, C_1, C_2, C_3 \in \SU(2): C_1C_2C_3 = 1, \Tr(C_k) = 0\}/\text{conjugation}.
\]

Note that if the orientations of the three punctures are chosen appropriately, then up to conjugation $(C_1, C_2, C_3) = (\mathbf{i}, \mathbf{j}, -\mathbf{k})$ is the only ordered triple of traceless $\SU(2)$-matrices such that $C_1 C_2 C_3 = 1$ (see Example \ref{exmp:Rgn}). We claim that the map
\[
	(S^3)^g \longrightarrow L_\alpha: (B_1, \dots, B_g) \mapsto [B_1, \dots, B_g, \mathbf{i}, \mathbf{j}, -\mathbf{k}]
\]
is a diffeomorphism. Injectivity is clear, since the common stabilizer of $\mathbf{i}$, $\mathbf{j}$, and $-\mathbf{k}$ is the center $\{\pm 1\}$ of $\SU(2)$. For surjectivity, let $[B_1, \dots, B_g, C_1, C_2, C_3] \in L_\alpha$ be arbitrary. There is some $D \in \SU(2)/\{\pm I\} = \SO(3)$ such that $D(C_1, C_2, C_3)D^{-1} = (\mathbf{i}, \mathbf{j}, -\mathbf{k})$, so that
\[
	[B_1, \dots, B_g, C_1, C_2, C_3] = [DB_1 D^{-1}, \dots, DB_g D^{-1}, \mathbf{i}, \mathbf{j}, -\mathbf{k}],
\]
and therefore $[B_1, \dots, B_g, C_1, C_2, C_3]$ lies in the image of our map. Furthermore, the choice of $D$ can be made so that it varies smoothly with $(C_1, C_2, C_3)$ (depending only on the angles between $C_1, C_2, C_3$ and $\mathbf{i}, \mathbf{j}, -\mathbf{k}$, respectively). It follows that $L_\alpha$ is diffeomorphic to $(S^3)^g$ when $\bfalpha$ is the standard set of $\alpha$-curves.

Up to isotopy, nonstandard $\alpha$-curves may be obtained as the image of the standard ones under a mapping class $\phi \in \MCG(\Sigma_{g,n})$, and mapping classes induce symplectomorphisms on $\scrR_{g,3}$. Hence it follows that $L_\alpha$ is diffeomorphic to $(S^3)^g$ for any $g$-tuple of attaching curves $\bfalpha$.
\end{proof}

The proof of Theorem \ref{thm:lagrangianS3} additionally allows us to explicitly identify the intersection $L_\alpha \cap L_\beta$.

\begin{thm}
\label{thm:lagrangianintersection}
For any two sets of attaching curves $\bfalpha$ and $\bfbeta$, we have that
\[
	L_\alpha \cap L_\beta \cong \Hom(\pi_1(Y), \SU(2)).
\]
\end{thm}

\begin{proof}
Let $H_\alpha$ and $H_\beta$ denote the $\alpha$- and $\beta$-handlebodies, respectively. Then the proof of Theorem \ref{thm:lagrangianS3} shows that
\[
	L_\alpha = \Hom(\pi_1(H_\alpha), \SU(2)), \quad\quad L_\beta = \Hom(\pi_1(H_\beta), \SU(2)),
\]
where we do not mod out by the action of conjugation. Furthermore, the proof of Theorem \ref{thm:lagrangianS3} also shows that
\begin{align}
	\label{eqn:zero-level}
	\mathscr{E} = \left\{[A_1, B_1, \dots, A_g, B_g, C_1, C_2, C_3] \in \scrR_{g,3} ~\left|~ \prod_{k = 1}^g [A_k, B_k] = I\right\}\right. & \cong \Hom(\pi_1(\Sigma_g), \SU(2)),
\end{align}
and $L_\alpha$, $L_\beta$ always lie entirely inside this subset of $\scrR_{g,3}$. Therefore by the Seifert-Van Kampen theorem an intersection point of $L_\alpha$ and $L_\beta$ corresponds to an element of $\Hom(\pi_1(Y), \SU(2))$.
\end{proof}

Because of Theorem \ref{thm:lagrangianintersection}, we see that a Hamiltonian perturbation will almost always be required to achieve transversality for the Lagrangians. In forthcoming work \cite{horton-holonomy}, we show that holonomy perturbations, familiar from gauge theory, can be used to induce Hamiltonian isotopies of the Lagrangians through perturbed character varieties, so that perturbations can preserve the interpretation of $L_\alpha \cap L_\beta$ as a space of (perturbed) $\SU(2)$-representations of $\pi_1(Y)$.

\hide{On the other hand, since $\Hom(\pi_1(Y), \SU(2))$ can be perturbed to a disjoint union of an isolated point, $S^2$'s, and $\R P^3$'s (via holonomy perturbations in $\Hom(\pi_1(Y), \SU(2))/\text{conj.}$), so it may be possible to consider a Morse-Bott approach in working with the Floer homology of $L_\alpha$ and $L_\beta$.}

There is an extension of the identification (\ref{eqn:zero-level}) which is useful. Write
\[
	R^\ast = \Hom(\pi_1(\Sigma_g \setminus \{\text{pt}\}), \SU(2)) \cong \SU(2)^{2g}.
\]
Certainly we have
\[
	E = \left\{[A_1, B_1, \dots, A_g, B_g, C_1, C_2, C_3] \in \scrR_{g,3} ~\left|~ \prod_{k = 1}^g [A_k, B_k] = I\right\}\right. \subset R^\ast.
\]
The desired strengthening of (\ref{eqn:zero-level}) is that the subspaces $\mathscr{E} \subset \scrR_{g,3}$ and $E \subset R^\ast$ are neighborhood equivalent:

\begin{thm}
\label{thm:zero-level-nbhd}
There is an open neighborhood $U$ of $E$ in $R^\ast$, an open neighborhood $\mathscr{U}$ of $\mathscr{E}$ in $\scrR_{g,3}$, and a diffeomorphism $h: U \longrightarrow \mathscr{U}$ such that $h|_E: E \longrightarrow \mathscr{E}$ is the identification (\ref{eqn:zero-level}).
\end{thm}

\begin{proof}
\hide{The $\SU(2)$-representation spaces $\Hom(\pi_1(H_\alpha),\SU(2)) \cong L_\alpha$ and $\Hom(\pi_1(H_\beta),\SU(2)) \cong L_\beta$ embed into $R^\ast = \Hom(\pi_1(\Sigma_g \setminus \{\text{pt}\}), \SU(2))$. We will therefore establish the desired result by proving that there is an open neighborhood of
\[
	\mathscr{E} = \left\{[A_1, B_1, \dots, A_g, B_g, C_1, C_2, C_3] \in \scrR_{g,3} : \prod_{k = 1}^g [A_k,B_k] = I = C_1 C_2 C_3\right\} \subset \scrR_{g,3},
\]
which contains all possible Lagrangians coming from handlebodies by the proof of Theorem \ref{thm:lagrangianintersection}, that is homeomorphic to an open neighborhood of
\[
	E = \left\{A_1, B_1, \dots, A_g, B_g \in \SU(2) : \prod_{k = 1}^g [A_k,B_k] = I\right\} \subset R^\ast
\]
Indeed, w}Write
\[
	\mathscr{U}_t = \left\{[A_1, B_1, \dots, A_g, B_g, C_1, C_2, C_3] \in \scrR_{g,3} : \prod_{k = 1}^g [A_k,B_k] = C_1 C_2 C_3, \Tr(C_1 C_2 C_3) = t \right\},
\]
and consider the open neighborhood
\[
	\mathscr{U} = \bigcup_{t \in (1,2]} \mathscr{U}_t
\]
of $\mathscr{E}$ in $\scrR_{g,3}$. Note that $\mathscr{U}_2 = \mathscr{E}$. We have analogous subsets
\[
	U_t = \left\{A_1, B_1, \dots, A_g, B_g \in \SU(2) : \Tr \left(\prod_{k = 1}^g [A_k,B_k]\right) = t\right\}
\]
of $R^\ast$ with $U_2 = E$. For each $t \in (1, 2]$, we define a map
\[
	h_t: U_t \longrightarrow \mathscr{U}_t,
\]
\[
	h_t(A_1, B_1, \dots, A_g, B_g) = \left[A_1, B_1, \dots, A_g, B_g, \bfi, e^{2\pi\gamma \bfk}\bfi, -e^{2\pi\gamma \bfk} \prod_{k = 1}^g [A_k, B_k]\right],
\]
where $\gamma = \gamma(A_1, \dots, B_g)$ is given by
\[
	\gamma(A_1, \dots, B_g) = \begin{cases} \displaystyle \frac{1}{2\pi} \arctan \left( \frac{\cot \tfrac{\pi}{4}t}{Q_{\bfk}}\right), & \displaystyle \text{if } \prod_{k = 1}^g [A_k, B_k] \neq I, \\ \dfrac{1}{2}, & \text{otherwise}; \end{cases}
\]
here $Q_\bfk$ denotes the $\bfk$-component of $\displaystyle \prod_{k = 1}^g [A_k, B_k]$. The induced map
\[
	h : \bigcup_{t \in (1,2]} U_t \longrightarrow \bigcup_{t \in (1,2]} \mathscr{U}_t
\]
is easily seen to be a diffeomorphism.\hide{ Since $h$ is a homeomorphism of open neighborhoods of $E$ and $\mathscr{E}$ identifying the respective embeddings of $L_\alpha$ and $L_\beta$, we conclude that that $(L_\alpha \cdot L_\beta)_{\scrR_{g,3}} = (L_\alpha \cdot L_\beta)_{R^\ast}$.}
\end{proof}

Since the trivial representation is always isolated and transverse in $\Hom(\pi_1(Y), \SU(2))$ when $Y$ is a rational homology sphere \cite[Proposition III.1.1(c)]{akbulut-mccarthy}, we immediately obtain the following:

\begin{cor}
The trivial representation $\theta = [I, I, \dots, I, I, \bfi, \bfj, -\bfk] \in \scrR_{g,3}$ is always a transverse intersection point of $L_\alpha$ and $L_\beta$, for any sets of attaching circles $\bfalpha$ and $\bfbeta$. When $Y$ is a rational homology sphere, $\theta$ is a transverse intersection point.
\label{cor:trivialint}
\end{cor}

\hide{\todo{Can holonomy perturbations be chosen in such a way that the trivial representation is a transverse perturbed intersection point for arbitrary $3$-manifolds?}}

\subsection{Whitney $n$-gons}

Since our goal is to study pseudoholomorphic disks with boundaries on Lagrangians of the form $L_\alpha$, we should start by understanding the topology of the relevant spaces of maps.

Suppose we have $n$ sets of attaching curves $\bfalpha_1, \dots, \bfalpha_n$ and furthermore suppose that the associated Lagrangians $L_{\alpha_1}, \dots, L_{\alpha_n}$ intersect transversely (perhaps after a Hamiltonian perturbation). Let $\D_n$ denote the closed unit disk in $\C$ with markings $z_k = \exp(2\pi k i/n)$ at the $n^\text{th}$ roots of unity. Starting from $1$ and moving clockwise, denote the connected components of $\D_n \setminus \{z_1, \dots, z_n\}$ by $A_k$, $k = 1, \dots, n$. Finally, suppose we have $n$ intersection points $y_k \in L_{\alpha_k} \cap L_{\alpha_{k+1}}$, where we have $L_{\alpha_{n+1}} = L_{\alpha_1}$. Then a {\bf Whitney $n$-gon} for $(y_1, \dots, y_n)$ is a continuous map $\phi: \D_n \longrightarrow \scrR_{g,3}$ such that
\[
	\phi(z_k) = y_k, \quad  \phi(A_k) \subset L_{\alpha_k} \text{ for each } k = 1, \dots, n.
\]
We will write $\Whit(y_1, \dots, y_n)$ for the space of all Whitney $n$-gons for $(y_1, \dots, y_n)$ and $\pi_2(y_1, \dots, y_n)$ for the set of all homotopy classes of Whitney $n$-gons for $(y_1, \dots, y_n)$.

It turns out that we can determine the homotopy classes of Whitney $n$-gons for any $n$-tuple of intersection points.

\begin{lem}
For any transverse $n$-tuple of Lagrangians $L_{\alpha_1}, \dots, L_{\alpha_n}$ in $\scrR_{g,3}$ associated to attaching sets $\bfalpha_1, \dots, \bfalpha_n$ and distinct intersection points $y_k \in L_{\alpha_k} \cap L_{\alpha_{k+1}}$, $\pi_2(y_1, \dots, y_n) \cong \Z^4$.
\label{lem:pi2xy}
\end{lem}

\begin{proof}
If we use the notation $\Omega_Y(a,b)$ for the space of all continuous paths in $Y$ from $a$ to $b$, then the ``boundary evaluation'' map
\[
	\Whit(y_1, \dots, y_n) \longrightarrow \Omega_{L_{\alpha_1}}(y_1, y_2) \times \Omega_{L_{\alpha_2}}(y_2, y_3) \times \cdots \times \Omega_{L_{\alpha_n}}(y_n,y_1)
\]
is a Serre fibration with homotopy fiber $\Map_\ast(S^2, \scrR_{g,3})$ (to identify the fiber this way, fix a reference $n$-gon $\phi_0$ with the correct boundary evaluation and use it to cap off any other $n$-gon with the same boundary data). Part of the associated long exact sequence in homotopy reads
\[
	\pi_1(\Omega_{L_{\alpha_1}}(y_1, y_2) \times \cdots \times \Omega_{L_{\alpha_n}}(y_n,y_1)) \longrightarrow \pi_0(\Map_\ast(S^2, \scrR_{g,3})) \longrightarrow \pi_0(\Whit(y_1, \dots, y_n)) 
\]
\[
	\longrightarrow \pi_0(\Omega_{L_{\alpha_1}}(y_1, y_2) \times \cdots \times \Omega_{L_{\alpha_n}}(y_n,y_1)).
\]
Since $L_{\alpha_1} \cong \cdots \cong L_{\alpha_n} \cong S^3$ is $2$-connected, the outer two homotopy groups vanish and hence
\[
	\pi_0(\Map_\ast(S^2,\scrR_{g,3})) \cong \pi_0(\Whit(y_1, \dots, y_n))
\]
\[
	\implies \pi_2(\scrR_{g,3}) \cong \pi_2(y_1, \dots, y_n).
\]
Since $\pi_2(\scrR_{g,3}) \cong \Z^4$ by Proposition \ref{homotopyModuli}, we conclude that
\[
	\pi_2(y_1, \dots, y_n) \cong \Z^4. \qedhere
\]
\end{proof}

Note that the identification $\pi_2(y_1, \dots, y_n) \cong \pi_2(\scrR_{g,3})$ is affine (\emph{i.e.} non-canonical), since it depends on a choice of reference $n$-gon.

It should be observed that in contrast to the case of Heegaard Floer homology, the homotopy classes of Whitney bigons $\pi_2(x,y)$ in our setting contain no information about the $3$-manifold the Lagrangians are coming from.

\subsection{Disk Invariants of Heegaard Lagrangians}

Because $\scrR_{g,3}$ has minimal Chern number $1$ and the Lagrangians $L_\alpha$, $L_\beta$ coming from a Heegaard diagram are monotone, $L_\alpha$ and $L_\beta$ have minimal Maslov number $2$. This makes it possible for the Floer differential on $\CF(L_\alpha, L_\beta)$ to not square to zero, due to contributions coming from Maslov index $2$ disk bubbles. In this section, we study these moduli spaces of Maslov index $2$ disks and prove the necessary results that imply they do not obstruct the Floer differential.

Given any symplectic manifold $(M, \omega)$, a relatively spin Lagrangian $L \subset M$, a point $x \in L$, and an $\omega$-compatible complex structure $J$, we may define a moduli space
\[
	\calN_J(L, x) = \left\{ u: [0,\infty) \times \bfi\R \longrightarrow M \left| \begin{array}{l} \displaystyle \lim_{z \to \infty} u(z) = x \\ u(\bfi t) \in L \\ u \text{ is $J$-holomorphic} \\ u \text{ has Maslov index $2$}\end{array}\right.\right\}.
\]
The half-plane $[0,\infty) \times \bfi \R$ is conformally equivalent to a disk with a single puncture on its boundary, so we may equivalently think of $\calN_J(L,x)$ as the space of Maslov index $2$ $J$-holomorphic disks in $M$ with boundary mapped to $L$ and a marked point on the boundary sent to $x \in L$.

There is a two-dimensional group of conformal automorphisms of the domain $[0,\infty) \times \bfi\R$, corresponding to imaginary translations and real dilations. Write $\bar{\calN}_J(L,x)$ for the quotient of $\calN_J(L,x)$ by reparametrization of the domain. These disk moduli spaces satisfy the following compactness and transversality properties (cf. \cite[Proposition 16.3.1 and Proposition 16.3.2]{oh-vol2}):

\hide{Careful about whether the homology class of the disk needs to be fixed}

\begin{prop}
\label{prop:disk-moduli}
There is a dense subset $\J_\omega(L,x)$ of the space of $\omega$-compatible almost complex structures on $M$ such that the following hold:
\begin{itemize}
	\item[(1)] $\bar{\calN}_J(L,x)$ is a smooth, compact, oriented manifold of dimension $0$ for any $J \in \J_\omega(L,x)$.
	\item[(2)] If $\calN_J^{<2}(L,x)$ denotes the space of $J$-holomorphic disks satisfying the same conditions for disks in $\calN_J(L,x)$, except that the Maslov index is less than $2$ instead of exactly $2$, then $\calN_J^{<2}(L,x) = \varnothing$.
	\item[(3)] For any $J_1, J_2 \in \J_\omega(L,x)$, $\bar{\calN}_{J_1}(L,x)$ and $\bar{\calN}_{J_2}(L,x)$ are oriented cobordant.
	\item[(4)] For any $x_1, x_2 \in L$ and $J \in \J_\omega(L, x_1) \cap \J_\omega(L,x_2)$, $\bar{\calN}_J(L,x_1)$ and $\bar{\calN}_J(L,x_2)$ are oriented cobordant.
\end{itemize}
\end{prop}

Now, given $J \in \J_\omega(L,x)$, define the {\bf disk invariant} of $L$ to be the signed count of points
\[
	\Phi_J(L,x) = \#\bar{\calN}_J(L,x).
\]
Proposition \ref{prop:disk-moduli} implies that this is an invariant of the Lagrangian $L$, since for any points $x_1, x_2 \in L$ and almost complex structures $J_1 \in \J_\omega(L,x_1), J_2 \in \J_\omega(L,x_2)$, we have
\[
	\Phi_{J_1}(L,x_1) = \Phi_{J_2}(L,x_2).
\]
For this reason, we will tend to not explicitly refer to the basepoint and write $\Phi_J(L)$ instead of $\Phi_J(L,x)$. Since it is useful to keep track of the almost complex structure when considering Floer homology, we will not eliminate $J$ from the notation.

We now explain why the disk invariants of the Lagrangians coming from a Heegaard diagram are the same for all sets of attaching curves.

\begin{prop}
\label{prop:disk-inv-equal}
For any sets of attaching curves $\bfalpha$ and $\bfbeta$ in a surface of genus $g$, $\Phi_J(L_\alpha) = \Phi_J(L_\beta)$.
\end{prop}

\begin{proof}
As noted in the proof of Theorem \ref{thm:lagrangianS3}, for any pair of attaching curves $\bfalpha$ and $\bfbeta$, there exists a symplectomorphism $\phi: \scrR_{g,3} \longrightarrow \scrR_{g,3}$ (induced by an automorphism of the surface taking the $\alpha$-curves to the $\beta$-curves) such that $\phi(L_\alpha) = L_\beta$. Since the disk invariant of a Lagrangian is unchanged under symplectomorphisms, we conclude that $\Phi_J(L_\alpha) = \Phi_J(L_\beta)$.
\end{proof}

\hide{By \cite[Theorem 6.4]{auroux-anticanonical}, $\mathfrak{m}_0 = \Phi_J(L,x)e_L$.}

%% file: SI-heegaard-splitting.tex
\section{Symplectic Instanton Homology via Heegaard Diagrams}

\subsection{The Definition}

Let $\calH = (\Sigma_g, \bfalpha, \bfbeta, \bfx)$ be a pointed Heegaard diagram. We have now shown how to associate a pair of Lagrangian submanifolds $L_\alpha$, $L_\beta$ in $\scrR_{g,3}$ to this data. The next step is to consider their Lagrangian Floer homology.

Fix orientations of $L_\alpha$ and $L_\beta$ \hide{(can we do this canonically at trivial rep?)} and an almost complex structure $J$ on $\scrR_{g,3}$ compatible with the symplectic form $\omega_{g,3}$. After applying a Hamiltonian isotopy, we may assume that $L_\alpha$ and $L_\beta$ intersect transversely, so that $L_\alpha \cap L_\beta$ is a finite set of points. Define the {\bf symplectic instanton chain complex} $(\CSI(\Sigma_g, \bfalpha, \bfbeta, \bfx), \partial_J)$ by
\[
	\CSI(\calH) = \bigoplus_{x \in L_\alpha \cap L_\beta} \Z \langle x\rangle,
\]
\[
	\partial_J x = \sum_{\sigma \in L_\alpha \cap L_\beta} \# \bar{\calM}(x,y) y,
\]
where $\# \bar{\calM}(x,y)$ denotes the signed count of Maslov index $1$ $J$-holomorphic strips (modulo translation) from $x$ to $y$ with boundary on $L_\alpha \cup L_\beta$. Note that since $L_\alpha$ and $L_\beta$ are $2$-connected, their chosen orientations induce unique relative spin structures (in the sense of \cite{FOOO}) which orient the moduli spaces $\bar{\calM}(x,y)$, so that the signed count of points makes sense. Of course, one could also work with $\F_2$ coefficients and ignore orientations.

\begin{thm}
For a generic almost complex structure $J$, $\partial_J^2 = 0$.
\end{thm}

\begin{proof}
As usual, the proof involves a study of the Gromov compactification of the moduli space $\bar{\calM}_2(x, y)$ of holomorphic representatives $u: \D \longrightarrow \scrR_{g,3}$ of Maslov index $2$ Whitney disks $\phi \in \pi_2(x,y)$ modulo translation. The ends of $\bar{\calM}_2(x,y)$ correspond to one of the following types of degenerations of $u$:\footnote{In principle, combinations of these degenerations may occur, but they are easily ruled out for Maslov index reasons.}
\begin{itemize}
	\item[(1)] (Strip breaking) $u = u_1 \cup u_2$, where $u_1 \in \bar{\calM}(x,z)$ and $u_2 \in \bar{\calM}(z,y)$ for some $z \in L_\alpha \cap L_\beta$.
	\item[(2)] (Sphere bubbling) $u = u' \cup v$, where $u' \in \bar{\calM}(x,y)$ and $v: S^2 \longrightarrow \scrR_{g,3}$ is a holomorphic sphere such that the images of $u'$ and $v$ meet at exactly one point $z$ of the interior of $u'(\D)$.
	\item[(3)] (Disk bubbling) $u = u_1 \cup u_2$, where $u_1 \in \bar{\calM}(x,y)$ and $u_2 \in \bar{\calN}(L_\alpha, z)$ or $u_2 \in \bar{\calN}(L_\beta, z)$ for some $z \in u_2(\partial \D)$.
\end{itemize}

First, we rule out the possibility of sphere bubbles in the Gromov compactification of $\bar{\calM}_2(x,y)$. Since $\scrR_{g,3}$, $L_\alpha$, and $L_\beta$ are monotone and the minimal Chern number of $\scrR_{g,3}$ is $1$ (Corollary \ref{cor:minchern}), any sphere bubble $v: S^2 \longrightarrow \scrR_{g,3}$ has Maslov index at least $2$. Since $u = u' \cup v \in \bar{\calM}_2(x,y)$, transversality implies that $v$ has Maslov index $2$ and hence $u'$ has Maslov index zero, which in turn implies $x = y$. Then $v$ is a holomorphic sphere passing through $x \in L_\alpha \cap L_\beta$, which is impossible by our definition of genericity of $J$. Therefore sphere bubbling does not occur in the Gromov compactification of $\bar{\calM}_2(x,y)$.

As a second step, we describe the possible disk bubbles. By Proposition \ref{prop:disk-moduli}(2), $u_2 \in \bar{\calN}(L_\bullet, z)$ ($\bullet = \alpha \text{ or } \beta$) must have Maslov index at least $2$. Therefore $u_1$ must have Maslov index zero, implying $x = y$ and $u_1$ is the constant map. This further implies that $z = x$, so that disk bubbles appear only in the Gromov compactification of $\bar{\calM}_2(x,x)$, $x \in L_\alpha \cap L_\beta$.

It now follows that
\[
	\partial_J^2 x = \#\bar{\calN}(L_\alpha, x) - \#\bar{\calN}(L_\beta, x) + \sum_{\substack{y \in L_\alpha \cap L_\beta \\ \phi \in \pi_2(x,y)}} \hspace{.2cm}\sum_{\substack{\psi_1 \in \pi_2(x,z), \psi_2 \in \pi_2(z,y) \\ \text{such that }\psi_1 \ast \psi_2 = \phi}} \left(\#\bar{\calM}(\psi_1) \cdot \#\bar{\calM}(\psi_2)\right)y.
\]
Proposition \ref{prop:disk-inv-equal} implies that $\#\bar{\calN}(L_\alpha, x) - \#\bar{\calN}(L_\beta,x) = 0$, and the usual method for analyzing strip breaking implies that the last term is also zero. Therefore $\partial_J^2 = 0$ for generic $J$.
\end{proof}

It follows that $(\CSI(\calH), \partial_J)$ is a chain complex and the {\bf symplectic instanton homology}
\[
	\SI(\calH) = H_\ast(\CSI(\calH), \partial_J)
\]
is a well-defined invariant of the pointed Heegaard diagram $\calH = (\Sigma_g, \bfalpha, \bfbeta, \bfx)$ (\emph{i.e.} independent of the Hamiltonian isotopies used to achieve transversality of the Lagrangians and the choice of almost complex structure $J$, by the usual continuation map arguments in Floer theory).

\subsection{Gradings}

\label{sect:gradings}

The Lagrangian Floer chain group $\CF(L_\alpha, L_\beta)$ always admits an absolute $\Z/2$-grading as follows. We assume that $L_\alpha$ intersects $L_\beta$ transversely (so that perhaps a Hamiltonian perturbation has been applied). First, choose orientations on $L_\alpha$ and $L_\beta$ (this has already been done if we are using Floer homology with $\Z$-coefficients); $\scrR_{g,3}$ has an orientation determined by the symplectic form. Any given generator $y \in \CF(L_\alpha, L_\beta)$ corresponds to a transverse intersection point of $L_\alpha$ and $L_\beta$ since we assume $L_\alpha \pitchfork L_\beta$. We define the $\Z/2$-grading of the generator $y$ by
\[
	\gr(y) = \begin{cases} 0 & \text{ if $y$ is a positive intersection point of $L_\alpha$ with $L_\beta$,} \\ 1 & \text{ if $y$ is a negative intersection point of $L_\alpha$ with $L_\beta$.} \end{cases}
\]
It is clear that the Floer differential is of degree $1$ with respect to this grading, so that this grading descends to the Lagrangian Floer homology group $\SI(\calH) = \HF(L_\alpha, L_\beta)$.

%% file: naturality.tex
\section{Invariance and Naturality with Respect to Heegaard Diagrams}

\label{sect:invariance}

Now we turn to an investigation of the invariance and naturality of symplectic instanton homology as an invariant of Heegaard diagrams. Recall from Appendix \ref{sect:heegaard-diagrams} that all Heegaard diagrams of a fixed $3$-manifold are related by a sequence of isotopies of attaching curves, handleslides, and (de)stabilizations. Hence to prove invariance we just need to check that these moves induce isomorphisms on symplectic instanton homology.



\subsection{Isotopies of Attaching Curves}

Since the Lagrangians $L_\alpha$, $L_\beta$ are defined in terms of $\SU(2)$-representations of $\pi_1(\Sigma_g \setminus \mathbf{x})$ sending the $\alpha$- or $\beta$-curves to $I$, they are unchanged under isotopies of the $\alpha$- and $\beta$-curves. Therefore:

\begin{thm}
Let $(\Sigma_g, \bfalpha, \bfbeta, \bfx)$ and $(\Sigma_g, \bfalpha', \bfbeta', \bfx)$ be two Heegaard diagrams such that $\alpha_k$ is isotopic to $\alpha_k^\prime$ and $\beta_k$ is isotopic to $\beta_k^\prime$ for $k = 1, \dots, g$. Then
\[
	\CSI(\Sigma_g, \bfalpha, \bfbeta, \bfx) = \CSI(\Sigma_g, \bfalpha', \bfbeta', \bfx).
\]
\end{thm}

\subsection{Handleslides}

It also is easy to see that handleslides do not affect the symplectic instanton chain complex at all.

\begin{thm}
Let $(\Sigma_g, \bfalpha, \bfbeta, \bfx)$ and $(\Sigma_g, \bfalpha', \bfbeta', \bfx)$ be two Heegaard diagrams such that $\alpha_k^\prime$ is obtained from $\alpha_k$ via handleslides with the $\alpha$-curves and $\beta_k^\prime$ is obtained from $\beta_k$ via handleslides with the $\beta$-curves for $k = 1, \dots, g$. Then
\[
	\CSI(\Sigma_g, \bfalpha, \bfbeta, \bfx) = \CSI(\Sigma_g, \bfalpha', \bfbeta', \bfx).
\]
\end{thm}

\begin{proof}
It suffices to consider the case of a single handleslide between two $\alpha$-curves, say a handleslide of $\alpha_1$ over $\alpha_2$. Hence we get two sets of attaching curves
\[
	\bfalpha = (\alpha_1, \alpha_2, \dots, \alpha_g), \quad\quad \bfalpha' = (\alpha_1^\prime, \alpha_2, \dots, \alpha_g),
\]
where $\alpha_1^\prime$ is the result of sliding $\alpha_1$ over $\alpha_2$. We claim that the Lagrangians $L_\alpha$ and $L_{\alpha^\prime}$ are identical as subsets of $\scrR_{g,3}$. To see this, we use the gauge theoretic description of the Lagrangians. Let $[A]$ be any gauge equivalence class of connections in $L_{\alpha^\prime}$, and pick a specific representative $A$ for this class. By the definition of a handle slide, there is a smooth embedding $\phi: P \longrightarrow \Sigma_g$ of the pair of pants $P$ such that the boundary of the image $\phi(P)$ is $\alpha_1 \amalg \alpha_2 \amalg \alpha_1^\prime$.  Since $[A] \in L_{\alpha^\prime}$, the pullback connection $\phi^\ast A$ on $P$ has trivial holonomy around the two boundary components of $P$ mapping to $\alpha_1^\prime$ and $\alpha_2$ under $\phi$. Therefore $\phi^\ast A$ extends to a connection on $D^2$, and is necessarily trivial, so that we may conclude $\Hol_A(\alpha_1) = I$. Therefore if $[A] \in L_{\alpha^\prime}$, we also have $[A] \in L_{\alpha}$.

A similar argument shows that if $[A] \in L_\alpha$, then $[A] \in L_{\alpha^\prime}$. Therefore handleslides do not change the Lagrangians and the result follows.
\end{proof}

\subsection{Stabilization}

In contrast to the other Heegaard moves, invariance under stabilization is not trivial to prove. We will use the quilted Floer theory of Wehrheim and Woodward \cite{quilts} (reviewed in Appendix \ref{sect:quilts}) to give a fairly simple proof of stabilization invariance.

Given a Heegaard diagram $(\Sigma_g, \bfalpha, \bfbeta, \bfx)$, write $(\Sigma_{g+1}, \bfalpha', \bfbeta', \bfx)$ for its stabilization and let
\begin{align}
	L_{\alpha\alpha^\prime} = \left\{[\rho_1] \in \scrR_{g,3}, [\rho_2] \in \scrR_{g+1,3} ~\left|~ \begin{array}{l l}\rho_1(\alpha_k) = \rho_2(\alpha_k), & k = 1, \dots, g \\ \rho_1(\beta_k) = \rho_2(\beta_k), & k = 1, \dots, g \\ \rho_1(\mu_k) = \rho_2(\mu_k), & k = 1, 2, 3 \\ \rho_2(\alpha_{g+1}) = I & \end{array}\right\}\right.,
	\label{eqn:Laa}
\end{align}
\begin{equation}
	L_{\beta^\prime\beta} = \left\{[\rho_1] \in \scrR_{g+1,3}, [\rho_2] \in \scrR_{g,3} ~\left|~ \begin{array}{l l}\rho_1(\alpha_k) = \rho_2(\alpha_k), & k = 1, \dots, g \\ \rho_1(\beta_k) = \rho_2(\beta_k), & k = 1, \dots, g \\ \rho_1(\mu_k) = \rho_2(\mu_k), & k = 1, 2, 3 \\ \rho_1(\beta_{g+1}) = I & \end{array}\right\}\right..
	\label{eqn:Lbb}
\end{equation}

One may easily check the following:

\begin{lem}
$L_{\alpha\alpha^\prime}$ is a Lagrangian correspondence from $\scrR_{g,3}$ to $\scrR_{g+1,3}$ and $L_{\beta^\prime\beta}$ is a Lagrangian correspondence from $\scrR_{g+1,3}$ to $\scrR_{g,3}$. Furthermore, the geometric compositions $L_\alpha \circ L_{\alpha\alpha^\prime}$, $L_{\beta^\prime\beta} \circ L_\beta$, and $L_{\alpha\alpha^\prime} \circ L_{\beta^\prime\beta}$ are all embedded and are respectively equal to $L_{\alpha^\prime}$, $L_{\beta^\prime}$, and $\Delta_{\scrR_{g,3}}$, where $\Delta_{\scrR_{g,3}}$ denotes the diagonal in $\scrR_{g,3} \times \scrR_{g,3}$.
\end{lem}

Now a basic series of manipulations and the fact that embedded geometric composition leaves Floer homology invariant up to canonical isomorphism finishes the job:

\begin{thm}
\label{thm:stabilization-inv}
Let $(\Sigma_g, \bfalpha, \bfbeta, \bfx)$ be a Heegaard diagram and $(\Sigma_{g+1}, \bfalpha', \bfbeta', \bfx)$ be its stabilization. Then there is a canonical isomorphism
\[
	\SI(\Sigma_g, \bfalpha, \bfbeta, \bfx) \cong \SI(\Sigma_{g+1}, \bfalpha', \bfbeta', \bfx).
\]
\end{thm}

\begin{proof}
We simply compute that
\begin{align*}
	\HF(L_\alpha, L_\beta) & = \HF(L_\alpha, \Delta_{\scrR_{g,3}}, L_\beta) \\
	 & = \HF(L_\alpha, L_{\alpha\alpha^\prime} \circ L_{\beta^\prime\beta}, L_\beta) \\
	 & \cong \HF(L_\alpha, L_{\alpha\alpha^\prime}, L_{\beta^\prime\beta}, L_\beta) \\
	 & \cong \HF(L_\alpha \circ L_{\alpha\alpha^\prime}, L_{\beta^\prime\beta} \circ L_\beta) \\
	 & = \HF(L_{\alpha^\prime}, L_{\beta^\prime}). \qedhere
\end{align*}
\end{proof}

Since all possible Heegaard diagrams for a fixed $3$-manifold $Y$ are related by a finite sequence of isotopies, handleslides, and stabilizations, we see that we have obtained a proof that $\SI(\calH)$ is a topological invariant of $Y$ (up to isomorphism).

\subsection{Naturality}

By showing the invariance of $\SI(\mathcal{H})$ under pointed Heegaard moves on $\mathcal{H}$, we established that it gives a well-defined isomorphism class of a group $\SI(Y)$ associated to the $3$-manifold $Y$ represented by $\mathcal{H}$. In order to define maps induced by cobordisms (as we do in \cite{horton2}), it is necessary to have a well-defined \emph{group} $\SI(Y)$, not just an isomorphism class of groups. To pin down a specific group $\SI(Y)$, we need to consider \emph{loops} of Heegaard diagrams $\mathcal{H}_t$ and ensure that $\SI(\mathcal{H}_t)$ has no monodromy. This is potentially a very complicated thing to check since the fundamental group of the space of Heegaard splittings of a $3$-manifold is highly nontrivial.

Fortunately, Juh\'asz and Thurston \cite[Definition 2.33]{HFnaturality} have determined a sufficient set of four conditions one must check to ensure that an algebraic invariant of pointed Heegaard diagrams gives a well-defined group. The first three conditions are trivial to check for symplectic instanton homology, so to shorten the exposition we omit a discussion of them here. The fourth condition is invariance under simple handleswaps, illustrated in Figure \ref{fig:handleswap}. The move is depicted locally, so that each picture is in fact representing a different Heegaard diagram $\calH_k$ ($k = 1, 2, 3$), each of equal genus, where the diagrams differ only in a subsurface diffeomorphic to $\Sigma_2 \setminus D^2$ in the way described by the local pictures.

\begin{figure}[h]
	\centering
	\includegraphics[scale=1]{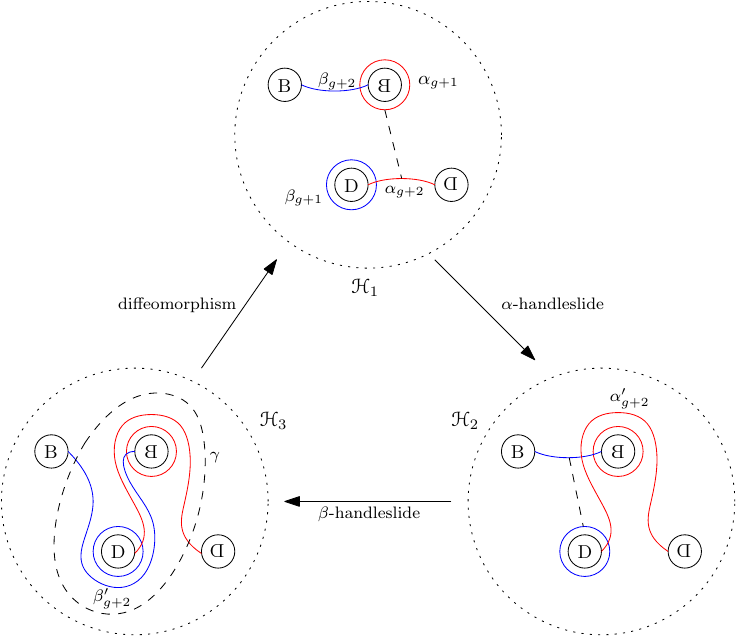}
	\caption{The simple handleswap move.}
	\label{fig:handleswap}
\end{figure}

In the simple handleswap move, the maps $\mathcal{H}_1 \longrightarrow \mathcal{H}_2$ and $\mathcal{H}_2 \longrightarrow \mathcal{H}_3$ are handleslides along the dotted arcs, while the third map $\mathcal{H}_3 \longrightarrow \mathcal{H}_1$ is the composition of Dehn twists, $\tau_\gamma^{-1} \circ \tau_{\beta_{g+1}} \circ \tau_{\alpha_{g+1}}$, where $\gamma$ denotes the large dotted curve in $\mathcal{H}_3$.

\begin{prop}
The map $\SI(\calH_1) \longrightarrow \SI(\calH_1)$ induced by a simple handleswap is the identity.
\label{prop:handleswap}
\end{prop}

\begin{proof}
Since handleslides have already been found to induce to identity on symplectic instanton homology, we just need to check that the diffeomorphism $\tau_\gamma^{-1} \circ \tau_{\beta_{g+1}} \circ \tau_{\alpha_{g+1}}$ induces the identity. As such, we only need to focus on $\CSI(\calH_1)$ and $\CSI(\calH_3)$. Let $Y$ be the $3$-manifold represented by the Heegaard diagram $\bar{\calH}$ obtained by taking any of the $\calH_k$, removing the piece pictured in Figure \ref{fig:handleswap}, and replacing it with a disk (with no attaching curves). Then each $\calH_k$ represents $Y \# S^3 \# S^3$, where the $S^3 \# S^3$ connect summand is built in a different way for each $\calH_k$.

Note, however, that before applying a Hamiltonian perturbation, $\CSI(\calH_1)$ and $\CSI(\calH_3)$ have the same generators. To see this, we proceed similarly to the proof of stabilization invariance. Let $\calH_1$ and $\calH_3$ have genus $g+2$, with the attaching curves enumerated such that the first $g$ $\alpha$- and $\beta$-curves are the ones that are outside the handleswap region (and hence correspond the attaching curves in the genus $g$ diagram $\bar{\calH}$ for $Y$). Consider Lagrangian correspondences
\begin{align}
	L_{\alpha\alpha^\prime,1} = \left\{[\rho_1] \in \scrR_{g,3}, [\rho_2] \in \scrR_{g+2,3} ~\left|~ \begin{array}{l l}\rho_1(\alpha_k) = \rho_2(\alpha_k), & k = 1, \dots, g \\ \rho_1(\beta_k) = \rho_2(\beta_k), & k = 1, \dots, g \\ \rho_1(\mu_k) = \rho_2(\mu_k), & k = 1, 2, 3 \\ \rho_2(\alpha_{g+1}) = \rho_2(\alpha_{g+2}) = I & \end{array}\right\}\right.,
\end{align}
\begin{equation}
	L_{\beta^\prime\beta,1} = \left\{[\rho_1] \in \scrR_{g+2,3}, [\rho_2] \in \scrR_{g,3} ~\left|~ \begin{array}{l l}\rho_1(\alpha_k) = \rho_2(\alpha_k), & k = 1, \dots, g, \\ \rho_1(\beta_k) = \rho_2(\beta_k), & k = 1, \dots, g, \\ \rho_1(\mu_k) = \rho_2(\mu_k), & k = 1, 2, 3 \\ \rho_1(\beta_{g+1}) = \rho_1(\beta_{g+2}) = I & \end{array}\right\}\right.,
\end{equation}
\begin{align}
	L_{\alpha\alpha^\prime,3} = \left\{[\rho_1] \in \scrR_{g,3}, [\rho_2] \in \scrR_{g+2,3} ~\left|~ \begin{array}{l l}\rho_1(\alpha_k) = \rho_2(\alpha_k), & k = 1, \dots, g \\ \rho_1(\beta_k) = \rho_2(\beta_k), & k = 1, \dots, g \\ \rho_1(\mu_k) = \rho_2(\mu_k), & k = 1, 2, 3 \\ \rho_2(\alpha_{g+1}) = \rho_2(\alpha_{g+2}^\prime) = I & \\ \end{array}\right\}\right.,
\end{align}
\begin{equation}
	L_{\beta^\prime\beta,3} = \left\{[\rho_1] \in \scrR_{g+2,3}, [\rho_2] \in \scrR_{g,3} ~\left|~ \begin{array}{l l}\rho_1(\alpha_k) = \rho_2(\alpha_k), & k = 1, \dots, g \\ \rho_1(\beta_k) = \rho_2(\beta_k), & k = 1, \dots, g \\ \rho_1(\mu_k) = \rho_2(\mu_k), & k = 1, 2, 3 \\ \rho_1(\beta_{g+1}) = \rho_1(\beta_{g+2}^\prime) = I & \end{array}\right\}\right..
\end{equation}
Write $L_\alpha$ and $L_\beta$ for the Lagrangians in $\scrR_{g,3}$ coming from the diagram $\bar{\calH}$, and $L_{\alpha,k}$, $L_{\beta,k}$ for the Lagrangians in $\scrR_{g+2,3}$ coming from the diagram $\calH_k$ ($k = 1, 3$). Then, similarly to the situation for stabilizations, one has
\[
	L_{\alpha,k} = L_\alpha \circ L_{\alpha\alpha^\prime,k}, \quad\quad L_{\beta,k} = L_{\beta^\prime\beta,k} \circ L_\beta. \quad\quad\quad (k = 1, 3)
\]
Furthermore, $L_{\alpha\alpha^\prime,k} \circ L_{\beta^\prime\beta,k} = \Delta_{\scrR_{g,3}}$, and hence there are isomorphisms
\[
	\CF(L_\alpha, L_\beta) \cong \CF(L_\alpha, L_{\alpha\alpha^\prime,k}, L_{\beta^\prime\beta,k}, L_\beta) \cong \CF(L_{\alpha,k}, L_{\beta,k})
\]
just as in the proof of stabilization invariance. If $H: \scrR_{g,3} \longrightarrow \R$ is any Hamiltonian perturbation achieving transversality for $L_\alpha$ and $L_\beta$ in $\scrR_{g,3}$, then $H \times 0: \scrR_{g,3} \times \scrR_{g+2,3} \longrightarrow \R$ achieves transversality for $L_\alpha \times L_{\beta^\prime \beta,k}$ and $(L_{\alpha\alpha^\prime,k} \times L_\beta)^T$ in $\scrR_{g,3} \times \scrR_{g+2,3}$. In particular, any perturbed intersection point of $L_\alpha \times L_{\beta^\prime \beta,k}$ and $(L_{\alpha\alpha^\prime,k} \times L_\beta)^T$ has the form
\[
	[A_1, B_1, \dots, A_g, B_g, I, I, I, I, \bfi, \bfj, -\bfk]
\]	
for some $H$-perturbed intersection point $[A_1, B_1, \dots, A_g, B_g, \bfi, \bfj, -\bfk]$ of $L_\alpha$ and $L_\beta$. Now the fact that $\tau_\gamma^{-1} \circ \tau_{\beta_{g+1}} \circ \tau_{\alpha_{g+1}}$ induces the identity map on Floer chain groups is easy to see: each of these Dehn twists induces a fibered Dehn twist on $\scrR_{g+2,3}$ along a fibered coisotropic that misses all the perturbed intersection points (see Appendix \ref{sect:dehntwists}). It follows that the induced map on homology $\SI(\calH_3) \longrightarrow \SI(\calH_1)$ is the identity, so that the simple handleswap move induces the identity map $\SI(\calH_1) \longrightarrow \SI(\calH_1)$. \qedhere


\end{proof}

Since the conditions required by Juh\'asz and Thurston are satisfied by $\SI(\calH)$, it follows that symplectic instanton homology defines a \emph{natural} invariant of Heegaard diagrams of $3$-manifolds.

\begin{cor}
To any oriented $3$-manifold $Y$, symplectic instanton homology assigns a well-defined, specific group $\SI(Y)$ (not merely an isomorphism class of groups).
\label{cor:naturality}
\end{cor}

%% file: SI-cerf.tex
\section{Symplectic Instanton Homology via Cerf Decompositions}

\label{chap:SI-cerf}

In Sections 3-5, we defined our symplectic instanton homology entirely in terms of a Heegaard splitting and proved it is an invariant of $3$-manifolds. However, working only with Heegaard splittings makes proving certain properties (the K\"unneth principle, well-definedness of cobordism maps, the surgery exact trangle) more difficult than necessary. For this reason, we will extend the definition of symplectic instanton homology to more general handlebody decompositions called Cerf decompositions, using the ``Floer field theory'' approach of Wehrheim and Woodward \cite{floerField, floerFieldCoprime, floerFieldPhil}.

\subsection{Cerf Decompositions}

Suppose $X_-$ and $X_+$ are two closed, oriented $n$-dimensional smooth manifolds. A {\bf bordism} from $X_-$ to $X_+$ is a pair $(Y,\phi)$ consisting of a compact, oriented $(n+1)$-dimensional smooth manifold $Y$ along with an orientation-preserving diffeomorphism $\phi: \partial Y \longrightarrow \overline{X}_- \amalg X_+$. We say that two bordisms $(Y, \phi)$ and $(Y', \phi')$ from $X_-$ to $X_+$ are {\bf equivalent} if there exists an orientation-preserving diffeomorphism $\psi: Y \longrightarrow Y'$ such that $\phi \circ \psi|_{\partial Y} = \phi$. We may form the {\bf connected bordism category} $\Bord^0_{n+1}$ whose objects are closed, oriented, connected smooth $n$-manifolds and whose morphisms are equivalence classes of $(n+1)$-dimensional compact, oriented, connected bordisms.

In order to break up bordisms into basic pieces, we bring Morse theory into the picture. A {\bf Morse datum} for the bordism $(Y, \phi)$ is  pair $(f, \underline{t})$ consisting of a Morse function $f: Y \longrightarrow \R$ and a strictly increasing list of real numbers $\underline{t} = (t_0, \dots, t_m)$ satisfying the following conditions:
\begin{itemize}
	\item[(i)] $\min f(y) = t_0$ and $f^{-1}(t_0) = \phi^{-1}(X_-)$ (\emph{i.e.} the minimum of $f$ is $t_0$ and this minimum is attained at all points of the incoming boundary of $Y$ and nowhere else), and similarly the $\max f(y) = t_m$ and $f^{-1}(t_m) = \phi^{-1}(X_+)$.
	\item[(ii)] $f^{-1}(t)$ is connected for all $t \in \R$.
	\item[(iii)] Critical points and critical values of $f$ are in one-to-one correspondence, \emph{i.e.} $f: \Crit(f) \longrightarrow f(\Crit(f))$ is a bijection.
	\item[(iv)] $t_0, \dots, t_m$ are all regular values of $f$ and each interval $(t_{k-1}, t_k)$ contains at most one critical point of $f$.
\end{itemize}

A bordism $(Y, \phi)$ is an {\bf elementary bordism} if it admits a Morse datum with at most one critical point. If $(Y,\phi)$ admits a Morse datum with no critical points, then it is called a {\bf cylindrical bordism}. Due to the correspondence between critical points of Morse functions and handle attachments, we see that an elementary bordism is a bordism arising from the attachment of at most one handle to $\partial Y_- \times [0,1]$.

We will give a special name to decompositions of $(Y,\phi)$ into elementary bordisms. A {\bf Cerf decomposition} of $(Y, \phi)$ is a decomposition
\[
	Y = Y_1 \cup_{X_1} Y_2 \cup_{X_2} \cdots \cup_{X_{m-1}} Y_m
\]
of $Y$ into a sequence of elementary bordisms embedded in $Y$ such that
\begin{itemize}
	\item[(i)] Each $X_k$ is connected and nonempty, and the $X_k$ are pairwise disjoint.
	\item[(ii)] The interiors of the $Y_k$ are disjoint, and $Y_k \cap \partial Y = \varnothing$ if $k \neq 1, m$.
	\item[(iii)] $Y_1 \cap \partial Y = \phi^{-1}(X_-)$, $Y_m = \phi^{-1}(X_+)$, and $Y_k \cap Y_{k+1} = X_k$.
\end{itemize}
Certainly any Morse datum induces a Cerf decomposition, and given a Cerf decomposition a compatible Morse function can be constructed.

We say that two Cerf decompositions
\[
	Y = Y_1 \cup_{X_1} Y_2 \cup_{X_2} \cdots \cup_{X_{m-1}} Y_m,
\]
\[
	Y = Y_1^\prime \cup_{X_1^\prime} Y_2^\prime \cup_{X_2^\prime} \cdots \cup_{X_{m-1}^\prime} Y_{m}^\prime
\]
are {\bf equivalent} if there are orientation-preserving diffeomorphisms $\psi_k: Y_k \longrightarrow Y_k^\prime$ ($k = 1, \dots, m$) such that $\psi_1$ and $\psi_m$ restrict to the identity on $\partial Y$. Note that this definition requires the number of elementary pieces in each decomposition to be the same, but we can always increase $m$ by inserting cylindrical (\emph{i.e.} trivial) bordisms anywhere in the decomposition.

In applications, we will really only care about bordisms up to equivalence. Accordingly, we should define a Cerf decomposition of an equivalence class $[(Y,\phi)]$.\hide{ Call an equivalence class $[(Y,\phi)]$ {\bf elementary} (respectively {\bf cylindrical}) if one (and hence all) of its representatives is an elementary (respectively cylindrical) bordism. Then a} A {\bf Cerf decomposition} of an equivalence class $[(Y,\phi)] \in \Hom_{\Bord_{n+1}^0}(X_-, X_+)$ is a factorization
\[
	[(Y,\phi)] = [(Y_1,\phi_1)] \circ \cdots \circ [(Y_m,\phi_m)]
\]
where each $[(Y_k, \phi_k)] \in \Hom_{\Bord_{n+1}^0}(X_k, X_{k+1})$ is elementary. Note that the $X_k$'s do not appear in the notation for the Cerf decomposition; they are usually understood by context.

Given two Cerf decompositions
\[
	[(Y,\phi)] = [(Y_1,\phi_1)] \circ \cdots \circ [(Y_m,\phi_m)],
\]
\[
	[(Y,\phi)] = [(Y_1^\prime,\phi_1^\prime)] \circ \cdots \circ [(Y_m^\prime,\phi_m^\prime)]
\]
of an equivalence class $[(Y,\phi)]$, we say they are {\bf equivalent} if there exist orientation-preserving diffeomorphisms $\psi_k: X_k \longrightarrow X_k^\prime$ ($k = 0, \dots, m$, where $X_0 = X_-$ and $X_m = X_+$) such that $\psi_0 = \Id_{X_-}$, $\psi_m = \Id_{X_+}$, and
\[
	[(Y_k, \phi_k)] = [(Y_k^\prime, (\psi_{k-1} \amalg \psi_k) \circ \phi_k^\prime)] \text{ for all } k = 0, \dots, m.
\]

We know we can find a Morse datum for any bordism $(Y,\phi)$ and a Morse datum is equivalent to a Cerf decomposition for $(Y,\phi)$. We would like to define invariants of $(Y, \phi)$ by picking a Morse datum, defining the invariants for the elementary bordisms  appearing in the associated Cerf decomposition, and then ``gluing together'' the invariants of the elementary bordisms to obtain an invariant for $(Y,\phi)$. Certainly we want such an invariant to only depend on the equivalence class $[(Y,\phi)] \in \Hom_{\Bord_{n+1}^0}(X_-, X_+)$. To this end, we would like to understand exactly how two different Cerf decompositions of a given equivalence class can differ.

\begin{defn}
(Cerf Moves) Let $[(Y,\phi)] \in \Hom_{\Bord_{n+1}^0}(X_-, X_+)$ be an equivalence class of bordisms and suppose we have a Cerf decomposition
\[
	[(Y,\phi)] = [(Y_1, \phi_1)] \circ \cdots \circ [(Y_m,\phi_m)].
\]
By a {\bf Cerf move} we mean one of the following modifications made to $[(Y_1, \phi_1)] \circ \cdots \circ [(Y_m,\phi_m)]$ (below we omit the boundary parametrizations $\phi_k$ to simplify notation):
\begin{itemize}
	\item[(a)] ({\bf Critical point cancellation}) Replace $\cdots \circ [Y_k] \circ [Y_{k+1}] \circ \cdots$ with $\cdots \circ [Y_k \cup Y_{k+1}] \circ \cdots$ if $[Y_k \cup Y_{k+1}]$ is a cylindrical bordism.
	\item[(b)] ({\bf Critical point switch}) Replace $\cdots \circ [Y_k] \circ [Y_{k+1}] \circ \cdots$ with $\cdots \circ [Y_k^\prime] \circ [Y_{k+1}^\prime] \circ \cdots$, where $Y_k$, $Y_{k+1}$, $Y_k^\prime$, and $Y_{k+1}^\prime$ satisfy the following conditions: $Y_k \cup Y_{k+1} \cong Y_k^\prime \cup Y_{k+1}^\prime$, and for some choice of Morse data $(f, \underline{t})$, $(f^\prime, \underline{t}^\prime)$ inducing the two Cerf decompositions and a metric on $Y$, the attaching cycles for the critical points $y_k$ and $y_{k+1}$ of $f$ in $X_k$ and $y_k^\prime$ and $y_{k+1}^\prime$ of $f^\prime$ in $X_k^\prime$ are disjoint; in $X_{k-1} = X_{k-1}^\prime$, the attaching cycles of $y_k$ and $y_{k+1}^\prime$ are homotopic; and in $X_{k+1} = X_{k+1}^\prime$ the attaching cycles of $y_{k+1}$ and $y_k^\prime$ are homotopic. See Figure \ref{fig:critptswitch} for an example of this move.
	\item[(c)] ({\bf Cylinder creation}) Replace $\cdots \circ [Y_k] \circ 
	\cdots$ with $\cdots \circ [Y_k^\prime] \circ [Y_k^{\prime\prime}] \circ \cdots$ where $Y_k \cong Y_k^\prime \cup Y_k^{\prime\prime}$ and one of $[Y_k^\prime]$, $[Y_k^{\prime\prime}]$ is cylindrical.
	\item[(d)] ({\bf Cylinder cancellation}) Replace $\cdots \circ [Y_k] \circ [Y_{k+1}] \circ \cdots$ with $\cdots \circ [Y_k \circ Y_{k+1}] \circ \cdots$ whenever one of $[Y_k]$, $[Y_{k+1}]$ is cylindrical.
\end{itemize}
\end{defn}

\begin{figure}[h]
	\centering
	\includegraphics[scale=.75]{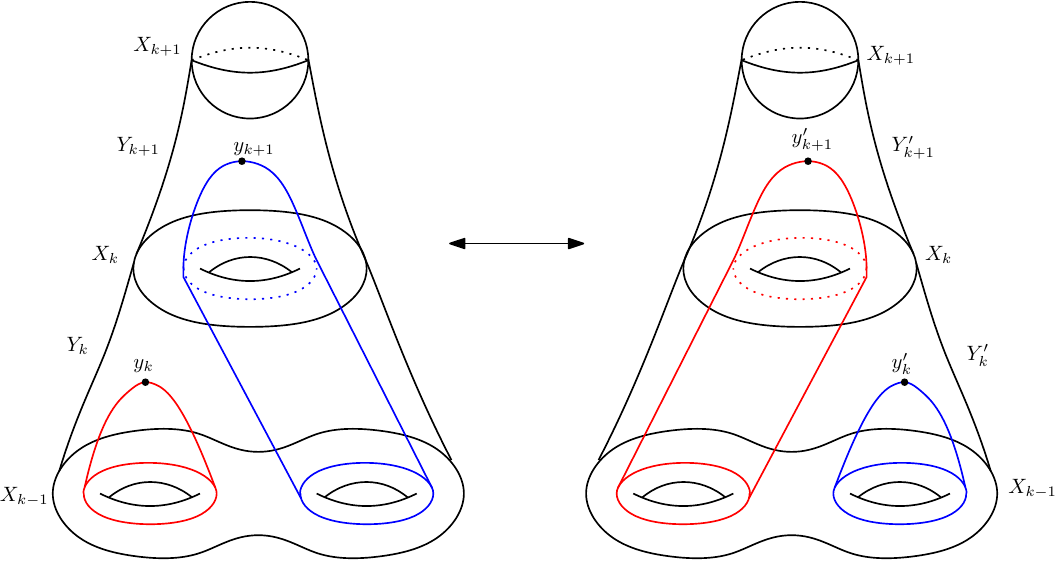}
	\caption{A $3$-dimensional critical point switch of two $2$-handles.}
	\label{fig:critptswitch}
\end{figure}

\begin{thm}
\label{thm:cerftheory}
\textup{(Theorem 3.4 of \cite{conn-cerf})} If $[(Y, \phi)]$ is a connected bordism of dimension at least three, then up to equivalence, any two Cerf decompositions of $[(Y,\phi)]$ are related by a finite sequence of Cerf moves.
\end{thm}

The main use of Theorem \ref{thm:cerftheory} for us is the following. Let $\mathcal{C}$ be some category. Suppose we wish to define a ``$\mathcal{C}$-valued connected field theory'' for $(n+1)$-dimensional bordisms ($n \geq 2$), \emph{i.e.} a functor $\mathcal{F}: \Bord_{n+1}^0 \longrightarrow \mathcal{C}$. If we can define $\mathcal{F}$ on all objects (closed, connected, oriented $n$-manifolds) and all \emph{elementary} $(n+1)$-dimensional bordisms in such a way that $\mathcal{F}$ has the same value on compositions of elementary bordisms differing by Cerf moves, then from Theorem \ref{thm:cerftheory} it follows that this partially defined functor $\mathcal{F}$ uniquely extends to a functor $\mathcal{F}: \Bord_{n+1}^0 \longrightarrow \mathcal{C}$.

\subsection{Symplectic Instanton Homology via Lagrangian Correspondences}

We now precisely describe a way to add a ``trivial $k$-stranded tangle'' to a Cerf decomposition. Let $(Y^{n+1},\phi)$ be a bordism with $\partial Y_-, \partial Y_+ \neq \varnothing$ and fix a Morse datum $(f, \underline{t})$ for $Y$. Let $(Y_0, \dots, Y_m)$ denote the Cerf decomposition for $Y$ induced by $(f, \underline{t})$. Choose $k$ points $x_1, \dots, x_k \in D^n$ and let $S_{[a,b]} = (D^n - \{x_1, \dots, x_k\}) \times [a,b]$. $S_{[a,b]}$ should be thought of as (the complement of) a trivial $k$-stranded tangle in the cylinder $D^n \times [a,b]$. We wish to have a way to insert copies of $S_{[a,b]}$ into $(Y_0, \dots, Y_m)$.

Let $z \in \text{int}(Y)$ be a point contained in a gradient flow line of $f$ connecting $\partial Y_-$ to $\partial Y_+$. Denote this gradient flow line by $\gamma_z$. Then we may form a new Cerf decomposition $(Y_0 \#_z S_{[t_0,t_1]}, \allowbreak\dots, Y_m \#_z S_{[t_{m},t_{m+1}]})$, where $Y_j \#_z S_{[t_{j},t_{j+1}]}$ denotes the result of removing a neighborhood of $\gamma_z \cap Y_j$ and gluing $S_{[t_j,t_{j+1}]}$ back in its place. $(Y_0 \#_z S_{[t_0,t_1]}, \dots,\allowbreak Y_m \#_z S_{[t_{m},t_{m+1}]})$ is then a Cerf decomposition describing the complement of an unknotted $k$-stranded tangle connecting $\partial Y_-$ and $\partial Y_+$.

A picture is more illuminating than the construction in the previous paragraph; see Figure \ref{fig:trivialtangle} for an example of adding a trivial $3$-stranded tangle to a $3$-dimensional bordism.

\begin{figure}[h]
	\centering
	\includegraphics[scale=.75]{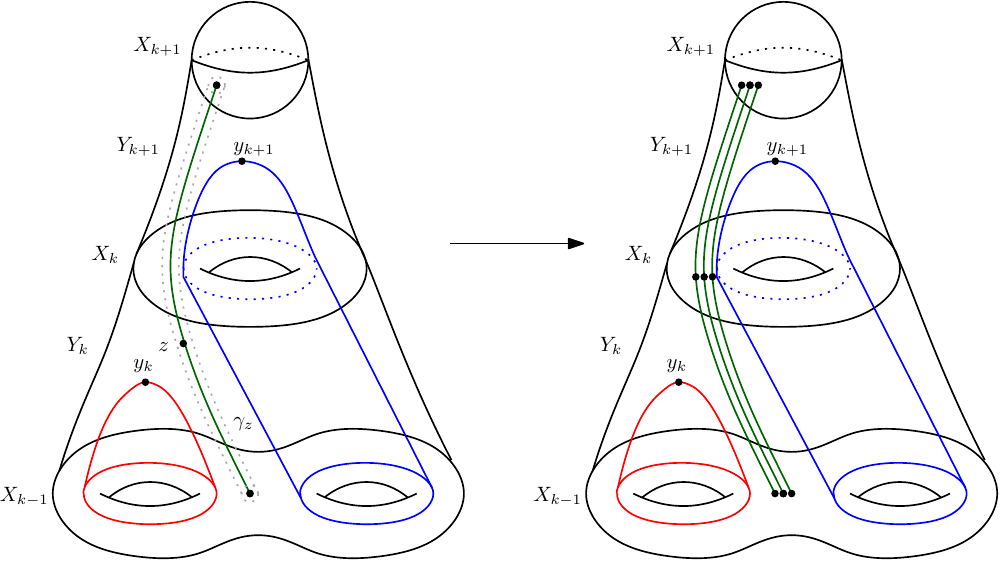}
	\caption{Adding a trivial $3$-stranded tangle to a $3$-dimensional Cerf decomposition of a genus $2$ handlebody minus a solid ball.}
	\label{fig:trivialtangle}
\end{figure}

To apply the above construction to a closed, oriented $3$-manifold $Y$, choose two $3$-balls $B_-$, $B_+$ in $Y$ and consider $Y' = Y - (B_- \amalg B_+)$ as a bordism from $\partial Y_- = \partial B_-$ to $\partial Y_+ = \partial B_+$. We may then choose a Morse datum $(f, \underline{t})$ for $Y'$ with associated Cerf decomposition $(Y_0^\prime, \dots, Y_m^\prime)$ and perform the above construction to obtain a Cerf decomposition $(Y_0^\prime \#_z S_{[t_0,t_1]}, \dots, Y_m^\prime \#_z S_{[t_m,t_{m+1}]})$ which describes $Y'$ with a trivial $k$-stranded tangle connecting $\partial B_-$ to $\partial B_+$. To simplify the notation, we will write $(Y_1^\theta, \dots, Y_m^\theta)$ for this Cerf decomposition of $Y'$ with an open regular neighborhood of the trivial $k$-stranded tangle removed\hide{(the $\theta$ indicates that $Y'$ is $Y$ with an open regular neighborhood of a theta graph removed)}. Note that each $Y_j^\theta$ has three boundary strata: the {\bf negative boundary} $\partial_- Y_j^\theta = X_{j-1} \setminus \{\text{3 open disks}\}$, the {\bf positive boundary} $\partial_+ Y_j^\theta = X_j \setminus \{\text{3 open disks}\}$, and the {\bf vertical boundary} $\partial_\text{vert} Y_j^\theta = \partial \nu T \cap Y_j^\theta$, where $T$ is the trivial $3$-stranded tangle we have constructed.

Now we bring moduli space of $\SU(2)$-representations into the picture. We fix the number of strands in our trivial tangle (denoted by $k$ above) to be $3$. For each $j$, we may define the moduli space $\scrR(Y_j^\theta)$ of conjugacy classes of $\SU(2)$-representations of $\pi_1(Y_j^\theta)$ which send the meridians of the trivial $3$-stranded tangle to the conjugacy class of $\SU(2)$ consisting of traceless matrices.

If $\iota_j: \partial Y_j^\theta \hookrightarrow Y_j^\theta$ denotes the inclusion of the boundary and $\phi_j: \partial Y_j^\theta \longrightarrow \bar{X}_{j,-} \amalg X_{j,+}$ denotes the parametrization of $\partial Y_j^\theta$ (which we have been suppressing from our notation), we can define a map 
\[
	\scrR(Y_j^\theta) \longrightarrow \scrR_{g(X_{j,-}),3} \times \scrR_{g(X_{j,+}),3},
\]
\[
	[\rho] \mapsto [\rho \circ \iota_{\#} \circ \phi_{\#}^{-1}].
\]
Denoting the image of this map by $L(Y_j^\theta)$, we have the following.

\begin{prop}
\label{prop:cerf-lagrangian}
$L(Y_j^\theta)$ is a smooth Lagrangian submanifold of $\scrR_{g(X_{j,-}),3}^- \times \scrR_{g(X_{j,+}),3}$, \emph{i.e.} it defines a Lagrangian correspondence between the moduli spaces of the boundary components of $Y_j^\theta$.
\end{prop}

\begin{proof}
This is easiest to see in terms of the gauge theoretic description of the moduli spaces. An element of $\scrR_{g,3}$ can be considered as a gauge equivalence class of $\SU(2)$-connections $[A]$ on $\Sigma_{g,3}$, where the holonomy of $A$ around any of the boundary components is a traceless $\SU(2)$ matrix. With this interpretation, $L(Y_j^\theta) \subset \scrR_{g(X_{j,-}),3}^- \times \scrR_{g(X_{j,+}),3}$ consists of pairs of $\SU(2)$-connections $([A^-], [A^+])$ which simultaneously extend to an $\SU(2)$-connection $[A]$ on $Y_j \setminus \nu T$ that has traceless holonomy around any of the strands of the trivial $3$-stranded tangle $T$. Here, $\nu T$ denotes a \emph{closed} tubular neighborhood of $T$, so that $Y_j \setminus \nu T$ is noncompact.

The tangent space $T_{[A]} \scrR_{g,3}$ may be identified with
\[
	\im(H^1_c(\Sigma_{g,3}; \su(2)_A) \longrightarrow H^1(\Sigma_{g,3}; \su(2)_A)),
\]
and the symplectic form $\omega$ is given by the familiar formula
\[
	\omega_{[A]}(\alpha, \beta) = \frac{1}{4\pi^2}\int_{\Sigma_g^\theta} \Tr (\alpha \wedge \beta).
\]
Similarly, we have
\[
	T_{[A]} \scrR(Y_j^\theta) = \im(H^1_c(Y_j \setminus \nu T; \su(2)_A) \longrightarrow H^1(Y_j \setminus \nu T; \su(2)_A)).
\]

Let $\tilde{\alpha}$, $\tilde{\beta}$ be two tangent vectors to the traceless character variety for $Y_j^\theta$ at a point $[A]$, and denote their images in $L(Y_j^\theta)$ by $\alpha$, $\beta$, elements of the tangent space at the connection $([A^-],[A^+])$, where $[A^\pm]$ are the boundary values of $[A]$. By Stokes' theorem,
\begin{align*}
	0 & = \int_{Y_j^\theta} d\langle \tilde{\alpha} \wedge \tilde{\beta} \rangle \\
	 & = \int_{X_{j,+}^\theta} \langle \alpha \wedge \beta \rangle - \int_{X_{j,-}^\theta} \langle \alpha \wedge \beta \rangle + \int_{\partial \nu T} \langle \tilde{\alpha} \wedge \tilde{\beta} \rangle \\
	 & = 4\pi^2\omega_{([A_-], [A_+])}(\alpha, \beta) + \int_{\partial \nu T} \langle \tilde{\alpha} \wedge \tilde{\beta} \rangle.
\end{align*}
The last term on the right is zero, since the forms $\tilde{\alpha}$, $\tilde{\beta}$ are compactly supported in $Y_j \setminus \nu T$. Therefore $L(Y_j^\theta)$ is isotropic in $\scrR_{g(X_{j,-}),3}^- \times \scrR_{g(X_{j,+}),3}$. It is also clearly half-dimensional, so that it is a Lagrangian.
\end{proof}

We wish to define the symplectic instanton homology of the Cerf decomposition $(Y_1^\theta, \dots, Y_m^\theta)$ as the quilted Floer homology $\HF(L(Y_1^\theta), \dots, L(Y_m^\theta))$. Since the Lagrangians $L(Y_k^\theta)$ have minimal Maslov number $2$, we will need to ensure that contributions of disk bubbles to $\partial^2$ cancel out. We show this in Theorem \ref{thm:cerf-disk-invariant}. First, it is useful to have a description of the effect of gluing cobordisms on the Lagrangian correspondences.

\begin{prop}
\label{prop:lag-composition}
Composition of bordisms corresponds to geometric composition (not necessarily embedded) of the associated Lagrangian correspondences:
\[
	L(Y_k^\theta \circ Y_{k+1}^\theta) = L(Y_k^\theta) \circ L(Y_{k+1}^\theta).
\]
\end{prop}

\begin{proof}
By the Seifert-Van Kampen theorem, any representation of $\pi_1(Y_k^\theta \circ Y_{k+1}^\theta)$ induces representations of $\pi_1(Y_k^\theta)$ and $\pi_1(Y_{k+1}^\theta)$ whose restrictions to $\pi_1(X_k^\theta)$ agree, and therefore an element of the geometric composition $L(Y_k^\theta) \circ L(Y_{k+1}^\theta)$. On the other hand, an element of $L(Y_k^{\theta}) \circ L(Y_{k+1}^\theta)$ consists of a pair of representations $[\rho_k] \in L(Y_k^\theta)$ and $[\rho_{k+1}] \in L(Y_{k+1}^\theta)$ whose restrictions to $\pi_1(X_k^\theta)$ are conjugate. Two such representations glue together to a representation of $\pi_1(Y_k^\theta \circ Y_{k+1}^\theta)$, unique up to conjugation, by fixing representatives for $[\rho_k]$ and $[\rho_{k+1}]$ and conjugating the second one to exactly agree with the first on $\pi_1(X_k^\theta)$.
\end{proof}

With the above in place, we now prove that the disk invariants of the Lagrangians associated to a Cerf decomposition cancel out.

\begin{thm}
\label{thm:cerf-disk-invariant}
For any Cerf decomposition of the form $(Y_1^\theta, \dots, Y_m^\theta)$ and generic compatible almost complex structures $J_k \in \J(X_k^\theta)$, the disk invariants of the associated Lagrangian correspondences satisfy
\[
	\Phi_{J_0 \oplus J_1}(L(Y_1^\theta), x_1) + \cdots + \Phi_{J_{m-1} \oplus J_m}(L(Y_m^\theta),x_m) = 0.
\]
\end{thm}

\begin{proof}
For each $k = 1, \dots, m$, $Y_k^\theta$ must be a $1$-handle attachment, a $2$-handle attachment, or a cylinder. Without loss of generality, we may assume there are no cylinders, since by \cite[Theorem 5.4.1]{quilts}, disk numbers are additive under geometric composition, $\Phi(L \circ L') = \Phi(L) + \Phi(L')$, and cylinders induce the diagonal correspondence. Furthermore, the number of $1$-handles must equal the number of $2$-handles, since the bordism starts and ends at the same surface.

Similarly to the proof of Proposition \ref{prop:disk-inv-equal}, if $L, L': \scrR_{g,3} \longrightarrow \scrR_{g-1,3}$ correspond to $2$-handle attachments, then $\Phi(L) = \Phi(L')$ because there is a symplectic automorphism of $\scrR_{g,3}^- \times \scrR_{g-1,3}$ (induced by a diffeomorphism of the $2$-handle cobordism $\Sigma_{g,3} \longrightarrow \Sigma_{g-1,3}$ taking $L$ to $L'$. Dually, a correspondence $L''$ coming from a $1$-handle attachment is the transpose of the correspondence associated to the $2$-handle attachment obtained by turning the $1$-handle cobordism ``upside down.'' This implies that $\Phi(L'')$ is the negative of the disk invariant of any $2$-handle correspondence. As noted in the previous paragraph, there are an equal number of $1$- and $2$-handles, so the total sum of disk invariants is therefore zero, as desired.
\end{proof}

Now, given a Cerf decomposition of the form $(Y_1^\theta, \dots, Y_m^\theta)$, we define its {\bf symplectic instanton chain group} to be the quilted Floer chain group
\[
	\CSI(Y_1^\theta, \dots, Y_m^\theta) = \CF(L(Y_1^\theta), \dots, L(Y_m^\theta)).
\]
For a generic $(m+1)$-tuple of compatible almost complex structures $\underline{J} = (J_0, \dots, J_m) \in \J(X_0) \times \cdots \times \J(X_m)$, we have a quilted Floer differential $\partial_{\underline{J}}$ on $\CSI(Y_1^\theta, \dots, Y_m^\theta)$. Studying the ends of the moduli space of Maslov index $2$ trajectories as usual, we find that
\[
	\partial_{\underline{J}}^2 = (\Phi_{J_0 \oplus J_1}(L(Y_1^\theta)) + \cdots + \Phi_{J_{m-1} \oplus J_m}(L(Y_m^\theta))\mathrm{Id}.
\]
Therefore by Theorem \ref{thm:cerf-disk-invariant} we have that $\partial_{\underline{J}}^2 = 0$, and we have proven the following:

\begin{thm}
$(\CSI(Y_1^\theta, \dots, Y_m^\theta), \partial_{\underline{J}})$ is a chain complex.
\end{thm}

We would like to show that the homology of $(\CSI(Y_1^\theta, \dots, Y_m^\theta), \partial_{\underline{J}})$ is an invariant of the equivalence class of $(Y_1, \dots, Y_m)$ (the original Cerf decomposition of the closed, oriented $3$-manifold $Y$), and hence an invariant of the original $3$-manifold $Y$. This is indeed the case. In fact, we will show that $(L(Y_1^\theta), \dots, L(Y_m^\theta)) \in \Hom_\Symp(\scrR_{0,3}, \scrR_{0,3})$ depends only on $Y$.

\begin{thm}
\textup{(Invariance under Cerf moves)} Given a closed, oriented $3$-manifold $Y$ and any Cerf decomposition of the form $(Y_1^\theta, \dots, Y_m^\theta)$, the generalized Lagrangian correspondence
\[
	(L(Y_1^\theta), \dots, L(Y_m^\theta)) \in \Hom_\Symp(\scrR_{0,3}, \scrR_{0,3})
\]	
depends only on the diffeomorphism type of $Y$. More precisely, $(L(Y_1^\theta), \dots, L(Y_m^\theta))$ is invariant under Cerf moves applied to $(Y_1^\theta, \dots, Y_m^\theta)$.
\label{thm:cerfinvariance}
\end{thm}

\begin{proof}
We simply compare the Lagrangian correspondences for Cerf decompositions differing by a single Cerf move. Note that the trivial $3$-stranded tangle in the decomposition is not changed in any of the Cerf moves; we can consider all moves as coming from moves on the Cerf decomposition $(Y_1, \dots, Y_m)$ of the closed $3$-manifold $Y$.
\begin{itemize}
	\item[(a)] \emph{Critical Point Cancellation:} Invariance under critical point cancellation follows immediately from Proposition \ref{prop:lag-composition}.
	\item[(b)] \emph{Critical Point Switch:} Let $Y_k^\theta$, $Y_{k+1}^\theta$ be two consecutive bordisms in the Cerf decomposition satisfying the necessary conditions for a critical point switch to be performed, and let $(Y_k^\theta)^\prime$, $(Y_{k+1}^\theta)^\prime$ be the bordisms that replace $Y_k^\theta$, $Y_{k+1}^\theta$ after the critical point switch. It is apparent that the geometric compositions $L(Y_k^\theta) \circ L(Y_{k+1}^\theta)$ and $L((Y_k^\theta)^\prime) \circ L((Y_{k+1}^\theta)^\prime)$ are embedded and equal to one another, since $Y_k^\theta \circ Y_{k+1}^\theta \cong (Y_k^\theta)^\prime \circ (Y_{k+1}^\theta)^\prime$ by assumption. Therefore we have the following equivalences in $\Symp$:
\begin{align*}
	(\dots, L(Y_k^\theta), L(Y_{k+1}^\theta), \dots) & = (\dots, L(Y_k^\theta) \circ L(Y_{k+1}^\theta), \dots) \\
	 & = (\dots, L((Y_k^\theta)^\prime) \circ L((Y_{k+1}^\theta)^\prime), \dots) \\
	 & = (\dots, L((Y_k^\theta)^\prime), L((Y_{k+1}^\theta)^\prime), \dots),
\end{align*}
	which shows that a critical point switch does not change the morphism in $\Symp$ represented by the Cerf decomposition.
	\item[(c)] \emph{Cylinder Creation:} Suppose we replace $Y_k^\theta$ with $(Y_k^\theta)^\prime, (Y_k^\theta)^{\prime\prime}$ satisfying $(Y_k^\theta)^\prime \cup (Y_k^\theta)^{\prime\prime} \cong Y_k^\theta$, where without loss of generality we assume that $(Y_k^\theta)^\prime$ is cylindrical. Then $(Y_k^\theta)^{\prime\prime} \cong Y_k^\theta$, and the Lagrangian correspondence for $(Y_k^\theta)^\prime$ is the diagonal $\Delta_{\scrR(X_{k}^\theta)}: \scrR(X_{k}^\theta)^- \times \scrR(X_{k}^\theta)$. Then
\begin{align*}
	L((Y_k^\theta)^\prime \circ (Y_k^\theta)^{\prime\prime}) & = L((Y_k^\theta)^\prime) \circ L((Y_k^\theta)^{\prime\prime}) \\
	 & = \Delta_{\scrR(X_{k}^\theta)} \circ L(Y_k^\theta) \\
	 & = L(Y_k^\theta).
\end{align*}
	\item[(d)] \emph{Cylinder Cancellation:} If $Y_k^\theta$ is cylindrical, then similar to the above case,
\[
	L(Y_k^\theta \circ Y_{k+1}^\theta) = L(Y_k^\theta) \circ L(Y_{k+1}^\theta) = \Delta_{\scrR(X_k^\theta)} \circ L(Y_{k+1}^\theta) = L(Y_{k+1}^\theta). \qedhere
\]
\end{itemize}
\end{proof}

As a result of Theorem \ref{thm:cerfinvariance}, the homology
\[
	\SI(Y) := H_\ast(\CSI(Y_1^\theta, \dots, Y_m^\theta), \partial_{\underline{J}})
\]
depends only on the (oriented) diffeomorphism type of $(Y,z)$. We call $\SI(Y)$ the {\bf symplectic instanton homology} of $Y$ (again, we tend to ignore the basepoint in the notation). In fact, we have shown that the equivalence class $(L(Y_1^\theta), \dots, L(Y_m^\theta)) \in \Hom_\Symp(\scrR_{0,3}, \scrR_{0,3})$ is an invariant of $(Y,z)$, but we will not study the $\Symp$-valued invariant any further in this article.

Note that a pointed Heegaard diagram $\calH = (\Sigma_g, \bfalpha, \bfbeta, z)$ induces a  Cerf decomposition $Y = (Y_1, \dots, Y_{2g})$, where $Y_k$ corresponds to $1$-handle attachments determined by $\alpha_k$ ($1 \leq k \leq g$) and $Y_{g + k}$ corresponds to $2$-handle attachments determined by $\beta_k$. It is easily checked that for $1 \leq k \leq g$, the geometric compositions $L(Y_1 \circ \cdots \circ Y_{k-1}) \circ L(Y_k)$ and $L(Y_{g+1} \circ \cdots \circ Y_{g+k-1}) \circ L(Y_{g+k})$ are embedded and, when $k = g$, are respectively equal to $L_\alpha$, $L_\beta$, so that
\[
	\SI(Y) = H_\ast(\CSI(Y_1^\theta, \dots, Y_{2g}^\theta)) = H_\ast(\CSI(L_\alpha, L_\beta)),
\]
\emph{i.e.} the definition of symplectic instanton homology in terms of Heegaard diagrams is a special case of the definition using Cerf decompositions.

\hide{
\begin{rem}
As mentioned in the Introduction, there is another object in the literature called ``symplectic instanton homology'' which is also a finitely generated group $\HSI(Y)$ associated to a closed, oriented $3$-manifold, originally constructed by Manolescu and Woodward \cite{manolescu-woodward}. In Section \ref{sect:comparison} we will prove that $\SI(Y) \cong \HSI(Y)$.
\end{rem}
}

%% file: SI-nontrivial.tex
\section{Symplectic Instanton Homology of Nontrivial $\SO(3)$-Bundles}

The symplectic instanton homology $\SI(Y)$ can be thought of as using the trivial $\SU(2)$-bundle on $Y$ in its definition. More generally, one may wish to define symplectic instanton homology using a nontrivial $\SO(3)$-bundle on $Y$, and it will indeed be necessary to have such a generalization in order to state the surgery exact triangle for symplectic instanton homology.

\subsection{Moduli Spaces of Flat $\SO(3)$-Bundles}

On a compact, orientable manifold $Y$ of dimension at most $3$, principal $\SO(3)$-bundles $P \longrightarrow Y$ are classified by their second Stiefel-Whitney class $w_2(P) \in H^2(Y,\partial Y;\F_2)$. In particular, if $Y$ is a compact, orientable $3$-manifold, then by Poincar\'e-Lefschetz duality an $\SO(3)$-bundle $P \longrightarrow Y$ is classified by the homology class $\omega = \mathrm{PD}(w_2(P)) \in H_1(Y; \F_2)$. The following proposition, which describes how the homology class $\omega$ appears in the holonomy description for the moduli space of flat connections on $P \longrightarrow Y$, is well-known.

\begin{prop}
\label{thm:so3holonomy}
Let $Y$ be a compact, orientable $3$-manifold and $P \longrightarrow Y$ a principal $\SO(3)$-bundle on $Y$. Let $\omega = \mathrm{PD}(w_2(P))$ and also let $\omega$ denote a link in $Y$ representing this homology class, by abuse of notation. Then the moduli space of flat connections on $P$ has the holonomy description
\[
	\{\rho: \pi_1(Y \setminus \omega) \longrightarrow \SU(2) \mid \rho(\mu_{\omega_k}) = -I \text{ for } k = 1,\ \dots, m\}/\mathrm{conjugation},
\]
where $\mu_{\omega_k}$ is a meridian of the $k^\text{th}$ component of $\omega = \omega_1 \amalg \cdots \amalg \omega_m$.
\end{prop}

\subsection{Cerf Decompositions with Homology Class}

In analogy with the definition of $\SI(Y)$ (which corresponds to $\omega = 0$) we wish to define a ``twisted'' version of symplectic instanton homology using moduli spaces of flat connections on the nontrivial bundle $P \longrightarrow Y$ by looking at the pieces of a Cerf decomposition of $Y$. To do this, we will incorporate $P$ into our Cerf decompositions through the homology class $\omega = \text{PD}(w_2(P))$, by virtue of Proposition \ref{thm:so3holonomy}.

\begin{defn}
Let $Y$ be a compact, oriented $3$-manifold and $\omega \in H_1(Y; \F_2)$ be a mod $2$ homology class in $Y$. A {\bf Cerf decomposition} of $(Y, \omega)$ is a decomposition
\[
	[(Y, \omega)] = [(Y_1, \omega_1)] \circ \cdots \circ [(Y_m, \omega_m)],
\]
where $[Y] = [Y_1] \circ \cdots \circ [Y_m]$ is a Cerf decomposition of $Y$ in the usual sense and for $k = 1, \dots, m$, $\omega_k \in H_1(Y_k; \F_2)$ and these classes satisfy the condition that $\omega_1 + \cdots + \omega_m = \omega$ in $H_1(Y; \F_2)$ (where we conflate $\omega_k$ with its image in $H_1(Y; \F_2)$ under inclusion ).
\end{defn}

In order to decide which Cerf decompositions determine the same pair $(Y,\omega)$, we need Cerf moves for these Cerf decompositions with homology class. The Cerf moves in this context just require slight modifications made to the usual moves where the homology class is not considered.

\begin{defn}
(Cerf Moves) Let $[(Y,\phi)] \in \Hom_{\Bord_3^0}(X_-, X_+)$ be an equivalence class of bordisms and $\omega \in H_1(Y; \F_2)$. Suppose we have a Cerf decomposition
\[
	[(Y,\omega)] = [(Y_1, \omega_1)] \circ \cdots \circ [(Y_m,\omega_m)].
\]
By a {\bf Cerf move} we mean one of the following modifications made to $[(Y_1, \omega_1)] \circ \cdots \circ [(Y_m,\omega_m)]$:
\begin{itemize}
	\item[(a)] ({\bf Critical point cancellation}) Replace $\cdots \circ [(Y_k,\omega_k)] \circ [(Y_{k+1},\omega_{k+1})] \circ \cdots$ with $\cdots \circ [(Y_k \cup Y_{k+1}, \omega_k + \omega_{k+1})] \circ \cdots$ if $[Y_k \circ Y_{k+1}]$ is a cylindrical bordism.
	\item[(b)] ({\bf Critical point switch}) Replace $\cdots \circ [(Y_k,\omega_k)] \circ [(Y_{k+1},\omega_{k+1})] \circ \cdots$ with $\cdots \circ [(Y_k^\prime,\omega_k^\prime)] \circ [(Y_{k+1}^\prime,\omega_{k+1}^\prime)] \circ \cdots$, where $Y_k$, $Y_{k+1}$, $Y_k^\prime$, and $Y_{k+1}^\prime$ satisfy the following conditions: $Y_k \cup Y_{k+1} \cong Y_k^\prime \cup Y_{k+1}^\prime$, and for some choice of Morse data $(f, \underline{t})$, $(f^\prime, \underline{t}^\prime)$ inducing the two Cerf decompositions and a metric on $Y$, the attaching cycles for the critical points $y_k$ and $y_{k+1}$ of $f$ in $X_k$ and $y_k^\prime$ and $y_{k+1}^\prime$ of $f^\prime$ in $X_k^\prime$ are disjoint; in $X_{k-1} = X_{k-1}^\prime$, the attaching cycles of $y_k$ and $y_{k+1}^\prime$ are homotopic; and in $X_{k+1} = X_{k+1}^\prime$ the attaching cycles of $y_{k+1}$ and $y_k^\prime$ are homotopic. See Figure \ref{fig:critptswitch} for an example of this move. Furthermore, the homology classes involved should satisfy $\omega_k + \omega_{k+1} = \omega_k^\prime + \omega_{k+1}^\prime$.
	\item[(c)] ({\bf Cylinder creation}) Replace $\cdots \circ [(Y_k,\omega_k)] \circ \cdots$ with $\cdots \circ [(Y_k^\prime,\omega_k^\prime)] \circ [(Y_k^{\prime\prime},\omega_k^{\prime\prime})] \circ \cdots$ where $Y_k \cong Y_k^\prime \cup Y_k^{\prime\prime}$, one of $[Y_k^\prime]$, $[Y_k^{\prime\prime}]$ is cylindrical, and $\omega_k = \omega_k^\prime + \omega_k^{\prime\prime}$.
	\item[(d)] ({\bf Cylinder cancellation}) Replace $\cdots \circ [(Y_k,\omega_k)] \circ [(Y_{k+1},\omega_{k+1})] \circ \cdots$ with $\cdots \circ [(Y_k \circ Y_{k+1},\omega_k+\omega_{k+1})] \circ \cdots$ whenever one of $[Y_k]$, $[Y_{k+1}]$ is cylindrical.
	\item[(e)] ({\bf Homology class swap}) Replace $\cdots \circ [(Y_k,\omega_k)] \circ [(Y_{k+1},\omega_{k+1})] \circ \cdots$ with $\cdots \circ [(Y_k,\omega_k^\prime)] \circ [(Y_{k+1},\omega_{k+1}^\prime)] \circ \cdots$ whenever $\omega_k + \omega_{k+1} = \omega_k^\prime + \omega_{k+1}^\prime$ and at least one of $[Y_k]$, $[Y_{k+1}]$ is cylindrical.
\end{itemize}
\end{defn}

The following fundamental theorem follows effortlessly from the usual Cerf theory.

\begin{thm}
If $[(Y,\phi)]$ is a connected bordism of dimension at least three and $\omega \in H_1(Y; \F_2)$, then any two Cerf decompositions of $[(Y, \phi, \omega)]$ are related by a finite sequence of Cerf moves.
\end{thm}

\subsection{Definition via Cerf Decompositions}

As before, given a closed, oriented $3$-manifold $Y$, we can remove a pair of open balls from $Y$ to get a $3$-manifold $Y'$ with a Cerf decomposition $(Y_1^\prime, \dots, Y_m^\prime)$, and then we can remove a trivial $3$-stranded tangle connecting the boundary components and missing all critical points to get a $3$-manifold $Y^\theta$ with induced Cerf decomposition $(Y_1^\theta, \dots, Y_m^\theta)$.

We wish to incorporate the data of a nontrivial $\SO(3)$-bundle $\omega$ into our Lagrangian correspondences. Abusing notation, let $\omega \in H_1(Y; \F_2)$ be the Poincar\'e dual to the second Stiefel-Whitney class of our $\SO(3)$-bundle. We may represent $\omega$ by a smoothly embedded unoriented curve in $Y$, which furthermore may be assumed disjoint from the $3$-balls we remove, so that we get an induced curve $\omega$ (taking our abuse of notation further) in $Y^\theta$ that is geometrically unlinked with the trivial $3$-stranded tangle.

In terms of the Cerf decomposition $Y^\theta = (Y_1^\theta, \dots, Y_m^\theta)$, a generic embedded representative of $\omega$ intersects each intermediate surface $X_k^\theta$ in an even number of points (since they are all separating). We may then eliminate the intersection points in pairs without changing the homology class represented by the resulting curve; indeed, the two curves are homologous via a saddle cobordism. As a result, for each $k = 1, \dots, m$, we have an unoriented curve $\omega_k$ (possibly empty, but we may assume it is connected up to homology by using saddle moves to merge components) contained on the interior of $Y_k^\theta$, and $\omega_1 + \cdots + \omega_m$ is homologous to our original $\omega$.

With this data fixed, associate to each piece $Y_k^\theta$ of the Cerf decomposition the moduli space
\[
	\scrR(Y_k^\theta, \omega_k) = \left.\left\{ \begin{array}{c} \rho: \pi_1(Y_k^\theta \setminus (3\text{-tangle} \cup \omega_k)) \\ \longrightarrow \SU(2) \end{array} \left| \begin{array}{c} \rho(\text{tangle meridians}) \text{ are traceless} \\ \rho(\text{meridian of } \omega_k) = -I \end{array} \right\}\right.\right/ \text{conj.}
\]
Denote the image of $\scrR(Y_k^\theta, \omega_k)$ in $\scrR_{g(X_{k-1}),3}^- \times \scrR_{g(X_k),3}$ under restriction to the boundary by $L(Y_k^\theta, \omega_k)$.

\begin{prop}
$L(Y_k^\theta, \omega_k)$ defines a smooth Lagrangian correspondence from $\scrR_{g(X_{k-1}),3}$ to $\scrR_{g(X_k),3}$.
\end{prop}

\begin{proof}
$L(Y_k^\theta, \omega_k)$ is an isotropic submanifold of $\scrR_{g(X_{k-1}),3}^- \times \scrR_{g(X_k),3}$ for the same reason $L(Y_k^\theta)$ is (the proof applies verbatim). $L(Y_k^\theta, \omega_k)$ is clearly $3(g(X_{k-1}) + g(X_k)) = \frac{1}{2} \dim(\scrR_{g(X_{k-1}),3}^- \times \scrR_{g(X_k),3})$-dimensional, hence it is Lagrangian.
\end{proof}

Now we verify that the generalized Lagrangian correspondence $(L(Y_1^\theta, \omega_1), \dots, L(Y_m^\theta, \omega_m))$ doesn't depend on the choice of Cerf decomposition.

\begin{thm}
The morphism $(L(Y_1^\theta, \omega_1), \dots, L(Y_m^\theta, \omega_m)) \in \Hom_\mathbf{Symp}(\scrR_{0,3}, \scrR_{0,3})$ is unchanged under Cerf moves on $((Y_1^\theta, \omega_1), \dots, (Y_m^\theta, \omega_m))$.
\label{thm:omega-cerf-invariance}
\end{thm}

\begin{proof}
The proof of invariance under Cerf moves (a)-(d) are just trivial modifications of the $\omega = 0$ case proved in Section \ref{chap:SI-cerf}. It remains to check invariance under the new Cerf move appearing when homology classes are introduced, the homology class swap. We wish to compare $(\dots, L(Y_k^\theta, \omega_k), L(Y_{k+1}^\theta, \omega_{k+1}), \dots)$ and $(\dots, L(Y_k^\theta, \omega_k^\prime), L(Y_{k+1}^\theta, \omega_{k+1}^\prime), \dots)$, where at least one of $Y_k^\theta$, $Y_{k+1}^\theta$ are cylindrical and $\omega_k + \omega_{k+1} = \omega_k^\prime + \omega_{k+1}^\prime$. Without loss of generality, suppose $Y_{k+1}^\theta$ is cylindrical. Then $Y_k^\theta \cup Y_{k+1}^\theta \cong Y_k^\theta$ and the associated Lagrangian correspondences are composable, and
\begin{align*}
	L(Y_k^\theta, \omega_k) \circ L(Y_{k+1}^\theta, \omega_{k+1}) & \cong L(Y_k^\theta, \omega_k + \omega_{k+1}) \\
	 & = L(Y_k^\theta, \omega_k + \omega_{k+1}^\prime) \\
	 & \cong L(Y_k^\theta, \omega_k^\prime) \circ L(Y_{k+1}^\theta, \omega_{k+1}^\prime).
\end{align*}
It follows that $(\dots, L(Y_k^\theta, \omega_k), L(Y_{k+1}^\theta, \omega_{k+1}), \dots)$ and $(\dots, L(Y_k^\theta, \omega_k^\prime), L(Y_{k+1}^\theta, \omega_{k+1}^\prime), \dots)$ are identical morphisms in $\Hom_{\Symp}(\scrR_{0,3}, \scrR_{0,3})$.
\end{proof}

As in the $\omega = 0$ case, we wish to take the quilted Floer homology of the generalized Lagrangian correspondence $(L(Y_1^\theta, \omega_1), \dots, L(Y_m^\theta,\omega_m))$ as our definition of the symplectic instanton homology of $(Y,\omega)$. But first we must make sure the Floer homology is actually defined, as the Lagrangians have minimal Maslov number $2$.

\begin{thm}
\label{thm:cerf-disk-invariant-nontrivial}
For any Cerf decomposition with homology class of the form $((Y_1^\theta,\omega_1), \dots,\allowbreak (Y_m^\theta,\omega_m))$ and generic compatible almost complex structures $J_k \in \J(X_k^\theta)$, the disk invariants of the associated Lagrangian correspondences satisfy
\[
	\Phi_{J_0 \oplus J_1}(L(Y_1^\theta,\omega_1), x_1) + \cdots + \Phi_{J_{m-1} \oplus J_m}(L(Y_m^\theta,\omega_m),x_m) = 0.
\]
\end{thm}

\begin{proof}
We show that the statement is equivalent to the $\omega = 0$ case already proven in Theorem \ref{thm:cerf-disk-invariant}. This will be achieved by showing that $\Phi(L(Y_k^\theta,\omega_k)) = \Phi(L(Y_k^\theta, 0))$ for each $k$. To see this, write $W_k^\theta = \partial_+ Y_k^\theta \times [0,1]$. Then $(Y_k^\theta, \omega_k)$ is equivalent to $(Y_k^\theta, \omega_k) \circ (W_k^\theta, 0)$, and by a homology class swap this is further equivalent to $(Y_k^\theta, 0) \circ (W_k^\theta, \omega_k^\prime)$ (here we should assume that the genus of $\partial_+ Y_k^\theta$ is greater than the genus of $\partial_- Y_k^\theta$ so that the homology class swap is possible; if this is not the case, perform the above construction for $W_k^\theta = \partial_- Y_k^\theta \times [0,1]$ instead). Given any simple closed curve $\gamma \subset \partial_+ Y_k^\theta$, write $\omega_k^\prime \cdot \gamma$ for the mod $2$ intersection number of $\omega_k^\prime$ with $\gamma \times [0,1]$ in $W_k^\theta$. Then it is easily seen (cf. Lemma \ref{lem:lagr-omega}) that $L(W_k^\theta, \omega_k^\prime)$ is the graph of the diffeomorphism
\[
	\scrR_{g,3} \longrightarrow \scrR_{g,3}: \rho \mapsto \rho',
\]
where $\rho'(\gamma) = (-1)^{\omega_k^\prime \cdot \gamma}\rho(\gamma)$ for any $\gamma \in \pi_1(\Sigma_{g,3})$. Since $L(W_k^\theta,\omega_k^\prime)$ is Lagrangian in $\scrR_{g,3}^- \times \scrR_{g,3}$, this diffeomorphism is in fact a symplectomorphism, and hence
\[
	\Phi(L(Y_k^\theta, \omega_k)) = \Phi(L(Y_k^\theta,0) \circ L(W_k^\theta,\omega_k^\prime)) = \Phi(L(Y_k^\theta,0)).
\]
Therefore we reduce the desired equality to the $\omega = 0$, which is known to be true by Theorem \ref{thm:cerf-disk-invariant}.
\end{proof}

Theorem \ref{thm:cerf-disk-invariant-nontrivial} implies that the Floer boundary operator $\partial_{\underline{J}}$ on
\[
	\CSI((Y_1^\theta, \omega_1), \dots, (Y_m^\theta,\omega_m)) = \CF(L(Y_1^\theta,\omega_1), \dots, L(Y_m^\theta,\omega_m))
\]
satisfies $\partial_{\underline{J}}^2 = 0$ for generic tuples of almost complex structures $\underline{J} = (J_1, \dots, J_m) \in \J(\scrR_{g_1,3}) \times \cdots \times \J(\scrR_{g_m,3})$. We may therefore define the {\bf symplectic instanton homology} of $(Y,\omega)$ to be the quilted Floer homology group
\[
	\SI(Y,\omega) = H_\ast(\CF(L(Y_1^\theta, \omega_1), \dots, L(Y_m^\theta, \omega_m), \partial_{\underline{J}}).
\]
By Theorem \ref{thm:omega-cerf-invariance}, $\SI(Y,\omega)$ depends only on the diffeomorphism type of $Y$ and the homology class $\omega \in H_1(Y; \F_2)$.

\subsection{Definition via Heegaard Diagrams}

A Heegaard diagram $(\Sigma_g, \bfalpha, \bfbeta, z)$ for a $3$-manifold $Y$ gives a Cerf decomposition with trivial $3$-tangle for $Y^\theta$ by the following process. Start with $\Sigma_g \times [-1,1]$. Attach $2$-handles to $\Sigma_g \times \{-1\}$ along the curves $\alpha_k \times \{-1\}$ to obtain the $\alpha$-handlebody minus a $3$-ball, $H_\alpha^\prime$, and attach $2$-handles to $\Sigma_g \times \{1\}$ along the curves $\beta_k \times \{1\}$ to obtain the $\beta$-handlebody minus a $3$-ball, $H_\beta^\prime$; we then have $Y' := Y \setminus (\text{two $3$-balls}) = H_\alpha^\prime \cup (\Sigma_g \times [-1, 1]) \cup H_\beta^\prime$. Since the $\alpha$-curves (resp. $\beta$-curves) are pairwise disjoint, the order of the handle attachments for each handlebody is irrelevant. We can dually think of $H_\alpha^\prime$ as being obtained from $1$-handle attachments, so that we have a Cerf decomposition
\[
	Y' = Y_{\alpha_1} \cup_{\Sigma_1} Y_{\alpha_2} \cup_{\Sigma_2} \cdots \cup_{\Sigma_{g-1}} Y_{\alpha_g} \cup_{\Sigma_g} Y_{\beta_g} \cup_{\Sigma_{g - 1}} \cdots \cup_{\Sigma_1} Y_{\beta_1},
\]
where each $Y_{\alpha_k}: \Sigma_{k-1} \longrightarrow \Sigma_k$ is a $1$-handle cobordism induced by the curve $\alpha_k$, and each $Y_{\beta_k}: \Sigma_k \longrightarrow \Sigma_{k-1}$ is a $2$-handle cobordism induced by the curve $\beta_k$. Since the basepoint $z \in \Sigma_g$ is disjoint from all attaching cycles, we can fix a Morse function inducing this Cerf decomposition and consider the gradient flow line $\gamma_z$ passing through $z$; this flow line necessarily connects the two boundary $2$-spheres of $Y'$. We may therefore remove a trivial $3$-stranded tangle from a regular neighborhood of $\gamma_z$ get a Cerf decomposition
\[
	Y^\theta = Y_{\alpha_1}^\theta \cup_{\Sigma_{1,3}} Y_{\alpha_2}^\theta \cup_{\Sigma_{2,3}} \cdots \cup_{\Sigma_{g-1,3}} Y_{\alpha_g}^\theta \cup_{\Sigma_{g,3}} Y_{\beta_g}^\theta \cup_{\Sigma_{g - 1,3}} \cdots \cup_{\Sigma_{1,3}} Y_{\beta_1}^\theta.
\]
Given a mod $2$ homology class $\omega \in H_1(Y; \F_2)$, we can write
\begin{align*}
	\omega & = \omega_{\alpha_1} + \cdots + \omega_{\alpha_g} + \omega_{\beta_g} + \cdots + \omega_{\beta_1} \\
	& \in H_1(Y_{\alpha_1}; \F_2) \oplus \cdots \oplus H_1(Y_{\alpha_g}; \F_2) \oplus H_1(Y_{\beta_g}; \F_2) \oplus \cdots \oplus H_1(Y_{\beta_1}; \F_2).
\end{align*}
Clearly these induce well-defined mod $2$-homology classes in $Y^\theta$, which we will still denote by $\omega$, $\omega_{\alpha_k}$, and $\omega_{\beta_k}$. We therefore get a Cerf decomposition with homology class
\begin{equation}
	(Y^\theta, \omega) = (Y_{\alpha_1}^\theta, \omega_{\alpha_1}) \cup_{\Sigma_{1,3}} \cdots \cup_{\Sigma_{g-1,3}} (Y_{\alpha_g}^\theta, \omega_{\alpha_g}) \cup_{\Sigma_{g,3}} (Y_{\beta_g}^\theta, \omega_{\beta_g}) \cup_{\Sigma_{g - 1,3}} \cdots \cup_{\Sigma_{1,3}} (Y_{\beta_1}^\theta, \omega_{\beta_1}).
	\label{eqn:cerf-3mfd}
\end{equation}
We wish to simplify the generalized Lagrangian correspondence associated to this Cerf decomposition. The first thing to note is that we can take almost all of the homology classes to be zero. Indeed, by applying the Mayer-Vietoris sequence to the Heegaard splitting $Y = H_\alpha \cup_{\Sigma_g} H_\beta$, one sees that both $i_\ast^\alpha: H_1(H_\alpha; \F_2) \longrightarrow H_1(Y; \F_2)$ and $i_\ast^\beta: H_1(H_\beta; \F_2) \longrightarrow H_1(Y; \F_2)$ are surjective, so that one may take $\omega$ to be represented by a circle lying entirely in just one of the handlebodies, say $H_\alpha$. Furthermore, $i_\ast: H_1(\Sigma_g; \F_2) \longrightarrow H_1(H_\alpha; \F_2)$ is also a surjection, so we can further restrict to the case where $\omega$ is represented by a curve in a collar neighborhood of the boundary of $H_\alpha$. This implies that we can insert a cylinder $(\Sigma_{g,3} \times [0,1], \omega)$ between $(Y_{\alpha_g}^\theta, \omega_{\alpha_g})$ and $(Y_{\beta_g}^\theta, \omega_{\beta_g})$ in the Cerf decomposition (\ref{eqn:cerf-3mfd}) and set all the other homology classes equal to zero. The Lagrangian correspondence associated to $(\Sigma_{g,3} \times [0,1], \omega)$ turns out to have a simple description:

\begin{lem}
Let $W^\theta = \Sigma_{g,3} \times [0,1]$ be a trivial cobordism with a trivial $3$-stranded tangle removed, and suppose $\omega \in H_1(W^\theta; \F_2)$. Then
\[
	L(W^\theta, \omega) = \{([\rho], [\rho^\prime]) \in \scrR_{g,3} \times \scrR_{g,3} \mid \rho^\prime(\gamma) = (-1)^{\gamma \cdot \omega} \rho(\gamma) \text{ for each } \gamma\},
\]
where $\gamma$ is any smooth curve in $\Sigma_{g,3}$ and $\gamma \cdot \omega$ is the mod $2$ intersection number of $\gamma$ with the projection of a geometric representative for $\omega_\alpha$ to $\Sigma_{g,3}$.
\label{lem:lagr-omega}
\end{lem}

\begin{proof}
$L(W^\theta, \omega)$ consists of pairs of (gauge equivalence classes of) connections $(A_0, A_1) \in \scrR_{g,3} \times \scrR_{g,3}$ which are the boundary values of a connection $\bfA$ on $W^\theta$ that has traceless holonomy around each strand of the trivial tangle and holonomy $-I$ around the meridian of the curve $\omega$. For any based curve $\gamma$ in $\Sigma_{g,3}$ (with basepoint missing $\omega$), consider $\gamma \times [0,1] \subset W^\theta$ as a rectangle. This rectangle intersects $\omega$ exactly $\gamma \cdot \omega$ times. Therefore the holonomy of $\bfA$ around the boundary of this square is $(-I)^{\gamma \cdot \omega}$. This holonomy is also equal to $\Hol_{\gamma \times \{1\}}(\bfA)\Hol_{\gamma \times \{0\}}(\bfA)^{-1}$, and therefore
\[
	\Hol_{\gamma \times \{1\}}(\bfA) = (-1)^{\gamma \cdot \omega} \Hol_{\gamma \times \{0\}}(\bfA)
\]
for any curve $\gamma$ in $\Sigma_{g,3}$.
\end{proof}

The following theorem shows that in the current setup, we can associate genuine Lagrangian submanifolds of $\scrR_{g,3}$ to $H_\alpha$ and $H_\beta$, not just generalized Lagrangian correspondences $\text{pt} \longrightarrow \scrR_{g,3}$.

\begin{thm}
The geometric compositions
\[
	L_\alpha^{\omega} = L(Y_{\alpha_1}^\theta) \circ \cdots \circ L(Y_{\alpha_g}^\theta) \circ L(\Sigma_{g,3} \times [0,1], \omega)
\]
and
\[
	L_\beta^\omega = L(\Sigma_{g,3} \times [0,1], \omega) \circ L(Y_{\beta_g}^\theta) \circ \cdots \circ L(Y_{\beta_1}^\theta)
\]
are embedded.
\end{thm}

\begin{proof}
The proof that the geometric composition
\[
	L_\alpha = L(Y_{\alpha_1}^\theta) \circ \cdots \circ L(Y_{\alpha_g}^\theta)
\]
is embedded is already known from Section \ref{chap:SI-cerf}. $L_\alpha$ has the holonomy description
\[
	L_\alpha = \{[\rho] \in \scrR_{g,3} \mid \rho(\alpha_k) = I \text{ for each } k\}.
\]
Then it is clear by Lemma \ref{lem:lagr-omega} that
\begin{align*}
	L_\alpha^\omega & = L_\alpha \circ L(\Sigma_{g,3} \times [0,1], \omega) \\
		 & = \{[\rho] \in \scrR_{g,3} \mid \rho(\alpha_k) = (-1)^{\alpha_k \cdot \omega} I \text{ for each } k\}
\end{align*}
and that this composition is embedded. The analogous statement for $L_\beta^\omega$ follows in the same way.
\end{proof}

The holonomy description for $L_\alpha^\omega$ (and the obvious analogue for $L_\beta^\omega$) is useful in its own right, so we record it separately here:

\begin{prop}
Given $\omega \in H_1(Y; \F_2)$ and a Heegaard splitting $Y = H_\alpha \cup_{\Sigma_g} H_\beta$, represent $\omega$ by a smooth loop in a bicollar neighborhood of the Heegaard surface, which we may assume lies on a parallel copy of $\Sigma_g$. Then
\[
	L_\alpha^\omega = \{[\rho] \in \scrR_{g,3} \mid \rho(\alpha_k) = (-1)^{\alpha_k \cdot \omega} I \text{ for each } k\},
\]
\[
	L_\beta^\omega = \{[\rho] \in \scrR_{g,3} \mid \rho(\beta_k) = (-1)^{\beta_k \cdot \omega} I \text{ for each } k\},
\]
where $\alpha_k \cdot \omega$ is the mod $2$ intersection number of $\alpha_k$ with the projection of the geometric representative for $\omega$ to $\Sigma_g$, and similarly for $\beta_k \cdot \omega$.
\label{thm:lagr-omega}
\end{prop}

As a consequence of the above, when $\SI(Y, \omega)$ is computed from a Heegaard splitting, it can be computed as a classical Lagrangian Floer homology group for two Lagrangians in $\scrR_{g,3}$, rather than the quilted Floer homology group of a pair of generalized Lagrangian correspondences. Furthermore, we may always think of $\omega \in H_1(Y; \F_2)$ as represented by a curve in a parallel copy of the Heegaard surface, and incorporate it either into the Lagrangian for the $\alpha$-handlebody or the $\beta$-handlebody. Hence we may compute the symplectic instanton homology as
\[
	\SI(Y, \omega) = \HF(L_\alpha^\omega, L_\beta) = \HF(L_\alpha, L_\beta^\omega),
\]
so that only one of the Lagrangians is different from the $\omega = 0$ case\hide{ in \cite{horton1}}, and the way it is different is described exactly as in Proposition \ref{thm:lagr-omega}.

Note that Theorem \ref{thm:lagrangianS3}\hide{4.2 of \cite{horton1}} (which says that $L_\alpha$ is a copy of $(S^3)^g$ for any set of attaching curves $\bfalpha$) applies equally well to $L_\alpha^\omega$, so that $L_\alpha^\omega$ is a Lagrangian $(S^3)^g$ in $\scrR_{g,3}$, and furthermore this identification is achieved by conjugating $[A_1, \dots, B_g, C_1, C_2, C_3] \in L_\alpha^\omega$ to $[A_1^\prime, \dots, B_g^\prime, \bfi, \bfj, -\bfk]$ (the $(A_1^\prime, \dots, B_g^\prime)$ describe a $g$-fold product of $3$-spheres when varying over all $[A_1, \dots, B_g, C_1, C_2, C_3] \in L_\alpha$). This identification says that we can consider $L_\alpha^\omega$ as $\SU(2)$-representations of the $\alpha$-handlebody minus $\omega$ sending the meridian of $\omega$ to $-I$, \emph{without} modding out by conjugation. This in turn leads to the following interpretation of $L_\alpha^\omega \cap L_\beta$:

\begin{thm}
\label{thm:omega-intersection}
For any two sets of attaching curves $\bfalpha$ and $\bfbeta$ and homology class $\omega \in H_1(Y; \F_2)$, we have that
\[
	L_\alpha^\omega \cap L_\beta \cong \{\rho: \pi_1(Y \setminus \omega) \longrightarrow \SU(2) \mid \rho(\mu_\omega) = -I\}.
\]
\end{thm}

\begin{proof}
Let $H_\alpha$ and $H_\beta$ denote the $\alpha$- and $\beta$-handlebodies, respectively. Then the discussion above shows that
\[
	L_\alpha = \{\rho: \pi_1(H_\alpha \setminus \omega) \longrightarrow \SU(2) \mid \rho(\mu_\omega) = -I\}, \quad\quad L_\beta = \Hom(\pi_1(H_\beta), \SU(2)),
\]
where we do not mod out by the action of conjugation. Furthermore, wee see that
\[
	\left\{[A_1, B_1, \dots, A_g, B_g, C_1, C_2, C_3] \in \scrR_{g,3} ~\left|~ \prod_{k = 1}^g [A_k, B_k] = I\right\}\right. = \Hom(\pi_1(\Sigma_g), \SU(2)),
\]
and $L_\alpha^\omega$, $L_\beta$ always lie entirely inside this subset of $\scrR_{g,3}$. Therefore by the Seifert-Van Kampen theorem an intersection point of $L_\alpha^\omega$ and $L_\beta$ corresponds to an element of $\{\rho: \pi_1(Y \setminus \omega) \longrightarrow \SU(2) \mid \rho(\mu_\omega) = -I\}$.
\end{proof}

We also remark that the proof of naturality for $\SI(Y)$ in Section \ref{sect:invariance}\hide{\cite[Section 6]{horton1}} directly translates to a proof of naturality for $\SI(Y, \omega)$, resulting in the following:

\begin{thm}
$\SI(Y, \omega)$ is a natural invariant of the pair $(Y, \omega)$, in that one can pin down $\SI(Y, \omega)$ as a concrete group as opposed to an isomorphism class of groups. 
\end{thm}

%% file: surgery-triangle.tex
\section{Exact Triangle for Surgery Triads}

Now that we have extended symplectic instanton homology to take into account nontrivial $\SO(3)$-bundles, we can properly state and prove the exact triangle for Dehn surgery on a knot.

\subsection{$3$-Manifold Triads and the Statement}

Suppose $\widetilde{Y}$ is a compact, oriented $3$-manifold with torus boundary. Given three oriented simple closed curves $\gamma$, $\gamma_0$, and $\gamma_1$ in $\partial \widetilde{Y}$ satisfying
\[
	\#(\gamma \cap \gamma_0) = \#(\gamma_0 \cap \gamma_1) = \#(\gamma_1 \cap \gamma) = -1,
\]
we may form three closed, oriented $3$-manifolds $Y$, $Y_0$, and $Y_1$ by Dehn filling $\widetilde{Y}$ along $\gamma$, $\gamma_0$, and $\gamma_1$, respectively. If a triple of closed, oriented $3$-manifolds $(Y, Y_0, Y_1)$ arises in this way from \emph{some} $\widetilde{Y}$, we call $(Y, Y_0, Y_1)$ a {\bf surgery triad}.

It is easy to see that $(Y, Y_0, Y_1)$ is a surgery triad if and only if there is a framed knot $\bbK = (K, \lambda)$ in $Y$ such that $Y_0 = Y_\lambda$ and $Y_1 = Y_{\lambda + \mu}$, where $\mu$ is the meridian of $K$. Note that surgery triads are cyclic, if that if $(Y, Y_0, Y_1)$ is a surgery triad, then so are $(Y_0, Y_1, Y)$ and $(Y_1, Y, Y_0)$. If we want to emphasize this particular framed knot $\bbK$, we will say that $(Y, Y_0, Y_1)$ is a surgery triad {\bf relative to $(K, \lambda)$}.

If $\bbK = (K, \lambda)$ is a framed knot in a closed, oriented $3$-manifold $Y$ with meridian $\mu$ and $p, q$ are integers, we write
\[
	Y_{p/q}(\bbK) = Y_{p\mu + q\lambda}(\bbK)
\]
for the result of removing $\nu K$ from $Y$ and gluing back in a solid torus such that the curve $p\mu + q\lambda$ in $\partial (\nu K)$ bounds a disk.

If $K$ is a nullhomologous knot in a closed, oriented $3$-manifold $Y$, there is a unique framing $\lambda_\text{Seifert}$ for $K$ that is nullhomologous in $Y \setminus \nu K$ which is called the {\bf Seifert framing} of $K$. In this case, we just write $Y_{p/q}(K)$ for $Y_{p,q}(K, \lambda_\text{Seifert})$, where the lack of framing in the notation means we are using the Seifert framing by default.

We can form several interesting families of surgery triads:

\begin{exmp}
\label{ex:n-triad}
Given a framed knot $\bbK$ in a closed, oriented $3$-manifold $Y$ and any integer $n$, both $(Y, Y_n(\bbK), Y_{n+1}(\bbK))$ and $(Y_0(\bbK), Y_{1/(n+1)}(\bbK), Y_{1/n}(\bbK))$ are surgery triads.
\end{exmp}

\begin{exmp}
More generally, given a framed knot $\bbK$ in a closed, oriented $3$-manifold $Y$ and relatively prime integers $p_1, q_1$, B\'ezout's identity allows us to find integers $p_2$, $q_2$ satisfying $p_1 q_2 - p_2 q_1 = 1$. If we write $p_3 = p_1+p_2$ and $q_3 = q_1+q_2$, then $(Y_{p_1/q_1}(\bbK), Y_{p_2/q_2}(\bbK), Y_{p_3/q_3}(\bbK))$ is a surgery triad.
\end{exmp}

\begin{exmp}
\label{ex:bdc}
Let $L$ be a link in $S^3$ and fix a planar diagram $\calD$ for $L$. Given a crossing in $\calD$, we may resolve it in one of two ways (see Figure \ref{fig:skein}) to obtain links $L_0$ and $L_1$ differing from $L$ only at the chosen crossing in the prescribed manner. Then the branched double covers of these three links, $(\Sigma(L), \Sigma(L_0), \Sigma(L_1))$ form a surgery triad. We will discuss surgery triads of this form in more detail in \cite{horton2} when establishing a relation between $\SI(\Sigma(L))$ and the Khovanov homology of the mirror of $L$.
\end{exmp}

\begin{figure}[h]
	\centering
	\includegraphics[scale=1.5]{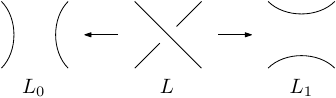}
	\caption{$0$- and $1$-resolutions of a crossing in a link.}
	\label{fig:skein}
\end{figure}

In general, in Floer theories for $3$-manifolds one expects an exact triangle relating the Floer homologies of any three $3$-manifolds fitting into a surgery triad. For symplectic instanton homology, we will prove the following:

\begin{thm}
\label{thm:surgery-triangle}
For any $3$-manifold triad $(Y, Y_0, Y_1)$ relative to a framed knot $(K,\lambda)$ in $Y$ and $\SO(3)$-bundle $P \longrightarrow Y$, there is an exact triangle of symplectic instanton homology groups:
\[
	\xymatrix{\SI(Y, \omega + \omega_K) \ar[rr] & & \SI(Y_0, \omega_0) \ar[dl] \\ & \SI(Y_1, \omega_1) \ar[ul] & }
\]
Here $\omega$, $\omega_0$, and $\omega_1$ are the mod $2$ homology classes Poincar\'e dual to the bundles induced by $P$ on $Y$, $Y_\lambda$, and $Y_{\lambda + \mu}$:
\[
	\omega = \PD(w_2(P)) \in H_1(Y; \F_2), \quad\quad\quad \omega_\lambda = i_\ast^0 \PD(i^\ast_{Y \setminus \nu K} w_2(P)) \in H_1(Y_0; \F_2),
\]
\[
	\omega_1 = i_\ast^1 \PD(i^\ast_{Y \setminus \nu K} w_2(P)) \in H_1(Y_1; \F_2).
\]
\end{thm}

\subsection{Seidel's Exact Triangle and Quilted Floer Homology}

\label{subsect:triangle}

Although the surgery exact triangle will essentially be an application of Seidel's exact triangle for symplectic Dehn twists \cite{seidel-triangle,fiberedtriangle,mak-wu} in the setting of quilted Floer homology for monotone Lagrangians, we recall some relevant ideas in the proof, as they will be needed later when we establish the link surgeries spectral sequence in \cite{horton2}. 

In what follows, we fix:
\begin{itemize}
	\item Two monotone symplectic manifolds $M$ and $M'$ with the same monotonicity constant.
	\item A monotone Lagrangian submanifold $L_0: \text{pt} \longrightarrow M$ (which we think of as a Lagrangian correspondence).
	\item A monotone Lagrangian correspondence $\underline{L}: M \longrightarrow M'$.
	\item Two monotone Lagrangian submanifolds $L_1, V: M' \longrightarrow \text{pt}$, where $V$ is topologically a sphere.
\end{itemize}
Our version of Seidel's exact triangle will be constructed from a sequence of maps
\[
	\CF(L_0, \underline{L}, V) \otimes \CF(V^T, L_1) \xrightarrow{~C\Phi_0~} \CF(L_0, \underline{L}, L_1) \xrightarrow{~C\Phi_1~} \CF(L_0, \underline{L}, \tau_V L_1).
\]
On the chain level, the maps appearing in the exact sequence are defined as follows. The first map
\[
	C\Phi_0: \CF(L_0, \underline{L}, V) \otimes \CF(V^T, L_1) \longrightarrow \CF(L_0, \underline{L}, L_1)
\]
is the relative invariant counting quilted pseudoholomorphic triangles as pictured in Figure \ref{fig:CPhi0}. The second map
\[
	C\Phi_1: \CF(L_0, \underline{L}, L_1) \longrightarrow \CF(L_0, \underline{L}, \tau_V L_1)
\]
is the relative invariant counting pseudoholomorphic sections of the quilted Lefschetz fibration in Figure \ref{fig:CPhi1}.

\begin{figure}[h]
\centering
\begin{minipage}{.45\textwidth}
	\includegraphics[scale=.95]{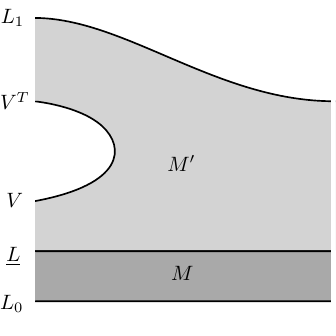}
	\caption{Quilts counted by the map $C\Phi_0$.}
	\label{fig:CPhi0}
\end{minipage}
\begin{minipage}{.45\textwidth}
	\includegraphics[scale=.95]{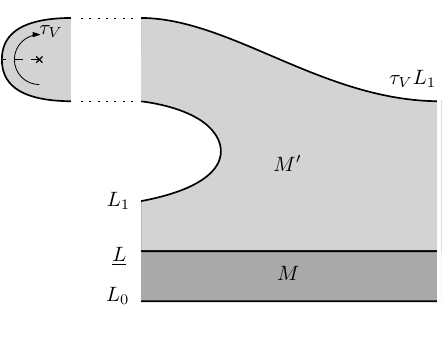}
	\caption{Quilted Lefschetz fibration defining $C\Phi_1$.}
	\label{fig:CPhi1}
\end{minipage}
\end{figure}

There is a distinguished Floer chain
\[
	c \in \CF(L_1, \tau_V L_1),
\]
\[
	\mu_1(c) = 0
\]
defined as follows. Let $E \longrightarrow \mathbb{D}$ be the standard Lefschetz fibration over the disk with regular fiber $M'$ and vanishing cycle $V$. By removing a point $z^c \in \partial \mathbb{D}$, we get a Lefschetz fibration with strip-like end $E^c \longrightarrow \mathbb{D}^c$. Denote the strip-like end by $\epsilon_{z^c}: \R^+ \times [0,1] \longrightarrow \mathbb{D}^c$. We equip $E^c \longrightarrow \mathbb{D}^c$ with the following moving Lagrangian boundary conditions: For $s \gg 0$, identify the fibers over $\epsilon_{z^c}(s,0)$ with $L_1$. As we leave the strip-like end and travel along $\partial \mathbb{D}^c$, we carry $L_1$ along by parallel transport. Once we reach the other side of the strip-like end, our parallel-transported Lagrangian will be isotopic to $\tau_V L_1$; the boundary condition along $\epsilon_{z^c}(s, 1)$, $s \gg 0$ realizes such an isotopy.

The next ingredient for the exact triangle is an explicit chain nullhomotopy of the triangle product with $c$:
\[
	k: \CF(V, L_1) \longrightarrow \CF(V, \tau_V L_1),
\]
\[
	\mu_1(k(\cdot)) + k(\mu_1(\cdot)) + \mu_2(\cdot, c) = 0.
\]

\begin{figure}[h]
	\centering
	\includegraphics[scale=.9]{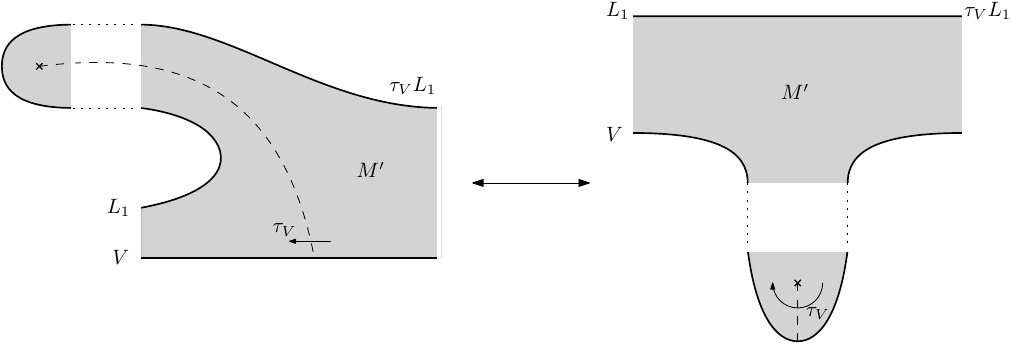}
	\caption{The $1$-parameter family of Lefschetz fibrations defining the chain nullhomotopy $k$.}
	\label{fig:nullhomotopy-k}
\end{figure}

The construction of $k$ is pictured schematically in Figure \ref{fig:nullhomotopy-k}. The meaning of the left and right surfaces with strip-like ends should be evident; the dotted lines connecting boundary components indicate gluings of components using sufficiently long gluing lengths, so that the associated relative invariant of the glued fibrations corresponds to the composition of the individual relative invariants. In particular, the left-hand surface represents $\mu_2(\cdot, c)$, and the right-hand surface represents a triangle glued with a Floer chain which is equal to zero (see \cite[Corollary 4.31]{fiberedtriangle}, for example). There is a clear $1$-parameter family of Lefschetz fibrations $E^{k,r} \longrightarrow S^{k,r}$ ($0 \leq r \leq 1$) over the strip, equal to the left of Figure \ref{fig:nullhomotopy-k} for $r = 0$ and equal to the right of Figure \ref{fig:nullhomotopy-k} for $r = 1$, achieved by moving the position of singular fiber. The nullhomotopy $k$ is then defined by counting isolated points of the parametrized moduli space of pseudoholomorphic sections of $E^{k,r} \longrightarrow S^{k,r}$.

Finally, we may use $k$ to construct the map
\[
	h: \CF(L_0, \underline{L}, V) \otimes \CF(V^T, L_1) \longrightarrow \CF(L_0, \underline{L}, \tau_V L_1),
\]
\[
	h(x \otimes y) = \tilde{\mu}_2(x, k(y)) + \tilde{\mu}_3(x,y, c),
\]
where $\tilde{\mu}_2$ (resp. $\tilde{\mu}_3$) is the quilted triangle (resp. quilted rectangle) map pictured in Figure \ref{fig:mu2tilde} (resp. Figure \ref{fig:mu3tilde}). By a routine calculation, one may show that $h$ defines a chain nullhomotopy of $C\Phi_1 \circ C\Phi_0$.

\begin{figure}[h]
	\centering
\begin{minipage}{.45\textwidth}
	\hspace{.15in}
	\includegraphics{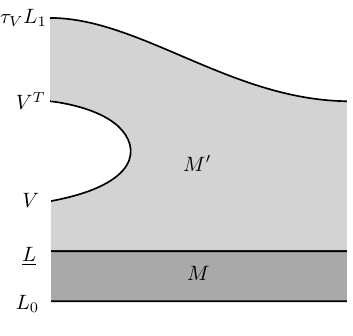}
	\caption{Quilted triangle map $\tilde{\mu}_2$.}
	\label{fig:mu2tilde}
\end{minipage}
\begin{minipage}{.45\textwidth}
	\hspace{.25in}
	\includegraphics{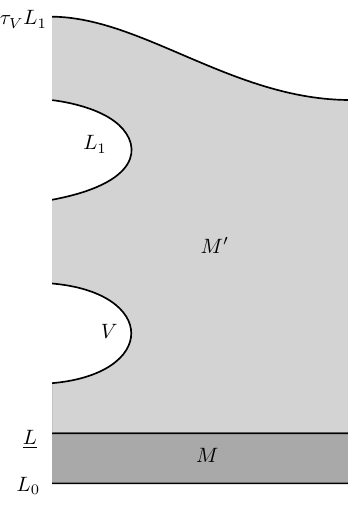}
	\caption{Quilted rectangle map $\tilde{\mu}_3$.}
	\label{fig:mu3tilde}
\end{minipage}
\end{figure}

The maps $C\Phi_0$, $C\Phi_1$, and $h$ turn out to fit into a general homological algebraic framework that allow one to establish the existence of an exact triangle. We recall some terminology necessary for the statement of the next result. An {\bf $\R$-graded $\Lambda$-module} is an $\Lambda$-module $M$ with a direct sum decomposition
\[
	M = \bigoplus_{s \in \R} M_s.
\]
The {\bf support} of $M$ is the set of real numbers
\[
	\mathrm{supp}(M) = \{s \in \R \mid M_s \neq 0\}.
\]
Given an interval $I \subset \R$, we say that $M$ has {\bf gap} $I$ if $s, t \in \mathrm{supp}(M)$ implies that $|s - t| \notin I$. Finally, a $\Lambda$-linear map $f: M \longrightarrow M'$ of $\R$-graded $R$-modules has {\bf order $I$} if
\[
	f(M_s) \subset \bigoplus_{t \in I} M_{s+t}^\prime
\]
for all $s \in \R$.
 
\begin{lem}
\textup{(Double Mapping Cone Lemma \cite[Lemma 5.4]{symp-gysin} \cite[Lemma 5.9]{fiberedtriangle})} Let $\varepsilon > 0$ and suppose that $(E_0, \delta_0)$, $(E_1, \delta_1)$, and $(E_2, \delta_2)$ are free, finitely generated chain complexes of $\R$-graded $\Lambda$-modules. Suppose we have chain maps $f: E_0 \longrightarrow E_1$ and $g: E_1 \longrightarrow E_2$ and a chain nullhomotopy $h: E_0 \longrightarrow E_2$ of $g \circ f$. Assume that this data satisfies the following conditions:
\begin{itemize}
	\item[(1)] $\delta_0$, $\delta_1$, and $\delta_2$ all have order $[2\varepsilon, \infty)$.
	\item[(2)] We have decompositions $f = f_0 + f_1$, $g = g_0 + g_1$, and $h = h_0 + h_1$ such that $f_0$ has order $[0, \varepsilon)$, $g_0$ and $h_0$ have order $0$, and $f_1$, $g_1$, and $h_1$ all have order $[2\varepsilon, \infty)$.
	\item[(3)] $0 \longrightarrow E_0 \xrightarrow{~f_0~} E_1 \xrightarrow{~g_0~} E_2 \longrightarrow 0$ is a short exact sequence of abelian groups.
\end{itemize}
Then $(h, g): \Cone(f) \longrightarrow E_2$ is a quasi-isomorphism.
\label{lem:double-cone}
\end{lem}

To relate this to our intended application, we need to explain how to give the Floer chain complexes $\R$-gradings. Given a ring $R$ (which will be $\Z$ or $\Z/2$ for us), let $\Lambda_R$ denote the ring
\[
	\Lambda_R = \left\{\sum_{k = 1}^n r_k q^{s_k} \mid r_k \in R, s_k \in \R\right\},
\]
where $q$ is a formal variable and the ring multiplication is defined on monomials by
\[
	(rq^s)(r'q^{s^\prime}) = (rr')q^{s + s^\prime}.
\]
$\Lambda_R$ admits the obvious $\R$-grading
\[
	\Lambda_R = \bigoplus_{s \in \R} \{rq^s \mid r \in R\}.
\]

Given two simply connected, transversely intersecting monotone Lagrangian submanifolds $L_0$, $L_1$ of a monotone symplectic manifold $M$, we define the {\bf Lagrangian Floer chain complex with $\Lambda_R$ coefficients} to have the obvious chain groups
\[
	\CF(L_0, L_1; \Lambda_R) = \bigoplus_{x \in L_0 \cap L_1} \Lambda_R \langle x \rangle
\]
but the modified differential defined on generators by
\[
	\partial_{\Lambda_R} x = \sum_{\substack{y \in L_0, L_1 \\ u \in \mathcal{M}(x,y)_0}} o(u)q^{E(u)}y,
\]
where $o(u)$ is the local orientation of $\mathcal{M}(x,y)_0$ at $u$ ($\pm 1$ if using $R = \Z$, $1$ if using $R = \Z/2$), and $E(u)$ is the energy of the pseudoholomorphic strip $u$. There is an obvious extension of the Floer chain complex with $\Lambda_R$ coefficients to the setting of quilted Floer homology, where the energy of a pseudoholomorphic quilt is simply the sum of the energies of each of its components.

The maps $C\Phi_0$, $C\Phi_1$, $k$, and $h$ all have obvious extensions to $\Lambda_R$ coefficients by incorporating the energy of the quilted pseudoholomorphic maps/sections they count into their definitions. One may show that these modified maps satisfy the hypotheses of Lemma \ref{lem:double-cone}, and it follows that $(h, C\Phi_1): \Cone(C\Phi_0) \longrightarrow \CF(L_0, \underline{L}, \tau_V L_1; \Lambda_R)$ is a quasi-isomorphism. Since mapping cones naturally fit into exact triangles, we obtain the desired exact triangle in Floer homology with $\Lambda_R$ coefficients.

To obtain the desired result for $R = \Z$ or $R = \Z/2$ coefficients, we note that the natural map
\[
	\HF(L_0, L_1; \Lambda_R)/(q - 1) \longrightarrow \HF(L_0, L_1)
\]
is an isomorphism (cf. \cite[Remark 4.3]{fiberedtriangle}). Therefore we conclude the following:

\begin{prop}
\label{prop:seidel-triangle}
Suppose $L_0$ is a monotone Lagrangian submanifold of the monotone symplectic manifold $M$, $L_1$ and $V$ are monotone Lagrangian submanifolds of the monotone symplectic manifold $M'$ with $V$ a sphere, and that $\underline{L}$ is a monotone generalized Lagrangian correspondence $\underline{L}: M \longrightarrow M'$. Then there is an exact triangle of quilted Floer homology groups
\[
	\xymatrix{\HF(L_0, \underline{L}, V) \otimes \HF(V^T, L_1) \ar[rr]^{\Phi_0} & & \HF(L_0, \underline{L}, L_1) \ar[dl]^{\Phi_1} \\
	 & \HF(L_0, \underline{L}, \tau_V L_1) \ar[ul] & } 
\]
\end{prop}

\subsection{The Surgery Exact Triangle}

\label{sect:triangle}

We now use Seidel's exact triangle, in the form of Proposition \ref{prop:seidel-triangle}, to establish the surgery exact triangle in symplectic instanton homology.

Let $Y$ be a closed, oriented $3$-manifold. Given a framed knot $\bbK = (K, \lambda)$ in $Y$, we may choose a Heegaard triple $(\Sigma_{g+1}, \bfalpha, \bfbeta, \bfgamma, z)$ subordinate to a bouquet for $\bbK$ (cf. \cite[Section 8.3]{horton1}). We may also choose a Heegaard triple $(\Sigma_{g+1}, \bfalpha, \bfbeta, \bfdelta, z)$ subordinate to $(K, \lambda + \mu)$. The quadruple diagram $\calH = (\Sigma_{g+1}, \bfalpha, \bfbeta, \bfgamma, \bfdelta, z)$ satisfies the following conditions:
\begin{itemize}
	\item $\beta_{g+1}$ is a meridian for $K$.
	\item $\gamma_{g+1}$ represents the framing $\lambda$ of $K$ and $\delta_{g+1}$ represents the framing $\lambda + \mu$ of $K$.
	\item $(\Sigma_{g+1}, \{\alpha_1, \dots, \alpha_{g+1}\}, \{\beta_1, \dots, \beta_{g}\}, z)$ represents the complement of $K$ in $Y$.
	\item $\gamma_k = \beta_k$ and $\delta_k = \beta_k$ for $k \neq g+1$.
	\item $\calH_{\alpha\beta}$ represents $Y$, $\calH_{\alpha\gamma}$ represents $Y_\lambda$, and $\calH_{\alpha\delta}$ represents $Y_{\lambda + \mu}$.
\end{itemize}

In particular, $\calH_{\alpha\beta}$, $\calH_{\alpha\gamma}$, and $\calH_{\alpha\delta}$ induce length $2g+2$ Cerf decompositions of $Y^\theta$ where we first attach $(g+1)$ $1$-handles with attaching cycles in $\Sigma_{g+1}$ given by $\bfalpha$, and then we attach $(g+1)$ $2$-handles with attaching cycles in $\Sigma_{g+1}$ given by $\bfbeta$, $\bfgamma$, or $\bfdelta$ (depending on which diagram we are considering). The conditions on $\bfbeta$, $\bfgamma$, and $\bfdelta$ imply that these three Cerf decompositions differ only in the final $2$-handle attachment.

The final $2$-handle attachments may be represented by Lagrangian submanifolds $L_{\beta_{g+1}}$, $L_{\gamma_{g+1}}$, and $L_{\delta_{g+1}}$ of $\scrR_{1,3}$. Now note that we in fact have $\delta_{g+1} = \tau_{\beta_{g+1}} \gamma_{g+1}$, so that if
\[
	V = \{[\rho] \in \scrR_{1,3} \mid \rho(\beta_{g+1}) = -I\},
\]
then on the level of Lagrangians
\[
	L_{\delta_{g+1}} = \tau_V L_{\gamma_{g+1}}.
\]
Writing $\underline{L}$ for the Lagrangian correspondence coming from the first $g$ $\beta$- (equivalently, $\gamma$- or $\delta$-) handle attachments, Seidel's exact triangle reads
\[
	\xymatrix{\HF(L_\alpha, \underline{L}, V) \otimes \HF(V^T, L_{\gamma_{g+1}}) \ar[rr]^{\Phi_0} & & \HF(L_\alpha, \underline{L}, \tau_V L_{\gamma_{g+1}}) \ar[dl]^{\Phi_1} \\
	 & \HF(L_\alpha, \underline{L}, L_{\gamma_{g+1}}) \ar[ul] & } 
\]
Two of these Floer homology groups may be immediately identified:
\[
	\HF(L_\alpha, \underline{L}, L_{\gamma_{g+1}}) = \SI(Y_{\lambda}), \quad\quad \HF(L_\alpha, \underline{L}, \tau_V L_{\gamma_{g+1}}) = \SI(Y_{\lambda+\mu}).
\]
The remaining Floer homology group is identified thusly:
\begin{align*}
	\HF(L_\alpha, \underline{L}, V) \otimes \HF(V^T, L_{\gamma_{g+1}}) = \SI(Y, \omega_K) \otimes \Z = \SI(Y, \omega_K),
\end{align*}
where $\HF(L_\alpha, \underline{L}, V) = \SI(Y, \omega_K)$ since $\beta_{g+1}$ is a meridian of $K$ and $V$ precisely represents pairs of flat connections in $\scrR_{1,3}^- \times \scrR_{0,3}$ which simultaneously extend to flat connections on the nontrivial $\SO(3)$ bundle over the bordism $\Sigma_{1,3} \longrightarrow \Sigma_{0,3}$ whose second Stiefel-Whitney class is Poincar\'e dual to $K$, and $\HF(V^T, L_{\gamma_{g+1}}) \cong \Z$ because the only intersection point of the two Lagrangians is $[-I, \pm I, \bfi, \bfj, -\bfk] \in \scrR_{1,3}$ (where the $\pm$ is determined by the framing $\lambda$).

It follows that we have a surgery exact triangle
\[
	\xymatrix{\SI(Y, \omega_K) \ar[rr] & & \SI(Y_\lambda) \ar[dl] \\
	 & \SI(Y_{\lambda + \mu}) \ar[ul] & } 
\]
In fact, given a nontrivial $\SO(3)$-bundle $P \longrightarrow Y$, we get three induced homology classes Poincar\'e dual to the induced bundles on $Y$, $Y_\lambda$, and $Y_{\lambda + \mu}$:
\[
	\omega = \PD(w_2(P)) \in H_1(Y; \F_2), \quad\quad\quad \omega_\lambda = i_\ast^\lambda \PD(i^\ast_{Y \setminus \nu K} w_2(P)) \in H_1(Y_\lambda; \F_2),
\]
\[
	\omega_{\lambda+\mu} = i_\ast^{\lambda+\mu} \PD(i^\ast_{Y \setminus \nu K} w_2(P)) \in H_1(Y_{\lambda+\mu}; \F_2).
\]
These may all be represented by embedded curves disjoint from $\nu K$ (indeed, they can be represented by curves lying entirely in the $\alpha$-handlebody), so in the above setup we may incorporate nontrivial bundles into the generalized Lagrangian correspondence $\underline{L}$ without any modification to the argument. Hence we have more generally an exact triangle
\[
	\xymatrix{\SI(Y, \omega + \omega_K) \ar[rr] & & \SI(Y_\lambda, \omega_\lambda) \ar[dl] \\
	 & \SI(Y_{\lambda + \mu}, \omega_{\lambda+\mu}) \ar[ul] & } 
\]

\begin{rem}
Recall that symplectic instanton homology $\SI(Y, \omega) = \HF(L_\alpha^{\omega}, L_\beta)$ admits a $\Z/2$-grading as follows. Assume (perhaps after perturbation) that $L_\alpha^\omega$ and $L_\beta$ intersect transversely, and fix orientations of $L_\alpha^\omega$, $L_\beta$, and $\scrR_{g,3}$, so that for any $x \in L_\alpha^\omega \cap L_\beta$, the local intersection number $(L_\alpha^\omega \cdot L_\beta)_x \in \{\pm 1\}$ is well-defined. Considered as a generator of $\CF(L_\alpha^\omega, L_\beta)$, the $\Z/2$-grading of $x$ is $0$ if $(L_\alpha^\omega \cdot L_\beta)_x = +1$ and $1$ otherwise. The maps in Seidel's exact triangle have well-defined gradings; in our context the gradings are as indicated below:
\[
	\xymatrix{\SI(Y, \omega + \omega_K) \ar[rr]^{[0]} & & \SI(Y_\lambda, \omega_\lambda) \ar[dl]^{[0]} \\
	 & \SI(Y_{\lambda + \mu}, \omega_{\lambda+\mu}) \ar[ul]^{[1]} & } 
\]
\end{rem}

%% file: computations.tex
\section{Computations}

\label{sect-comp}

In this section, we give a formula for the Euler characteristic of $\SI(Y)$, and compute $\SI(Y)$ for $S^3$, $S^2 \times S^1$, lens spaces, and connected sums. We also define a concept of an ``instanton $L$-space'' and give several infinite families of examples.

\subsection{Connected Sums}

In Floer homology theories for $3$-manifolds, connected sums of $3$-manifolds tend to correspond to tensor products of Floer chain groups. This is indeed the case for symplectic instanton homology as well: we can generalize the proof of Theorem \ref{thm:stabilization-inv} (which essentially says $\CSI(Y \# S^3) \cong \CSI(Y) \otimes \CSI(S^3) \cong \CSI(Y)$) to establish the behavior of symplectic instanton chain groups under connected sum of $3$-manifolds.

\begin{thm}
\label{thm:connectsum}
For any closed, oriented $3$-manifolds $Y$ and $Y'$ equipped with $\SO(3)$-bundles $\omega$ and $\omega'$, the symplectic instanton chain complex for $(Y \# Y', \omega \cup \omega')$ satisfies the following K\"unneth principle:
\[
	\CSI(Y \# Y', \omega \cup \omega') \cong \CSI(Y, \omega) \otimes \CSI(Y', \omega').
\]
\end{thm}

\begin{proof}
Given Heegaard diagrams $\calH = (\Sigma_g, \bfalpha, \bfbeta, z)$ for $Y$ and $\calH' = (\Sigma_{g^\prime}, \bfalpha', \bfbeta', z')$ for $Y'$, we get a Heegaard diagram $\calH \# \calH' = (\Sigma_{g+g^\prime}, \bfalpha \cup \bfalpha', \bfbeta \cup \bfbeta', z'')$ for $Y \# Y'$, where we perform the connect sum by removing neighborhoods of $z$ and $z'$, and $z'' \in \Sigma_{g + g^\prime}$ is a point in the connect sum region.

Note that in $\calH \# \calH'$, the attaching regions for the $\alpha$- and $\beta$-handles are completely disjoint from the attaching regions for the $\alpha'$- and $\beta'$-handles. Therefore by a sequence of critical point switches, we can construct $Y \# Y'$ by attaching handles in the following order: $\alpha$-handles, $\beta$-handles, $\alpha'$-handles, $\beta'$-handles (if the attaching regions were not disjoint, in general we could only attach $(\alpha \cup \alpha')$-handles and then $(\beta \cup \beta')$-handles). Furthermore, the embedded loops $\omega$ and $\omega'$ can be chosen to lie in the parts of the Cerf decomposition involving the $\alpha$- and $\alpha^\prime$-handle attachments, respectively. Therefore we get a sequence of Lagrangian correspondences
\[
	\scrR_{0,3} \xrightarrow{~L_{\alpha}^\omega~} \scrR_{g,3} \xrightarrow{~L_\beta~} \scrR_{0,3} \xrightarrow{~L_{\alpha^\prime}^{\omega^\prime}~} \scrR_{g^\prime,3} \xrightarrow{~L_{\beta^\prime}~} \scrR_{0,3},
\]
which by Theorem \ref{thm:omega-cerf-invariance} is geometrically equivalent to the sequence given by the Heegaard diagram $\calH \# \calH'$:
\[
	\scrR_{0,3} \xrightarrow{~L_{\alpha\cup\alpha^\prime}^{\omega \cup \omega^\prime}~} \scrR_{g+g^\prime,3} \xrightarrow{~L_{\beta\cup\beta^\prime}~} \scrR_{0,3}.
\]
Hence we have an identification of quilted Floer complexes
\[
	\CF(L_{\alpha\cup\alpha^\prime}^{\omega \cup \omega^\prime}, L_{\beta\cup\beta^\prime}) \cong \CF(L_\alpha^\omega, L_\beta, L_{\alpha^\prime}^{\omega^\prime}, L_{\beta^\prime}).
\]
On the other hand, $L_\beta$ and $L_{\alpha^\prime}^{\omega^\prime}$ have embedded geometric composition, and since $\scrR_{0,3} = \{\mathrm{pt}\}$, it is clear that $L_\beta \circ L_{\alpha^\prime}^{\omega^\prime} \cong L_\beta \times L_{\alpha^\prime}^{\omega^\prime} \subset \scrR_{g,3} \times \scrR_{g^\prime,3}$. Therefore
\begin{align*}
	\CF(L_{\alpha\cup\alpha^\prime}^{\omega \cup \omega^\prime}, L_{\beta\cup\beta^\prime}) & \cong \CF(L_\alpha^\omega, L_\beta \times L_{\alpha^\prime}^{\omega^\prime}, L_{\beta^\prime}) \\
	 & = \CF(L_\alpha^\omega \times L_{\beta^\prime}, L_\beta \times L_{\alpha^\prime}^{\omega^\prime}) \\
	 & \cong \CF(L_\alpha^\omega, L_\beta) \otimes \CF( L_{\alpha^\prime}^{\omega^\prime}, L_{\beta^\prime}). \qedhere
\end{align*}
\end{proof}

\subsection{Euler Characteristic and an Absolute $\Z/2$-Grading}

Since it is useful for some later computations, we start by determining the Euler characteristic of symplectic instanton homology. We proceed in two steps: first work with a trivial $\SU(2)$-bundle, then use a topological argument to generalize to arbitrary $\SO(3)$-bundles.

\begin{thm}
\label{thm:eulerchar}
For any closed, oriented $3$-manifold $Y$,
\[
	\chi(\SI(Y)) = \begin{cases} \pm |H_1(Y;\Z)|, & \text{if } b_1(Y) = 0, \\ 0, & \text{otherwise.} \end{cases}
\]
\end{thm}

\begin{proof}
In Lagrangian Floer homology, the Euler characteristic (with respect to the $\Z/2$-grading determined by choosing orientations of the Lagrangians) is given by $\chi(\HF(L_0, L_1)) = L_0 \cdot L_1$, \emph{i.e.} Lagrangian Floer homology is a categorification of oriented intersection number. Therefore we really just need to fix a Heegaard splitting $Y = H_\alpha \cup_{\Sigma_g} H_\beta$ and determine the intersection number of $L_\alpha$ and $L_\beta$ in $\scrR_{g,3}$.

Recall the representation space
\[
	R^\ast = \Hom(\pi_1(\Sigma_g \setminus \{\text{pt}\}), \SU(2)) \cong \SU(2)^{2g}
\]
and Theorem \ref{thm:zero-level-nbhd}, which states that there are open sets $\mathscr{U} \subset \scrR_{g,3}$ and $U \subset R^\ast$ and a diffeomorphism $h: U \longrightarrow \mathscr{U}$. The Heegaard splitting $Y = H_\alpha \cup_{\Sigma_g} H_\beta$ gives two Lagrangians
\[
	L_\alpha^\ast = \im(\Hom(\pi_1(H_\alpha), \SU(2)) \longrightarrow R^\ast),
\]
\[
	L_\beta^\ast = \im(\Hom(\pi_1(H_\beta), \SU(2)) \longrightarrow R^\ast),
\]
which are embedded copies of $(S^3)^g$ in $R^\ast$. The diffeomorphism $h$ satisfies
\[
	h(L_\alpha^\ast) = L_\alpha, \quad\quad h(L_\beta^\ast) = L_\beta,
\]
and therefore the intersection numbers $L_\alpha \cdot L_\beta$ and $L_\alpha^\ast \cdot L_\beta^\ast$ are equal. It is well-known that the intersection number $L_\alpha^\ast \cdot L_\beta^\ast$ in $R^\ast$ is given by
\[
	L_\alpha^\ast \cdot L_\beta^\ast = \begin{cases} \pm |H_1(Y;\Z)|, & \text{if } b_1(Y) = 0, \\ 0, & \text{otherwise} \end{cases}
\]
(see \cite[Chapter III]{akbulut-mccarthy}, especially Proposition III.1.1). Since $\chi(\SI(Y)) = L_\alpha \cdot L_\beta$ by the first paragraph of the proof, the desired result follows.
\end{proof}

\begin{thm}
\label{thm:euler-char}
For any closed, oriented $3$-manifold $Y$ and any $\SO(3)$-bundle $\omega$ on $Y$,
\[
	\chi(\SI(Y,\omega)) = \begin{cases} \pm |H_1(Y; \Z)|, & \text{if } b_1(Y) = 0, \\ 0, & \text{otherwise.} \end{cases}
\]
\end{thm}

\begin{proof}
Scaduto's computation of the Euler characteristic of framed instanton homology in \cite[Corollary 1.4]{scaduto} essentially applies word-for-word here; we reproduce his argument. The proof proceeds by considering cases of increasing generality.

\emph{Case 1.} First, suppose that $Y$ is an integral homology $3$-sphere. Then $H_1(Y; \F_2) = 0$, so that we necessarily have $\omega = 0$ and $\chi(\SI(Y)) = 1$ by Theorem \ref{thm:eulerchar}.

\emph{Case 2.} Next, suppose that $Y$ is a rational homology $3$-sphere obtained by integral surgery on an algebraically split link $L = L_1 \cup \cdots \cup L_m$. Hence $Y = S^3_{(p_1,\dots,p_m)}(L)$ for some integers $p_k \in \Z$, $k = 1, \dots, m$ and furthermore $|H_1(Y; \Z)| = |p_1 \cdots p_m|$ since $L$ is algebraically split. Suppose the Euler characteristic formula holds for rational homology spheres $M$ which are obtained by integral surgery on an algebraically split link and satisfy $|H_1(M;\Z)| < |p_1 \cdots p_m|$. As Case 1 establishes the base case for this induction, we may assume $Y$ is not an integral homology $3$-sphere and therefore without loss of generality $p_1 > 1$. Note that
\[
	(S^3_{(p_1 - 1,p_2, \dots, p_m)}(L), S^3_{(p_1,p_2,\dots,p_m)}(L), S^3_{(\infty,p_2,\dots,p_m)}(L))
\]
is a surgery triad (cf. Example \ref{ex:n-triad}), and therefore we get a surgery exact triangle
\[
	\xymatrix{\SI(S^3_{(p_1 - 1,p_2, \dots, p_m)}(L), \omega \cup \omega_L) \ar[rr]^{[0]} & & \SI(S^3_{(p_1,p_2,\dots,p_m)}(L), \omega) \ar[dl]^{[0]} \\
	 & \SI(S^3_{(\infty,p_2,\dots,p_m)}(L), \omega) \ar[ul]^{[1]} & }
\]
where $[\cdot]$ denotes the mod $2$ degree of each map and $\omega \in H_1(Y; \F_2)$ is an arbitrary homology class in $Y$, which induces homology classes in the other two spaces. By exactness and the degree of the maps, we have that
\begin{align*}
	\chi(\SI(Y,\omega)) & = \chi(\SI(S^3_{(p_1 - 1,p_2, \dots, p_m)}(L), \omega \cup \omega_L)) + \chi(\SI(S^3_{(\infty,p_2,\dots,p_m)}(L), \omega)) \\
	& = |(p_1 - 1)p_2 \cdots p_m| + |p_2 \cdots p_m| \\
	& = |p_1 \cdots p_m|,
\end{align*}
where in the second line we used the inductive hypothesis. Case 2 therefore is true by induction.

\emph{Case 3.} Now suppose $Y$ is any rational homology $3$-sphere. By \cite[Corollary 2.5]{ohtsuki}, there exists an algebraically split link $L \subset S^3$ such that there is an integral framing $(p_1, \dots, p_m)$ with
\[
	S^3_{(p_1, \dots, p_m)}(L) \cong Y \# L(n_1, 1) \# \cdots \# L(n_k, 1)
\]
for some positive integers $n_1, \dots, n_k$. Note that both $S^3_{(p_1, \dots, p_m)}(L)$ and $L(n_1, 1) \# \cdots \#\allowbreak L(n_k,1)$ fall under Case 2, so that the K\"unneth principle for connected sums implies that
\begin{align*}
	\chi(\SI(Y,\omega)) & = \frac{\chi(\SI(S^3_{(p_1, \dots, p_m)}(L), \omega \cup \omega_1 \cup \cdots \cup \omega_k))}{\chi(\SI(L(n_1, 1) \# \cdots \# L(n_k, 1), \omega_1 \cup \cdots \cup \omega_k))} \\
	 & = \frac{|H_1(S^3_{(p_1, \dots, p_m)}(L); \Z)|}{|H_1(L(n_1, 1) \# \cdots \# L(n_k,1); \Z)|} \\
	 & = |H_1(Y; \Z)|.
\end{align*}
Therefore the Euler characteristic formula holds for all rational homology $3$-spheres.

\emph{Case $b_1(Y) > 0$.} Now suppose the $3$-manifold $Y$ is such that $b_1(Y) > 0$. Then we wish to show that $\chi(\SI(Y, \omega)) = 0$ for all $\omega \in H_1(Y; \F_2)$. Note that in this case, there is a $3$-manifold $M$ with $b_1(M) = b_1(Y) - 1$ and a framed knot $(K, \lambda)$ in $M$ such that $Y \cong M_\lambda(K)$. Since $(M_{\lambda + \mu}(K), M, Y)$ is a surgery triad, for any $\omega \in H_1(Y; \F_2)$ we have a surgery exact triangle
\[
	\xymatrix{\SI(M_{\lambda + \mu}(K), \omega + \omega_K) \ar[rr]^{[0]} & & \SI(M, \omega) \ar[dl]^{[0]} \\
	 & \SI(Y, \omega) \ar[ul]^{[1]} & }
\]
where again $[\cdot]$ indicates the mod $2$ degree of the map. Then it follows that
\[
	\chi(\SI(Y, \omega)) = \chi(\SI(M_{\lambda + \mu}(K), \omega + \omega_K)) - \chi(\SI(M, \omega)).
\]
Applying induction and using this equation will allow us to make our conclusion. For the base case, if $b_1(Y) = 1$, then $M$ and $M_{\lambda + \mu}(K)$ are rational homology spheres with $|H_1(M; \Z)| = |H_1(M_{\lambda+\mu}(K); \Z)|$, and therefore Case 3 implies that $\chi(\SI(Y, \omega)) = 0$. Now suppose that the desired statement is true for all $3$-manifolds with first Betti number less than $n$. Then if $b_1(Y) = n$, the inductive hypothesis implies that $\chi(\SI(M_{\lambda + \mu}(K), \omega + \omega_K)) = \chi(\SI(M, \omega)) = 0$ and the result follows.
\end{proof}

Recall (Section \ref{sect:gradings}) that $\SI(Y)$ admits a $\Z/2$-grading once we fix orientations of $L_\alpha$ and $L_\beta$. Theorem \ref{thm:eulerchar} allows us to make a specific choice of orientations for the Lagrangians when $Y$ is a rational homology $3$-sphere. Namely, we orient the Lagrangians $L_\alpha$ and $L_\beta$ so that $L_\alpha \cdot L_\beta = +|H_1(Y; \Z)|$.

\subsection{Elementary Examples}

Direct computations of Lagrangian Floer homology are typically impossible. Nevertheless, in favorable circumstances, we have enough information and tools to indirectly work out some examples.

\begin{prop}
$\SI(S^3) \cong \Z$.
\end{prop}

\begin{proof}
By taking the genus zero pointed Heegaard diagram for $S^3$, we see that the relevant flat moduli space is $\scrR_{0,3} \cong \{\text{pt}\}$, and the Lagrangians are $L_\alpha \cong \{\text{pt}\} \cong L_\beta$. Hence $\CSI(S^3)$ can be generated by a single point, so that we necessarily have $\SI(S^3) \cong \Z$.
\end{proof}

Before computing $\SI(S^2 \times S^1)$, we recall the following. The self-Floer cohomology of a (monotone) Lagrangian submanifold $L$ in a (monotone) symplectic manifold $M$ admits the structure of a unital algebra. The product is defined by counting pseudoholomorphic triangles with two prescribed vertices, and the unit is defined as a count of pseudoholomorphic disks with one positive boundary puncture. See Appendix \ref{sect:quilts} for a review of how counts of disks with positive/negative boundary punctures define maps between Floer cohomology groups.

\begin{prop}
\label{prop:S1xS2}
As a $\Z/2$-graded unital algebra, we have that
\[
	\SI(S^2 \times S^1, \omega) = \begin{cases} H^{3-\ast}(S^3), & \text{if $\omega$ is trivial,} \\ 0, & \text{if $\omega$ is nontrivial.} \end{cases}
\]
\end{prop}

\begin{proof}
Consider the genus 1 pointed Heegaard diagram $(\Sigma_1, \alpha, \beta, z)$ for $S^2 \times S^1$ where $\alpha$ and $\beta$ are the \emph{same} meridian. Then since $L_\alpha \cap L_\beta = L_\alpha$, $\CSI(\Sigma_1, \alpha, \beta, z)$ is actually the complex for the self Floer homology of $L_\alpha$,
\[
	\CSI(\Sigma_1, \alpha, \beta, z) = \CF(L_\alpha)
\]
It is a folklore fact that under appropriate conditions, the Lagrangian Floer \emph{co}homology $\HF^\ast(L)$ is isomorphic to the singular cohomology $H^\ast(L)$ as a unital algebra, and hence the Floer homology $\HF_\ast(L)$ is isomorphic to $H^{3-\ast}(L)$ as a unital algebra. This holds true for $\HF(L_\alpha)$, but we need to show it in a roundabout way because our setup does not satisfy the usual hypotheses.

First we explain why $\SI(S^2 \times S^1) \cong \Z \oplus \Z$ (just as an abelian group). The proof is indirect and uses the surgery exact triangle (Theorem \ref{thm:surgery-triangle}); we remark that our proof of the surgery exact triangle does not rely in any way on the current Proposition we are proving. Applying the surgery exact triangle to the standard unknot in $S^3$ gives an exact triangle
\[
	\xymatrix{\SI(S^3) \ar[rr]^{[0]} & & \SI(S^2 \times S^1) \ar[dl]^{[0]} \\ & \SI(S^3) \ar[ul]^{[1]} & }
\]
where $[k]$ indicates that the map shifts the $\Z/2$-grading by $k$. Since $\SI(S^3) \cong \Z$ is supported in a single grading, the lower left map must be the zero map, and we in fact have a short exact sequence
\[
	0 \longrightarrow \Z \longrightarrow \SI(S^2 \times S^1) \longrightarrow \Z \longrightarrow 0.
\]
The only possibility is $\SI(S^2 \times S^1) \cong \Z \oplus \Z$.

Next we turn to the product structure. Buhovsky \cite{buhovsky} shows that for a monotone Lagrangian submanifold $L$ in a closed, monotone symplectic manifold $(M,\omega)$, the Oh spectral sequence $\{E^{p,q}_r, d_r\}$ is multiplicative. This spectral sequence has $E_1$ page given by $E^{p,q}_1 \cong H^{p+q-pN_L}(L)$, abuts to $E_\infty^{p,q}$ which satisfies
\[
	\bigoplus_{q \in \Z} E^{p,q}_\infty \cong \HF^\ast(L) \text{ for any fixed $p$},
\]
and collapses at the $([\tfrac{\dim L + 1}{N_L}] + 1)$-stage. In our present situation, $\dim L_\alpha = 3$ and $N_{L_\alpha} = 2$, so we cannot be sure that the spectral sequence has collapsed until the $E_3$-page. However, the fact that $\HF(L_\alpha) = \SI(S^2 \times S^1) = \Z \oplus \Z$ and $H^\ast(L_\alpha) = \Z \oplus \Z$ implies that $E^{p,q}_1 \cong E^{p,q}_\infty$ and the multiplicative structures agree, so that $\SI(S^2 \times S^1) \cong H^{3-\ast}(S^3)$ as a unital algebra.

For $\omega$ nontrivial, we proceed as follows. The nonzero element $\omega \in H_1(S^2 \times S^1; \F_2) \cong \F_2$ can be represented by $\{\text{pt}\} \times S^1$. Consider the genus $1$ Heegaard diagram $(\Sigma_1, \alpha, \beta, z)$ for $S^2 \times S^1$, where $\alpha = \beta$ is the meridian of the torus. In this Heegaard diagram, $\omega$ can be realized as the standard longitude of the torus. Pushing this into the $\alpha$-handlebody, we see that
\[
	L_\alpha = \{[-I, A, \bfi, \bfj, -\bfk]\}, \quad\quad L_\beta = \{[I, A, \bfi, \bfj, -\bfk]\}.
\]
Then $L_\alpha \cap L_\beta = \varnothing$, so that we necessarily have $\SI(S^2 \times S^1, \omega) = 0$.
\end{proof}

\begin{prop}
\label{prop:Lpq}
Let $p$ and $q$ be relatively prime positive integers and $\omega$ be any $\SO(3)$-bundle on $L(p,q)$. Then $\SI(L(p,q), \omega) \cong \Z^p$.
\end{prop}

\begin{proof}
Consider the genus 1 pointed Heegaard diagram $(\Sigma_1, \alpha, \beta, z)$ for $L(p,q)$, where $\alpha$ is the meridian $m$, and $\beta$ represents the curve $qm + p\ell$ in $\pi_1(\Sigma_1) \cong \Z\langle m \rangle \oplus \Z\langle \ell \rangle$. We see that
\[
	L_\alpha = \{[A, B, \bfi, \bfj, -\bfk] \in \scrR_{1,3} : A = I\}, \quad\quad L_\beta = \{[A, B, \bfi, \bfj, -\bfk] \in \scrR_{1,3} : A^q B^p = I\}.
\]
$L_\alpha \cap L_\beta$ therefore corresponds to $p^\text{th}$ roots of $I$ in $\SU(2)$. If $p$ is odd, this set consists of the trivial representation $[I,I, \bfi, \bfj, -\bfk]$ and $\tfrac{1}{2}(p - 1)$ copies of $S^2$, each of which corresponds to the conjugacy class of $e^{2\pi \bfi k/p}$  for some $k \in \{1, \dots, \tfrac{1}{2}(p-1)\}$, where we identify $\SU(2)$ with the unit quaternions. Similarly, if $p$ is even, $L_\alpha \cap L_\beta$ consists of the trivial representation, the representation $[I, -I, \bfi, \bfj, -\bfk]$, and $\tfrac{1}{2}(p - 2)$ copies of $S^2$, coming from the conjugacy classes of $e^{2\pi \bfi k/p}$ for $k \in \{1, \dots, \tfrac{1}{2}(p-2)\}$.

For any $p$, the intersection is seen to be clean in the sense that $T(L_\alpha \cap L_\beta) = TL_\alpha|_{L_\alpha \cap L_\beta} \cap TL_\beta|_{L_\alpha \cap L_\beta}$, so a result of Po\'zniak\footnote{Technically, Po\'zniak's result only works for connected intersections and $\Z/2$-coefficients, but recent work of Schm\"askche \cite{cleanIntersection} establishes the spectral sequence for disconnected clean intersections and $\Z$-coefficients.} \cite{Pozniak} (see also \cite[Section 2]{seidel-knotted}) posits the existence of a Morse-Bott spectral sequence with $E^1 \cong H_\ast(L_\alpha \cap L_\beta)$ abutting to $\HF(L_\alpha, L_\beta) \cong \SI(L(p,q))$. From the previous paragraph, $H_\ast(L_\alpha \cap L_\beta) \cong \Z^p$ (regardless of whether $p$ is odd or even). Recall from Theorem \ref{thm:eulerchar} that $\chi(\SI(L(p,q))) = |H_1(L(p,q))| = p$. The spectral sequence therefore must collapse at the first page, so that $\SI(L(p,q)) \cong \Z^p$.

For $\omega \neq 0$, the computation proceeds in the same way: intersection points of $L_\alpha^\omega$ and $L_\beta$ correspond to $p^\text{th}$ roots of $\epsilon I$ in $\SU(2)$, where $\epsilon = \pm 1$ is some sign depending on $\omega$. Depending on the parity of $p$ and the choice of sign $\epsilon$, we have one of the following descriptions of $L_\alpha^\omega \cap L_\beta$:
\begin{itemize}
	\item If $p$ is even and $\epsilon = +1$, $L_\alpha^\omega \cap L_\beta$ consists of an isolated point and $\tfrac{1}{2}(p-2)$ copies of $S^2$.
	\item If $p$ is even and $\epsilon = -1$, $L_\alpha^\omega \cap L_\beta$ consists of $\tfrac{1}{2}p$ copies of $S^2$.
	\item If $p$ is odd, then for both $\epsilon = +1$ and $\epsilon = -1$ $L_\alpha^\omega \cap L_\beta$ consists of an isolated point and $\tfrac{1}{2}(p-1)$ copies of $S^2$.
\end{itemize}
In any case, Po\'zniak's Morse-Bott spectral sequence has $E^1$-page $H_\ast(L_\alpha^\omega \cap L_\beta) \cong \Z^p$ and converges to $\HF(L_\alpha^\omega, L_\beta) \cong \SI(L(p,q),\omega)$. Since $\chi(\SI(L(p,q),\omega)) = p$ by Theorem \ref{thm:euler-char}, the spectral sequence must collapse immediately, so that $\SI(L(p,q),\omega) \cong \Z^p$.
\end{proof}

\subsection{Instanton $L$-Spaces: Definition and Examples}

\label{sect:L-space}

Theorem \ref{thm:euler-char} suggests that for a rational homology sphere $Y$, $\SI(Y,\omega)$ is as ``simple'' as possible (\emph{i.e.} minimal rank) if it is free abelian of rank $|H_1(Y; \Z)|$. Note that by Proposition \ref{prop:Lpq}, lens spaces $L(p,q)$ satisfy this property for any $\omega \in H_1(L(p,q); \F_2)$, since $\SI(L(p,q), \omega) \cong \Z^p$. In analogy with Heegaard Floer theory, we make the following definition.

\begin{defn}
An {\bf instanton $L$-space} is a rational homology $3$-sphere $Y$ such that $\SI(Y, \omega)$ is free abelian of rank $|H_1(Y; \Z)|$ for all $\omega \in H_1(Y; \F_2)$.
\end{defn}

The reason we require the property to hold for all $\omega \in H_1(Y; \F_2)$ is because families of instanton $L$-spaces should be generated from certain surgery triads via the exact triangle, and in general there is an unavoidable twisting along the core of the surgery appearing in one of the terms of the exact triangle.

There is a useful ``two out of three'' principle for surgery triads containing instanton $L$-spaces.

\begin{prop}
Suppose $(Y, Y_0, Y_1)$ is a surgery triad of rational homology $3$-spheres such that
\[
	|H_1(Y;\Z)| = |H_1(Y_0;\Z)| + |H_1(Y_1;\Z)|
\]
(this condition may always be achieved by cyclically permuting the elements of the triad, which gives another triad). If $Y_0$ and $Y_1$ are instanton $L$-spaces, then so is $Y$.
\label{Lspacetriad}
\end{prop}

\begin{proof}
Fix some class $\omega \in H_1(Y; \F_2)$ and let $\omega_0, \omega_1$ denote the corresponding classes in $Y_0$ and $Y_1$. We first claim that the map $\SI(Y_0, \omega_0) \longrightarrow \SI(Y_1, \omega_1)$ in the exact triangle for $(Y, Y_0, Y_1)$ must be the zero map. To see this, assume for the sake of contradiction that the map is nonzero. Then exactness and the rank-nullity theorem imply the strict inequality 
\[
	\rk \SI(Y, \omega + \omega_K) < \rk \SI(Y_0, \omega_0) + \rk \SI(Y_1, \omega_1) = \chi(\SI(Y, \omega + \omega_K)),
\]
which is not possible. Therefore in the situation of the Proposition, we in fact have a short exact sequence
\[
	0 \longrightarrow \SI(Y_1, \omega_1) \longrightarrow \SI(Y, \omega + \omega_K) \longrightarrow \SI(Y_0, \omega_0) \longrightarrow 0,
\]
from which it follows that $\SI(Y, \omega + \omega_K)$ is free abelian of rank $|H_1(Y_0; \Z)| + |H_1(Y_1; \Z)| = |H_1(Y; \Z)|$. By varying $\omega \in H_1(Y; \F_2)$, we conclude that $Y$ is an instanton $L$-space.
\end{proof}

Using Proposition \ref{Lspacetriad}, we can generate many families of examples:

\begin{exmp}
(Large Surgeries) Suppose $K$ is a knot in $S^3$ and $n > 0$ is some integer such that $S^3_n(K)$ is an instanton $L$-space. Then for any integer $m > n$, $S^3_{m}(K)$ is also an instanton $L$-space, as can be seen by repeatedly applying the surgery exact triangles for surgery triads of the form $(S^3, S^3_{m-1}(K), S^3_m(K))$.

A concrete example of this is given by taking $K$ to be the torus knot $T_{p,q}$ and $n = pq - 1$. Then $S^3_{pq-1}(T_{p,q}) \cong L(p,q)$, and hence $S^3_s(T_{p,q})$ is an instanton $L$-space for all $m \geq pq - 1$.

Another interesting example is given by the pretzel knot $P(-2,3,7)$. Fintushel and Stern noted that $+18$-surgery on this knot gives the lens space $L(18,5)$. It is also known that $S^3_{19}(P(-2,3,7)) \cong L(19,7)$ and $S^3_n(P(-2,3,7))$ is hyperbolic for all $n > 19$. In any case, we see that $S^3_n(P(-2,3,7))$ is an instanton $L$-space for all $n \geq 18$, and this gives an infinite family of hyperbolic examples.
\end{exmp}

\begin{exmp}
(Plumbing Graphs) By a {\bf weighted graph} we mean a pair $(G,m)$ where $G$ is a graph (possibly with multiple connected components) and $m$ (which stands for ``multiplicity'') is an integer-valued function on the vertices of $G$. We can form a $4$-manifold with boundary, $W(G,m)$, by associating to each vertex $v$ of $G$ the disk bundle over $S^2$ with Euler number $m(v)$ and plumbing together bundles whose vertices are connected by an edge. Let $Y(G,m)$ denote the boundary of $W(G,m)$. For a vertex $v$, let $d(v)$ (the {\bf degree} of $v$) denote the number of edges of $G$ containing $v$.

We claim that if $(G,m)$ satisfies
\begin{itemize}
	\item[(1)] $G$ is a disjoint union of trees;
	\item[(2)] $d(v) \leq m(v)$ with the inequality being strict for at least one vertex $v$;
\end{itemize}
then $Y(G,m)$ is an instanton $L$-space. To see this, we use induction on the number of vertices and subinduction on the multiplicity of the new vertex added in the inductive step. First, if $G = \{v\}$ with $m(v) \geq 1$ then the hypotheses of our claim are satisfied, and $Y(G,m)$ is the lens space $L(m(v),1)$, hence an instanton $L$-space. Now suppose the claim holds for all disjoint unions of trees with at most $n$ vertices, and suppose $(G,m)$ is a disjoint union of trees with $n + 1$ vertices. Choose a leaf $v^\ast$ of $G$, \emph{i.e.} a vertex with $d(v^\ast) = 1$; we now subinduct on $m(v^\ast)$. If $m(v^\ast) = 1$, then $Y(G,m) = Y(G',m')$ where $G' = G \setminus \{v^\ast\}$ and $m'$ agrees with $m$ except on the vertex $v'$ that is connected to $v^\ast$ in $G$; we set $m'(v') = m(v') - 1$. $(G',m')$ satisfies our hypotheses and has $n$ vertices, so by the inductive hypothesis $Y(G,m) = Y(G',m')$ is an instanton $L$-space.

Now, for the subinductive hypothesis, suppose the claim holds for all $(G,m)$ with up to $n+1$ vertices where all leaves of $G$ have multiplicity at most $k$. For the subinductive step, let $(G,m)$ be a weighted graph with $n+1$ vertices, and suppose $v^\ast$ is a leaf of $G$ with $m(v^\ast) = k+1$ and all other leaves have multiplicity at most $k$. We can form two weighted graphs from $(G,m)$: $(G_1, m_1)$, with $G_1 = G \setminus \{v^\ast\}$ and $m_1(v)$ agreeing with $m(v)$ for $v \neq v^\ast$; and $(G_2, m_2)$, with $G_2 = G$ and $m_2(v^\ast) = k$, but with $m_2$ otherwise agreeing with $m$. It is clear that $(Y(G,m), Y(G_1,m_1), Y(G_2,m_2))$ is a surgery triad and that
\[
	|H_1(Y(G,m);\Z)| = |H_1(Y(G_1,m_1); \Z)| + |H_1(Y(G_2,m_2);\Z)|.
\]
Since $Y(G_1,m_1)$ and $Y(G_2,m_2)$ are instanton $L$-spaces by the subinductive hypothesis, Proposition \ref{Lspacetriad} implies that $Y(G,m)$ is as well.

\end{exmp}

\begin{exmp}
\label{ex:QAlinks}
(Branched Double Covers of Quasi- Alternating Links) The set of {\bf quasi-alternating links} $\mathcal{Q}$ is the smallest set of links satisfying the following properties:
\begin{itemize}
	\item[(1)] The unknot is in $\mathcal{Q}$.
	\item[(2)] If $L$ is any link admitting \emph{some} diagram with a crossing such that the associated $0$- and $1$-resolutions $L_0$ and $L_1$ (see Figure \ref{fig:skein}) satisfy
	\begin{itemize}
		\item $L_0$ and $L_1$ are in $\mathcal{Q}$,
		\item $\det(L_0), \det(L_1) \neq 0$,
		\item $\det(L) = \det(L_0) + \det(L_1)$,
	\end{itemize}
\end{itemize}
then $L$ is also in $\mathcal{Q}$. Quasi-alternating links were introduced and studied by Ozsv\'ath and Szab\'o in the context of Heegaard Floer theory \cite{oz-sz-branched}. They showed that all alternating knots are quasi-alternating, but gave an example of a quasi-alternating knot that is not alternating.

To see that branched double covers of quasi-alternating links are instanton $L$-spaces, we proceed by induction on the crossing number of the link. The base case certainly holds, as $\Sigma(U) = S^3$ is an instanton $L$-space. Now suppose all branched double covers of quasi-alternating links with crossing number at most $n$ are instanton $L$-spaces, and suppose $L$ is a quasi-alternating link with crossing number $n+1$. Choose a diagram for $L$ with a distinguished crossing as in item (2) above, so that in particular the $0$- and $1$-resolutions are quasi-alternating and satisfy $\det(L) = \det(L_0) + \det(L_1)$. Since $L_0$ and $L_1$ are quasi-alternating with crossing number at most $n$, the inductive hypothesis implies that $\Sigma(L_0)$ and $\Sigma(L_1)$ are instanton $L$-spaces. Then Proposition \ref{Lspacetriad} implies that $\Sigma(L)$ is also an instanton $L$-space. Therefore the branched double cover of any quasi-alternating link is an instanton $L$-space.
\end{exmp}

\begin{exmp}
(Connected Sums) If $Y$ and $Y'$ are instanton $L$-spaces, then so is their connected sum $Y \# Y'$. This is because $|H_1(Y \# Y';Z)| = |H_1(Y;\Z)| \cdot |H_1(Y';\Z)|$ and by the K\"unneth principle for symplectic instanton homology of connected sums, $\SI(Y \# Y') \cong \SI(Y) \otimes \SI(Y') \cong \Z^{|H_1(Y;\Z)| \cdot |H_1(Y';\Z)|}$.
\end{exmp}

%% file: heegaard-diagrams.tex
\section{Heegaard Diagrams for $3$-Manifolds}

\label{sect:heegaard-diagrams}

In much of this paper, we make use of Heegaard diagrams of $3$-manifolds. By now, Heegaard diagrams and their basic theory are well-known due to the success of Ozsv\'ath-Szab\'o's Heegaard Floer homology \cite{oz-sz}, but for convenience we quickly review them in this Appendix.

\subsection{Heegaard Splittings and Diagrams}

A {\bf genus $g$ handlebody} is a $3$-manifold homeomorphic to regular neighborhood in $\R^3$ of a wedge sum of $g$ circles. One may imagine this as a ``filling'' of a standardly embedded genus $g$ surface $\Sigma_g$ in $\R^3$.

Let $Y$ be a closed, oriented, smooth $3$-manifold. A (genus $g$) {\bf Heegaard splitting} of $Y$ is a decomposition of $Y$ as the union of two (genus $g$) handlebodies. More precisely, it consists of an embedding $i: \Sigma_g \longrightarrow Y$ such that $Y \setminus i(\Sigma_g)$ is a disjoint union of two genus $g$ handlebodies. The embedded surface $i(\Sigma_g) \subset Y$ is called a {\bf Heegaard surface}.

\begin{prop}
Any closed, oriented, smooth $3$-manifold $Y$ admits a Heegaard splitting.
\end{prop}

\begin{proof}
It is well-known that such a $Y$ admits a triangulation, \emph{i.e.} there exists a simplicial complex $Y_\Delta$ and a homeomorphism $h: Y_\Delta \longrightarrow Y$. Let $H_0$ denote a regular neighborhood in $Y$ of the image of the $1$-skeleton of $Y_\Delta$ under $h$. Then $H_0$ is a genus $g$ handlebody (for \emph{some} $g$). By considering the triangulation dual to $h: Y_\Delta \longrightarrow Y$, it is easy to see that $H_1 = \overline{Y \setminus H_0}$ is also a genus $g$ handlebody (for the \emph{same} $g$ as $H_0$). Therefore $Y$ admits a Heegaard splitting.
\end{proof}

A Heegaard splitting is an ``internal'' way of decomposing a $3$-manifold $Y$ into two handlebodies, \emph{i.e.} it uses an embedded surface \emph{inside} $Y$. For purposes of visualization, it is more convenient to adopt an ``external'' point of view for Heegaard splittings, where we start with an abstract surface $\Sigma_g$ and obtain a closed, oriented, smooth $3$-manifold in which $\Sigma_g$ can be considered a Heegaard surface by describing how to glue two handlebodies along their common boundary $\Sigma_g$.

In order to achieve this external depiction, we need to understand isotopy classes of diffeomorphisms of $\Sigma_g$. By the Alexander trick, an automorphism of $\Sigma_g$ is determined up to isotopy by the isotopy classes of the images of the $g$ ``standard $\alpha$-curves'' depicted in Figure \ref{fig:standard-alpha}. Hence we can describe a $3$-manifold if we have an abstract genus $g$ surface $\Sigma_g$ along with two $g$-tuples of simple closed curves in $\Sigma_g$ (subject to certain conditions to ensure each set is the image of the standard $\alpha$-curves under a single automorphism of $\Sigma_g$).

\begin{figure}[h]
	\centering
	\includegraphics{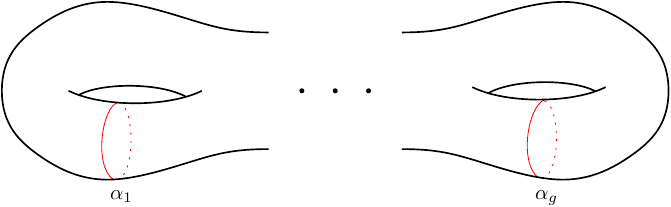}
	\caption{The ``standard $\alpha$-curves'' on $\Sigma_g$.}
	\label{fig:standard-alpha}
\end{figure}

The previous paragraph inspires the following definition. A (genus $g$) {\bf Heegaard diagram} is a collection $\calH = (\Sigma_g, \bfalpha, \bfbeta)$ where $\bfalpha = (\alpha_1, \dots, \alpha_g)$ and $\bfbeta = (\beta_1, \dots, \beta_g)$ are $g$-tuples of simple closed curves in $\Sigma_g$ such that the $\alpha_k$ (repectively the $\beta_k$) are disjoint and homologically independent (\emph{i.e.} their linear span in $H_1(\Sigma_g; \Z) \cong \Z^{2g}$ has rank $g$).

Given a Heegaard diagram $\calH = (\Sigma_g, \bfalpha, \bfbeta)$, we may construct a closed, oriented, smooth $3$-manifold $Y$ equipped with a genus $g$ Heegaard splitting as follows. Form genus $g$ handlebody $H_\alpha$ (the {\bf $\alpha$-handlebody}) by attaching $2$-handles to $\Sigma_g \times [0,1]$ along the curves $\alpha_k \times \{1\} \subset \Sigma_g \times \{1\}$ and capping off the remaining $S^2$ boundary component with a $3$-ball. The {\bf $\beta$-handlebody} $H_\beta$ is formed in the same way using the curves $\beta_k \times \{1\} \subset \Sigma_g \times \{1\}$. We then have a closed $3$-manifold $Y$ obtained by identifying the boundaries of $H_\alpha$ and $H_\beta$ (note that $H_\alpha$ and $H_\beta$ have the \emph{same} abstract surface $\Sigma_g$ as their boundary; the handlebodies may therefore be glued together using the identity map). By construction, $Y$ has the particular Heegaard splitting $H_\alpha \cup_{\Sigma_g} H_\beta$.

The two most basic and important examples of Heegaard diagrams are the genus $1$ diagrams for $S^3$ and $S^2 \times S^1$ pictured in Figure \ref{fig:heegaard-diagrams}.

\begin{figure}[h]
	\centering
	\includegraphics[scale=1.15]{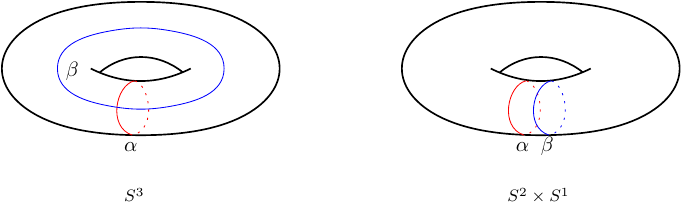}
	\caption{Standard Heegaard diagrams for $S^3$ and $S^2 \times S^1$.}
	\label{fig:heegaard-diagrams}
\end{figure}

If one has two Heegaard diagrams $\calH = (\Sigma_g, \bfalpha, \bfbeta)$ and $\calH' = (\Sigma_{g^\prime}, \bfalpha', \bfbeta')$ representing $3$-manifolds $Y$ and $Y'$, then the ``connected sum'' diagram $\calH \# \calH' = (\Sigma_g \# \Sigma_{g^\prime}, \bfalpha \cup \bfalpha', \bfbeta \cup \bfbeta')$ represents the $3$-manifold $Y \# Y'$.

\subsection{Heegaard Moves}

We would like to use Heegaard diagrams as an equivalent way of thinking of $3$-manifolds (up to diffeomorphism). A moment's thought reveals that this cannot be done without introducing some concept of equivalence of Heegaard diagrams, as any $3$-manifold is described uncountably many different Heegaard diagrams (for example, fix one diagram $(\Sigma_g, \bfalpha, \bfbeta)$ for your $3$-manifold and change the curve $\alpha_1$ by an isotopy). Because of this, we introduce the following three {\bf Heegaard moves}, which change the Heegaard diagram $\calH = (\Sigma_g, \bfalpha, \bfbeta)$ but not the associated $3$-manifold $Y$ (up to diffeomorphism).

\textit{Isotopies.} Replace $\calH$ with $\calH' = (\Sigma_g, \bfalpha', \bfbeta')$, where each $\alpha_k^\prime$ (respectively $\beta_k^\prime$) differs from $\alpha_k$ (respectively $\beta_k$) by an isotopy.

\textit{Handleslides.} Choose any two distinct $\alpha$-curves $\alpha_j$ and $\alpha_k$. Any curve $\alpha_j^\prime$ in $\Sigma_g$ such that together $\alpha_j$, $\alpha_k$, and $\alpha_j^\prime$ bound a pair of pants in $\Sigma_g$ is called a {\bf handleslide of $\alpha_j$ over $\alpha_k$}. The curve $\alpha_j^\prime$ is uniquely determined up to isotopy. If $\calH'$ denotes the Heegaard diagram which identical to $\calH$ except that $\alpha_j$ is replaced with $\alpha_j^\prime$, we say that $\calH'$ is obtained from $\calH$ by a handleslide (of $\alpha_j$ over $\alpha_k$). Handleslides may also be performed amongst the $\beta$-curves. 

\textit{Stabilization.} Replace $\calH$ with $\calH \# \calH_0$, where $\calH_0$ is the standard genus $1$ Heegaard diagram for $S^3$ as in Figure \ref{fig:heegaard-diagrams}. We can also perform the reverse operation, \emph{destabilization}, by getting rid of any $\calH_0$ connect summand of $\calH$; only the $\alpha$- and $\beta$-curve from $\calH_0$ can intersect this summand if we wish to remove it.

The significance of the above three Heegaard moves is that they provide the required equivalence condition on Heegaard diagrams in order to identify them with diffeomorphism classes of $3$-manifolds.

\begin{prop}
Any two Heegaard diagrams $\calH$ and $\calH'$ for the same $3$-manifold $Y$ are related by a finite sequence of isotopies, handleslides, and (de)stabilizations.
\label{prop:heegaardmoves}
\end{prop}

\subsection{Pointed Heegaard Diagrams}

We actually need to include a basepoint when using Heegaard diagrams in symplectic instanton homology, so we explain the minor modifications to the definitions needed to accomodate this.

A {\bf pointed Heegaard diagram} $\calH = (\Sigma_g, \bfalpha, \bfbeta, z)$ is a Heegaard diagram with the additional data of a basepoint $z \in \Sigma_g$ that is disjoint from both the $\alpha$- and $\beta$-curves. A pointed Heegaard diagram determined a closed, oriented, smooth $3$-manifold $Y$ in the same way that an unpointed Heegaard diagram does, and this $3$-manifold $Y$ comes with a distinguished basepoint $z$ on its equipped Heegaard surface.

Given a pointed Heegaard diagram $\calH = (\Sigma_g, \bfalpha, \bfbeta, z)$, the three {\bf pointed Heegaard moves} are nearly the same as the unpointed case, with the following small changes:
\begin{itemize}
	\item Isotopies should not pass through the basepoint.
	\item The pair of pants region defining a handleslide should not contain the basepoint.
	\item The connected sum region for stabilization should not contain the basepoint.
\end{itemize}

We then have a pointed analogue of Proposition \ref{prop:heegaardmoves}:

\begin{prop}
Any two pointed Heegaard diagrams $\calH$ and $\calH'$ for the same pointed $3$-manifold $(Y,z)$ are related by a finite sequence of isotopies, handleslides, stabilizations, and destabilizations.
\end{prop}

%% file: quilted-floer.tex
\section{Quilted Floer Homology}

\label{sect:quilts}

The construction of symplectic instanton homology and the proof of its main properties relies on the quilted Floer theory developed by Wehrheim and Woodward \cite{quilts}. We quickly review the relevant definitions and results in this Appendix.

\subsection{Lagrangian Correspondences and the Symplectic Category}

Recall that if $(M_0, \omega_0)$ and $(M_1, \omega_1)$ are symplectic manifolds, then their product $M_0 \times M_1$ is a symplectic manifold when equipped with the symplectic form $(-\omega_0) \boxplus \omega_1$.\footnote{Recall that $\boxplus$ denotes the external Whitney sum: $(-\omega_0) \boxplus \omega_1 = \pi_0^\ast(-\omega_0) + \pi_1^\ast \omega_1$, where $\pi_i: M_0 \times M_1 \longrightarrow M_i$ ($i = 0, 1$) is the obvious projection.} A {\bf Lagrangian correspondence} from $(M_0, \omega_0)$ to $(M_1, \omega_1)$ is a Lagrangian submanifold $L$ of $(M_0 \times M_1, (-\omega_0) \boxplus \omega_1)$. In what follows, we will typically drop the symplectic form from the notation (in general it will be understood by context), and given $M = (M, \omega)$ we will use the notation $M^- = (M, -\omega)$. Therefore a Lagrangian correspondence from $M_0$ to $M_1$ is a Lagrangian submanifold of $M_0^- \times M_1$.

We use the notation $L: M_0 \longrightarrow M_1$ as shorthand for ``$L$ is a Lagrangian correspondence from $(M_0, \omega_0)$ to $(M_1, \omega_1)$.'' This notation is inspired for our desire to have a ``symplectic category'' whose objects are symplectic manifolds and morphisms are Lagrangian correspondences. A naive construction of such a category is not possible, since compositions aren't always defined.

More precisely, given two Lagrangian correspondences $L_{01}: M_0 \longrightarrow M_1$ and $L_{12}: M_1 \longrightarrow M_2$, their {\bf (geometric) composition} $L_{01} \circ L_{12}$ (note the ordering) is the subset
\[
	L_{01} \circ L_{12} = \{(x,y) \in M_0 \times M_2 \mid \exists p \in M_1 \text{ s.t. } (x,p) \in L_{01}, (p,y) \in L_{12}\}.
\]
Certainly $L_{01} \circ L_{12}$ might not even be an immersed submanifold of $M_0^- \times M_2$, and hence the composition of two Lagrangian correspondences is not in general a Lagrangian correspondence. 

There is a useful case where the composition \emph{is} a Lagrangian correspondence: If $L_{01} \times L_{12}$ intersects $M_0^- \times \Delta_{M_1} \times M_2$ transversely (where $\Delta_{M_1}$ is the diagonal in $M_1^- \times M_1$) and the projection $\pi_{02}: L_{01} \times_{M_1} L_{12} \longrightarrow M_0^- \times M_2$ is an embedding, we say that the composition $L_{01} \circ L_{12}$ is (geometrically) {\bf embedded}. Clearly an embedded composition $L_{01} \circ L_{12}$ is a Lagrangian correspondence from $(M_0, \omega_0)$ to $(M_2, \omega_2)$.

Nevertheless, we will still naturally encounter Lagrangian correspondences whose composition is \emph{not} embedded, and therefore we simply formally introduce compositions in order to form our symplectic category. A {\bf generalized Lagrangian correspondence} from $(M, \omega)$ to $(M', \omega')$ is a finite sequence of Lagrangian correspondences
\[
	\underline{L} = (L_{01}, L_{12}, \dots, L_{(k-1)k})
\]
of the form $L_{(j-1)j}: (M_{j-1}, \omega_{j-1}) \longrightarrow (M_j, \omega_j)$, where $\{(M_j, \omega_j)\}_{j = 0}^k$ is a sequence of symplectic manifolds such that $(M_0, \omega_0) = (M, \omega)$ and $(M_k, \omega_k) = (M', \omega')$. In other words, a generalized Lagrangian correspondence is just some finite tuple of composable Lagrangian correspondences such that the first correspondence starts at $(M, \omega)$ and the last correspondence ends at $(M', \omega')$. The generalized Lagrangian correspondence is {\bf cyclic} if $(M, \omega) = (M', \omega')$.

We are now prepared to define our symplectic category. The {\bf symplectic category} $\Symp$ with objects smooth symplectic manifolds and morphism sets $\Hom_{\Symp}(M, M')$ consisting of generalized Lagrangian correspondences from $M$ to $M'$, modulo the equivalence relation
\[
	(\dots, L_{(j-1)j}, L_{j(j+1)}, \dots) \sim (\dots, L_{(j-1)j} \circ L_{j(j+1)}, \dots)
\]
whenever the composition $L_{(j-1)j} \circ L_{j(j+1)}$ is embedded. The composition of two morphisms $[\underline{L}]$ and $[\underline{L}']$ is achieved by simply concatenating the sequences of Lagrangian correspondences defining each. The identity morphism in $\Hom_{\Symp}(M, M)$ is simply the diagonal correspondence $[\Delta_M]$.

$\Symp$ also admits duals. The {\bf dual} of a morphism
\[
	\underline{L} = (L_{01}, \dots, L_{(k-1)k}) \in \Hom_{\Symp}(M_0, M_k)
\]
is
\[
	\underline{L}^T = (L_{(k-1)k}^T, \dots, L_{01}^T) \in \Hom_{\Symp}(M_k, M_0),
\]
where
\[
	L_{(j-1)j}^T = \{(m_j, m_{j-1}) \in M_j^- \times M_{j-1} \mid (m_{j-1}, m_j) \in L_{(j-1)j}\}.
\]

\subsection{Floer Homology of Lagrangian Correspondences}

We now turn to the task of extending the definition of Lagrangian Floer homology to accept generalized Lagrangian correspondences, rather than a pair of Lagrangians, as its input. In this section, we will assume all symplectic manifolds are monotone with fixed monotonicity constant $\tau \geq 2$ (meaning $\omega = \tau c_1$) and all Lagrangian correspondences are simply connected. These are the conditions satisfied in our main application, and they are sufficient to guarantee that Floer homology is well-defined with $\Z$ coefficients.

Recall that if $L_0$ and $L_1$ are two transversely intersecting Lagrangian submanifolds of a symplectic manifold $(M,\omega)$, then the {\bf Floer chain group} is the free abelian group generated by the intersection points of $L_0$ and $L_1$:
\[
	\CF_\ast(L_0, L_1) = \bigoplus_{p \in L_0 \cap L_1} \Z\langle p \rangle.
\]
If we fix a suitable, generic almost complex structure $J$ on $M$, for any two intersection points $p, q \in L_0 \cap L_1$, we may define the {\bf moduli space of pseudoholomorphic strips} from $p$ to $q$,
\[
	\calM(p,q; J) = \left\{ u: \R \times [0,1] \longrightarrow M ~\left|~ \begin{array}{l} \dbar_J u = 0, \\ u(t, 0) \in L_0, u(t, 1) \in L_1 \text{ for all } t \in \R \\ \displaystyle\lim_{t \to -\infty} u(t,s) = p, \lim_{t \to +\infty} u(t,s) = q \text{ for all } s \in [0,1] \end{array} \right\}. \right.
\]
Here $\dbar_J$ is the nonlinear Cauchy-Rieman operator,
\[
	\dbar_J u = \frac{1}{2}(du - J \circ du \circ i),
\]
where $i$ is the standard complex structure on $\R \times [0,1]$, which we consider as the infinite strip $\{a+ bi \mid b \in [0,1]\} \subset \C$. For suitable generic $J$, $\calM(p,q; J)$ is a union of smooth, oriented, finite-dimensional, compact manifolds (of possibly different dimensions). There is a natural free $\R$-action on $\calM(p,q; J)$ coming from translation in the $\R$-direction of the domain. We write
\[
	\overline{\calM}(p,q; J) = \calM(p,q;J)_1/\R,
\]
where the subscript ``$1$'' denotes the $1$-dimensional component. Hence $\overline{\calM}(p,q; J)$ consists of a finite number of signed points.

With the above in place, we can define the {\bf Floer boundary operator}
\[
	\partial p = \sum_{q \in L_0 \cap L_1} \# \overline{\calM}(p,q; J)q.
\]
One may show that $\partial^2 = 0$, and hence the {\bf Floer homology}
\[
	\HF_\ast(L_0, L_1) = H_\ast(\CF_\ast(L_0, L_1), \partial)
\]
is defined. A standard type of argument applies to show that $\HF_\ast(L_0, L_1)$ is independent, up to isomorphism, of the suitable, generic $J$ used to define it. By a similar standard argument, it is also invariant, up to isomorphism, under changing either of the Lagrangians by a Hamiltonian isotopy. If $L_0$ and $L_1$ don't intersect transversely to begin with, we therefore can still define $\HF_\ast(L_0, L_1)$ by applying a small Hamiltonian isotopy of one of the Lagrangians to achieve transversality.

Floer homology of cyclic generalized Lagrangian correspondences can be directly defined in terms of the classical Floer homology described above. Suppose we have a cyclic generalized Lagrangian correspondence
\[
	\underline{L} = M_0 \xrightarrow{~L_{01}~} M_1 \xrightarrow{~L_{12}~} \cdots \xrightarrow{~L_{(k-1)k}~} M_k,
\]
$M_k = M_0$. There are two cases, depending on whether $k$ is even or odd. First assume $k$ is even. Define a symplectic manifold
\[
	\mathbf{M} = M_0^- \times M_1 \times M_2^- \times \cdots \times M_k
\]
and two Lagrangian submanifolds
\[
	\mathbf{L}_0 = L_{01} \times L_{23} \times \cdots \times L_{(k-2)(k-1)},
\]
\[
	\mathbf{L}_1 = (L_{12} \times L_{34} \times \cdots \times L_{(k-1)k})^T.
\]
Then define
\[
	\HF_\ast(\underline{L}) = \HF_\ast(\mathbf{L}_0, \mathbf{L}_1).
\]
In the case that $k$ is odd, replace $\underline{L}$ with the generalized Lagrangian correspondence
\[
	M_0 \xrightarrow{~L_{01}~} M_1 \xrightarrow{~L_{12}~} \cdots \xrightarrow{~L_{(k-1)k}~} M_k \xrightarrow{~\Delta_{M_k}~} M_k
\]
and proceed as in the even case.

The most important feature of quilted Floer homology for our purposes is that is it preserved under embedded geometric composition:

\begin{prop}
Suppose $(L_{01}, L_{12}, \dots, L_{(k-1)k})$ is a cyclic generalized Lagrangian correspondence, and that for some $j$, the composition $L_{(j-1)j} \circ L_{j(j+1)}$ is embedded. Then
\[
	\HF(L_{01}, \dots, L_{(j-1)j}, L_{j(j+1)}, \dots, L_{(k-1)k}) \cong \HF(L_{01}, \dots, L_{(j-1)j} \circ L_{j(j+1)}, \dots L_{(k-1)k}).
\]
\end{prop}

%% file: dehn-twists.tex
\section{Symplectic Picard-Lefschetz Theory}

\label{sect:dehntwists}

The symplectic geometry of generalized Dehn twists factors crucially into our proofs of the surgery exact triangle and the link surgeries spectral sequence. In this Appendix, we outline the basic theory of symplectic Lefschetz fibrations, their monodromy, and their significance in Lagrangian Floer homology. As this material is now quite standard (cf. \cite[Chapter 6]{McDuff-Salamon}, \cite{seidel-fukaya}, \cite{seidel-triangle}, \cite{fiberedtriangle}), we omit proofs.

\subsection{Symplectic Fibrations}

Let $\pi: E \longrightarrow B$ be a smooth fiber bundle. Given $b \in B$, let $F_b = \pi^{-1}(b) \subset E$ denote the fiber over $b$, and write $\iota_b: F_b \longrightarrow E$ for the inclusion map. $\pi: E \longrightarrow B$ is called a {\bf symplectic fibration} if there is a $2$-form $\omega \in \Omega^2(E)$ such that for each $b \in B$, $\omega_b = \iota_b^\ast \omega$ is a symplectic form on the fiber $F_b$.

Note that each fiber of a symplectic fibration is then a symplectic manifold, but it is not necessary for the total space itself to be symplectic. However, when the base $B$ admits a symplectic form $\beta$, W. Thurston showed that there exists $K > 0$ such that $\omega_K = \omega + K\pi^\ast \beta$ is symplectic. Furthermore is it clear that $\iota_b^\ast \omega_K = \iota_b^\ast \omega$, so replacing $\omega$ by $\omega_K$ does not change the symplectic structure on the fibers. In our main application, $B$ will be the $2$-disk with the standard (exact) symplectic structure, in which case we see that monotonicity of the fibers $(F_b, \omega_b)$ implies monotonicity of $(E, \omega_K)$.

A symplectic fibration $\pi: E \longrightarrow B$ has a natural choice of connection (which we think of here as an invariant horizontal distribution) -- namely, the horizontal subspace of $T_b E$ is the annihilator of $\omega \in \Omega^2(E)$ restricted to $T_b E$. Given a smooth path $\gamma: [0,1] \longrightarrow B$, the parallel transport map
\[
	\Hol_\gamma: F_{\gamma(0)} \longrightarrow F_{\gamma(1)}
\]
associated to this connection is a symplectomorphism, \emph{i.e.} $\Hol_\gamma^\ast \omega_{\gamma(1)} = \omega_{\gamma(0)}$. One important consequence of this is that the canonical parallel transport in symplectic fibrations carries Lagrangians to Lagrangians.

\subsection{Symplectic Lefschetz Fibrations and Dehn Twists}

In practice, we want to consider symplectic fibrations with finitely many fibers that are ``singular'' in some well-controlled way. We first recall the holomorphic analogue of such fibrations, the well-known Lefschetz fibrations. If $B$ is a Riemann surface (possibly with boundary), then $(E, \pi)$ is a {\bf Lefschetz fibration} over $B$ if $E$ is a complex manifold and $\pi: E \longrightarrow B$ is a proper holomorphic map whose critical points are \emph{complex} Morse, \emph{i.e.}\ near any critical point of $\pi$ in $E$, there are holomorphic coordinates $(z_1, \dots, z_n)$ such that
\[
	\pi(z_1, \dots, z_n) = \sum_{k = 1}^m z_k^2
\]
for some $0 \leq m \leq n$.

A symplectic Lefschetz fibration is simply a rewriting of the holomorphic definition to make sense in the symplectic category. So, again, $B$ is a Riemann surface (possibly with boundary), and to start we just consider a smooth, proper map $\pi: E \longrightarrow B$. This $\pi$ has finitely many critical points, all of which lie on distinct fibers. Write $E^{\crit}$ for the critical points and $B^{\crit}$ for the critical values. We then say $(E, \pi)$ is a {\bf symplectic Lefschetz fibration} if the following conditions are satisfied:
\begin{itemize}
	\item[(a)] $\pi: E \setminus E^{\crit} \longrightarrow B$ is a symplectic fiber bundle.
	\item[(b)] There are almost complex structures $J_0$ on a neighborhood of $E^{\crit}$ and $j_0$ on a neighborhood of $B^{\crit}$ such that $\pi$ is $(J_0, j_0)$-holomorphic near $E^{\crit}$.
	\item[(c)] The Hessian $D^2 \pi$ defines a nondegenerate complex quadratic form on $T_e E$ for any $e \in E^{\crit}$.
\end{itemize}

Similarly to the holomorphic case, symplectic Lefschetz fibrations can be understood through their vanishing cycles and monodromy. The vanishing cycle of $e \in E^{\crit}$ is defined as follows: take any path $\gamma: [0,1] \longrightarrow B$ with $\gamma(1) = \pi(e) \in B^{\crit}$ and $\gamma([0,1)) \cap B^{\crit} = \varnothing$ (such paths are called {\bf vanishing paths}). The parallel transport $\Hol_\gamma$ (with respect to the canonical symplectic connection) is \emph{a priori} not well-defined since only $\pi: E \setminus E^{\crit} \longrightarrow B \setminus B^{\crit}$ is a symplectic fibration. Nevertheless, by considering the paths $\gamma_{\tau}: [0, \tau] \longrightarrow B$ ($0 \leq \tau < 1$), one may show that
\[
	\Hol_\gamma(x) = \lim_{\tau \to 1} \Hol_{\gamma_\tau}(x)
\]
defines a continuous map $\Hol_\gamma: F_{\gamma(0)} \longrightarrow F_{\gamma(1)}$. Furthermore, $V_\gamma = \Hol_\gamma^{-1}(e)$ is a Lagrangian sphere in $F_{\gamma(0)}$, called the {\bf vanishing cycle} of the path $\gamma$.

The significance of the vanishing cycle $V_\gamma$ is its connection with the monodromy around the singular fiber containing $e \in E^{\crit}$. If $\mu_e$ is a choice of meridian of $\pi(e)$ in $\pi_1(B \setminus \{\pi(e)\})$ (which we choose to miss all points of $B^{\crit}$), then the {\bf monodromy} of the singular fiber $F_{\pi(e)}$ is the symplectomorphism $\Hol_{\mu_e}: F_{\mu_e(0)} \longrightarrow F_{\mu_e(0)}$. It is well-defined up to Hamiltonian isotopy. To explain the connection between the vanishing cycle $V_\gamma$ (where we choose $\gamma$ such that $\gamma(0) = \mu(0)$) and the monodromy $\Hol_{\mu_e}$, we must recall Arnol'd's symplectic generalization of Dehn twists to higher dimensions.

We first define the model Dehn twist of $T^\ast S^n$ along the zero section, $\tau_{S^n}: T^\ast S^n \longrightarrow T^\ast S^n$. The idea is that $\tau_{S^n}$ is a compactly supported symplectomorphism that restricts to the antipodal map on the zero section. To make this precise, fix a smooth function $\zeta_\varepsilon: \R \longrightarrow \R$ such that
\[
	\zeta_\varepsilon(t) = 0 \text{ for } t > \varepsilon, \quad\quad \zeta_\varepsilon(-t) = \zeta_\varepsilon(t) - t.
\]
By thinking of $T^\ast S^n$ as $\{(u, v) \in \R^n \times \R^n : \|u\| = 1, (u, v) = 0\}$, we can define a Hamiltonian
\[
	H_\varepsilon: T^\ast S^n \longrightarrow \R,
\]
\[
	H_\varepsilon(u,v) = \zeta_\varepsilon(\|v\|).
\]
The time $2\pi$ Hamiltonian flow of $H_\varepsilon$ is certainly smooth on the complement of the zero section, and the conditions on $\zeta$ can be used to show that it smoothly extends to the zero section as the antipodal map $(u, 0) \mapsto (-u,0)$ and is compactly supported on $T^\ast S^n$. This Hamiltonian flow $\tau_{S^n}: T^\ast S^n \longrightarrow T^\ast S^n$ is called the {\bf model Dehn twist}. A different choice of $\varepsilon > 0$ changes $\tau_{S^n}$ by a Hamiltonian isotopy.

More generally, given any Lagrangian sphere $V$ in a symplectic manifold $(M,\omega)$, we may define the symplectic Dehn twist $\tau_V: M \longrightarrow M$ as follows. By the Weinstein tubular neighborhood theorem, there is an open neighborhood $U$ of $V$ in $M$ that is symplectomorphic to a neighborhood $U'$ of the zero section in $T^\ast S^n$. Write $\phi: U \longrightarrow U'$ for this symplectomorphism, and define the model Dehn twist $\tau_{S^n}$ using the Hamiltonian $H_\varepsilon$ with $\varepsilon$ small enough that the support of $H_\varepsilon$ is contained within $U'$. Then the {\bf symplectic Dehn twist} along $V$ is
\[
	\tau_V(x) = \begin{cases} (\phi^{-1} \circ \tau_{S^n} \circ \phi)(x), & \text{if } x \in U, \\ x, & \text{otherwise}. \end{cases}
\]
Since we chose $\varepsilon$ small, $\tau_V$ is a well-defined symplectomorphism of $(M,\omega)$ whose support is contained within the Weinstein neighborhood $U$ of $V$. Different choices of $U$ given Hamiltonian isotopic $\tau_V$.

For any $e \in E^{\crit}$, choose a meridian $\mu_e$ of $\pi(e)$ in $\pi_1(B \setminus B^{\crit})$ and let $\gamma$ be a vanishing path for $e$ with $\gamma(0) = \mu_e(0)$. Then one may show that up to Hamiltonian isotopy, the monodromy $\Hol_{\mu_e}$ is equal to the symplectic Dehn twist $\tau_{V_\gamma}$ along the vanishing cycle $V_\gamma$.

The main reason symplectic Dehn twists are important in this work is that the action of a classical Dehn twist on the traceless character variety $\scrR_{1,3}$ (via the action of the mapping class group on $\scrR_{1,3}$) can be identified with a symplectic Dehn twist -- a fact which is crucial for proving the surgery exact triangle.

\begin{thm}
\textup{(Corollary of \cite[Theorem 3.8(b)]{fiberedtriangle})} If $\gamma$ is a nonseparating simple closed curve in $\Sigma_{1,3}$, then the Dehn twist $\tau_\gamma$ induces a symplectic Dehn twist $\tau_V$ in $\scrR_{1,3}$ along the Lagrangian sphere
\[
	V = \{[\rho] \in \scrR_{1,3} \mid \rho(\gamma) = -I\}.
\]
\end{thm}

\begin{rem}
Dehn twists on non separating curves in $\Sigma_{g,3}$, $g > 1$, induce \emph{fibered} Dehn twists on $\scrR_{g,3}$. In order to avoid defining the more technical fibered Dehn twists, we have arranged our proof of the surgery exact triangle to only use symplectic Dehn twists in $\scrR_{1,3}$.
\end{rem}

\subsection{Relative Invariants of Lefschetz Fibrations with Strip-Like Ends}

The final topic to cover in this Appendix is how Lefschetz fibrations can be used to define homomorphisms between Floer homologies. As was the case in Appendix \ref{sect:quilts}, this requires the use of surfaces with strip-like ends. Hence we will replace the base Riemann surface $B$ with a surface with strip-like ends, and impose some conditions on the Lefschetz fibration near the ends.

Let $S$ be a surface with strip-like ends (see Appendix \ref{sect:quilts} for the definition and notation). Then $\pi: E \longrightarrow S$ is a {\bf Lefschetz fibration with strip-like ends} if it is a Lefschetz fibration (whose regular fiber is symplectomorphic to some fixed $(M, \omega_M)$) and for each end $e \in \calE(S)$, we have a trivialization $\varphi_{S,e}: \epsilon_{S,e}^\ast E \longrightarrow \R^\pm \times [0,1] \times M$ such that $\varphi_{S,e}^\ast \omega_E = \pi^\ast_M \omega_M$, where $\pi_M: \R^\pm \times [0,1] \times M$ is the natural projection.

A {\bf Lagrangian boundary condition} for a Lefschetz fibration with strip-like ends $\pi: E \longrightarrow S$ is a sub-bundle $Q \subset E|_{\partial S}$ such that
\begin{itemize}
	\item[(a)] For each $z \in \partial S$, $Q_z$ is a Lagrangian submanifold of $E_z$.
	\item[(b)] For each end $e \in \calE(S)$, there are Lagrangian submanifolds $L_{0,e}, L_{1,e} \subset M$ such that with respect to the trivialization $\varphi_{S,e}$ near the end $e$, $\varphi_{S,e}^{-1}(s, j, L_{j,e})= Q_{\epsilon_{S,e}(s,j)}$, $s \gg 0$.
	\item[(c)] For each end $e \in \calE(S)$, the Lagrangians $L_{0,e}, L_{1,e}$ intersect transversely.
\end{itemize}

Let $(M, \omega)$ be a closed, monotone symplectic manifold and $\pi: E \longrightarrow S$ be a Lefschetz fibration with strip-like ends and regular fiber $(M, \omega)$, and suppose that $\mathbf{L} = \{L_e\}_{e \in \calE(S)}$ is a collection of pairwise transverse, simply connected (for simplicity), monotone Lagrangian submanifolds of $M$ arising as the limiting Lagrangians for some Lagrangian boundary condition $Q$. We write $\calI_+(\mathbf{L})$ for the set of tuples of points $\mathbf{x}^+ = \{x_e^+ \in L_{e-1} \cap L_e\}_{e \in \calE_+(S)}$ and $\calI_-(\mathbf{L})$ for the set of tuples of points $\mathbf{x}^- = \{x_e^- \in L_e \cap L_{e-1}\}_{e \in \calE_-(S)}$. We may then define the moduli space $\calM_E(\bfx^-, \bfx^+)$, which consists of finite energy pseudoholomorphic sections $u: S \longrightarrow E$ satisfying the following boundary conditions and asymptotics:
\begin{itemize}
	\item $u(I_e) \subset Q|_{I_e}$ for all $e \in \calE(S)$.
	\item $\displaystyle \lim_{s \to \pm \infty} u(\epsilon_{S,e}(s,t)) = x_e^\pm$ for all $e \in \calE_\pm(S)$.
\end{itemize}
For generic almost complex structures on $(M, \omega)$, $\calM_E(\bfx^-, \bfx^+)$ is a smooth, oriented manifold whose zero-dimensional component $\calM_E(\bfx^-, \bfx^+)_0$ is a finite set of points. Hence the Lefschetz fibration with strip-like ends $\pi: E \longrightarrow S$ determines a {\bf relative invariant} $\Phi_E$ in Floer homology defined on the chain level by
\[
	C\Phi_E: \bigotimes_{e \in \calE_-(S)} \CF(L_e, L_{e-1}) \longrightarrow \bigotimes_{e \in \calE_+(S)} \CF(L_{e-1}, L_e),
\]
\[
	C\Phi_E\left(\bigotimes_{e \in \calE_-(S)} x_e^-\right) = \sum_{\bfx^+ \in \calI_+(\mathbf{L})} \#\calM_E(\bfx^-,\bfx^+)_0 \bigotimes_{e \in \calE_+(S)} x_e^+.
\]